\documentclass{article} 

\usepackage[dvipsnames,svgnames,table]{xcolor}

\usepackage[colorinlistoftodos,textsize=footnotesize,backgroundcolor=yellow!40,bordercolor=red,linecolor=red]{todonotes}

\usepackage{titling}   
\usepackage{titlesec}  

\usepackage{mathtools} 
\usepackage[skins,theorems]{tcolorbox} 
\usepackage{amsmath}   
\usepackage{amsthm}    
\usepackage{amssymb}   

\usepackage[a4paper, margin=3cm]{geometry}

\usepackage{tikz}

\usepackage{lineno}

\mathtoolsset{showonlyrefs}

\newtheorem{theorem}{Theorem} 
[section] 
\newtheorem{lemma}[theorem]{Lemma} 
\newtheorem{proposition}[theorem]{Proposition} 

\theoremstyle{definition} 
\newtheorem{remark}[theorem]{Remark} 

\newcommand{\N}{\mathbb{N}}    
\newcommand{\R}{\mathbb{R}}    
\newcommand{\Z}{\mathbb{Z}}    

\newcommand{\forwhich}[0]{{ \ ; \ }} 

\newcommand{\et}{\eta}

\DeclareMathOperator*{\supp}{supp}

\newcommand\norm[1]{\left\Arrowvert#1\right\Arrowvert} 
\newcommand\abs[1]{\left\vert#1\right\vert} 

\usepackage[colorlinks=true, linkcolor=black, urlcolor=black, citecolor=DarkBlue, filecolor=black]{hyperref} 

\numberwithin{equation}{section}

\pretitle{\begin{center}\sffamily\huge}    
\posttitle{\par\end{center}}               
\preauthor{\begin{center}\large\sffamily}   
\postauthor{\par\end{center}}              
\predate{\begin{center}\large\itshape\sffamily} 
\postdate{\par\end{center}}                

\usepackage{tocloft} 

\title{Improved Morse Index Stability for Sequences of Harmonic Maps from Degenerating Riemann Surfaces}          
\author{Francesca Da Lio, Tristan Rivi\`ere and Dominik Schlagenhauf \thanks{Department of Mathematics, ETH Zentrum,
CH-8092 Z\"urich, Switzerland.}} 
\date{\today}                

\begin{document}

\maketitle 

\begin{abstract}
We study the stability of the extended Morse index, defined as the number of negative and zero eigenvalues of the Jacobi operator, for sequences of harmonic maps on degenerating Riemann surfaces. As the conformal structure approaches the boundary of moduli space, collar collapse creates major analytical challenges. We analyze the second variation of the energy under these degenerations and identify conditions ensuring upper semicontinuity of the extended Morse index. \par
Refining earlier results of the first and second authors in \cite{DGR25}, we obtain sharper control of the spectrum of the Jacobi operator on degenerating domains. A key new aspect is the explicit contribution of geodesics arising as limits of the images of degenerating collars. We show that these neck regions converge to geodesic segments whose Morse index contributes nontrivially to the limiting extended index.
\end{abstract}



\tableofcontents
\vspace{10mm} 


\section{Introduction}

The study of harmonic maps between Riemannian manifolds plays a central role in geometric analysis, as they arise as critical points of the Dirichlet energy 
\begin{equation}
\int_{\Sigma} |\nabla u_k|^2 \, d\mathrm{vol}_{h} < \infty,
\end{equation}
where $(\Sigma,h)$ is an arbitrary smooth closed and oriented $2$-dimensional Riemannian manifold,
$u$ is in the space
\begin{equation}
W^{1,2}(\Sigma, \mathcal{N})
:=
\left\{
u \in W^{1,2}(\Sigma,\mathbb{R}^m)
\; ; \;
u(x)\in \mathcal{N}
\quad \mbox{a.e.}
\right\},
\end{equation}
and $(\mathcal{N},g)$ is an arbitrary closed smooth Riemannian manifold of arbitrary dimension $n$ that we assume without loss of generality thanks to Nash's embedding theorem to be embedded in Euclidean space $\R^m$.

A fundamental aspect of their variational structure is captured by the Morse index, which measures the instability of a harmonic map through the number of directions along which the energy decreases.
In many geometric problems, one is led to consider sequences of harmonic maps defined on Riemann surfaces whose conformal structures degenerate. 
By the Deligne--Mumford compactness theory, such degeneration is characterized by the collapse of simple closed geodesics and the formation of nodal surfaces. In this process, the sequence of harmonic maps typically exhibits a loss of compactness through bubbling phenomena, where energy concentrates and gives rise to harmonic spheres, as well as the formation of neck regions connecting different scales.
While the analytic behavior of harmonic maps under degeneration has been extensively studied, particularly in relation to energy quantization and bubble tree convergence (see e.g.\ \cite{CLW, Jos91, Jos06, LaRi1,Par96,RT18, RTZ, Zhu10}), much less is understood about the stability properties encoded by the Morse index. A natural and fundamental question is therefore:
\begin{center}
\emph{How does the Morse index behave along a sequence of harmonic maps from degenerating Riemann surfaces?}
\end{center}
More precisely, one expects that the index should asymptotically decompose in the limit into contributions coming from the limiting base map on the normalized surface and from the bubbles formed during the degeneration process. However, the presence of long cylindrical neck regions introduces additional difficulties, as they may carry nontrivial unstable modes that are not localized purely on the base or the bubbles.

\par
The lower semicontinuity of the Morse index is a fairly general and robust property. In this setting, the proof follows the approach outlined in the appendix of \cite{DRS} and will be presented in \cite{Schl}.
In contrast, the upper semicontinuity of the Morse index is typically much more delicate, as it requires precise estimates of the sequence of solutions in regions where compactness is lost.
In the joint paper \cite{DGR25} of the first and second authors with Gianocca, a new method has been developed to prove the upper-semicontinuity of the Morse index and the nullity associated to conformally invariant variational problems in $2$ dimensions, which include the case of harmonic maps. This new theory has turned out to be very efficient in several other geometric analysis settings (see the recent results in the context of biharmonic maps \cite{Mi, MiRi3}, of Willmore surfaces \cite{MiRi2}, of Yang--Mills connections \cite{GL, GLR}, of Ginzburg--Landau energies \cite{DG}, of constant mean curvature surfaces \cite{Work} and of Sacks--Uhlenbeck approximations \cite{DRS, Sch26}).
Related index estimates  for harmonic maps from surfaces, based on  L.Simon's three circles theorem approach, can be found in \cite{Yin21} and   in \cite{HiLa25} .

\par
The stability result proved in the paper \cite{DGR25} also includes the case of critical points $u_k$ of $E$ from a sequence of closed hyperbolic Riemann surfaces $(\Sigma,h_k)$ of fixed genus $g\geq 2$ and with Gauss curvature $\kappa_{h_k}=-1$ into a closed manifold $(\mathcal{N},\mathfrak{g})$ of class at least $C^2$, when the sequence of constant Gauss curvature metrics $h_k$ is degenerating in the moduli space of $\Sigma$.
Let us also remark at this point that the proof of this paper also applies also to the simpler case of degenerating flat tori of the form $\mathbb T_k=\R^2/\!\sim_k$, where we use the equivalence relation $(x,y)\sim_k (x,y)+ n(1,0)+ m(a_k,b_k)$ (for some $m,n\in \N$) and here $a_k\in (0,1/2)$ and $b_k \to \infty$.

\par
Before describing this particular case, we would like to recall some well-known facts regarding compactness properties of degenerating Riemann surfaces of fixed genus $g\geq 2$.

\medskip
\noindent
\textbf{Degenerating Riemann Surfaces.}
We recall Deligne--Mumford's description of the loss of compactness of the conformal class for a sequence of Riemann surfaces with a fixed topology (see for instance Proposition 5.1 of \cite{Hum} and Section 4 in \cite{Zhu10} for an extensive explanation of the framework). \par

Let $\left(\Sigma_k, h_k, c_k\right)$  be a sequence of closed Riemann
surfaces of fixed genus $g(\Sigma)>1$.  Then, modulo extraction of a subsequence,
 there exist finitely many pairwise
disjoint simple closed geodesics
\begin{equation}
    \gamma_k^1,\ldots,\gamma_k^J \subset (\Sigma_k,h_k)
\end{equation}
such that
\begin{equation}
    \ell_k^j:=\operatorname{length}_{h_k}(\gamma_k^j)\longrightarrow 0,
    \qquad j=1,\ldots,J.
\end{equation}
Cutting the surface $\Sigma_k$ along these geodesics and passing to the
limit gives a complete finite-area hyperbolic surface $(\Sigma,h)$ with
$2J$ cusps, or punctures,
$
    \{\xi_1^j,\xi_2^j\}_{j=1,\ldots,J}.
$
Here the two punctures $\xi_1^j$ and $\xi_2^j$ correspond to the two
boundary components created by cutting along the degenerating geodesic
$\gamma_k^j$.
By Huber's removable singularity theorem, the punctures of
$(\Sigma,c)$ are removable from the conformal point of view. Hence the
complex structure $c$ extends uniquely across the points
$
    \{\xi_1^j,\xi_2^j\}_{j=1,\ldots,J},
$
and one obtains a closed Riemann surface
$
    (\overline{\Sigma},\bar c)
$
such that
$
    \Sigma
    =
    \overline{\Sigma}
    \setminus
    \{\xi_1^j,\xi_2^j\}_{j=1,\ldots,J}.
$
The associated Deligne--Mumford nodal surface is obtained by identifying,
for each $j=1,\ldots,J$, the two points $\xi_1^j$ and $\xi_2^j$.
Finally, for each $k$ there exists a diffeomorphism
\begin{equation}
    \tau_k:
    \Sigma_k \setminus \bigcup_{j=1}^J \gamma_k^j
    \longrightarrow
    \Sigma
\end{equation}
such that, after identifying the two ends of
$\Sigma_k \setminus \gamma_k^j$ with the two punctures
$\xi_1^j,\xi_2^j$, the following hold:
\begin{equation}
\begin{aligned}
\label{EQ:DM-convergence}
 &\text{\rm (i)}\quad
\tau_k:
\Sigma_k \setminus \bigcup_{j=1}^J \gamma_k^j
\longrightarrow \Sigma
\text{ is a diffeomorphism},\\
&\text{\rm (iii)}\quad
(\tau_k)_* h_k \longrightarrow h
\quad\text{in } C^\infty_{\mathrm{loc}}(\Sigma),\\
&\text{\rm (iv)}\quad
(\tau_k)_* c_k \longrightarrow c
\quad\text{in } C^\infty_{\mathrm{loc}}(\Sigma).
\end{aligned}
\end{equation}

\medskip
\noindent
\textbf{Harmonic Maps.}
Let $u_k: \left(\Sigma_k, h_k\right) \to \mathcal N \subset \R^m$ be a sequence of harmonic maps with uniformly bounded energy
\begin{equation}
\sup_k E(u_k) \coloneqq \int_{\Sigma_k} |\nabla u_k|^2 \ dvol_{h_k} <\infty.
\end{equation}
These also satisfy the harmonic map equation
\begin{equation}
\label{EQ: HArm map Equation for u_k}
-\Delta u_k = \mathbb A_{u_k}(\nabla u_k,\nabla u_k), \qquad \text{ in } \Sigma_k,
\end{equation}
where $\mathbb A_q(\cdot,\cdot)$ denotes the second fundamental form of the embedding $\mathcal N\subset \R^m$.
By performing a   bubble tree analysis (see e.g  \cite{Par96}) we may assume, up to subsequences, that the energy of the sequence may concentrate at finitely many points where asymptotically bubbles form and that, away from this concentration points,  the sequence is converging to a base harmonic map $u_\infty:(\Sigma,h) \to \mathcal N$. It is clear that we cannot apriori exclude 
  the possibility of energy concentration and bubble formation   in the collar regions.
More precisely, the sequence $\left(\tau_k\right)_* u_k$ converges on any compact subset of $(\Sigma,h)$ away from the punctures in the bubble tree sense to the bubble tree
\begin{equation}
(u_\infty, v_\infty^1, \ldots, v_\infty^Q).
\end{equation}
As described  in \cite{Zhu10}, the $\varrho$-thin part neighborhood of the collapsing geodesic $\gamma^j_k$ can be represented as long cylinders $P^\varrho_{k,j}=\left[T_{k,j}^{1, \varrho}, T_{k,j}^{2, \varrho}\right] \times S^1$ called collars, which are asymptotically exhausting the infinite cylinder $\R \times S^1$.

To each $j$ there exist $2M_j \in \N$ sequences $(a^\ell_{k,j})_k$, $(b^\ell_{k,j})_k$ (for $j \in \{1,\ldots, M\}$, $\ell \in \{1,\ldots, M_j\}$)
 with  $T_{k,j}^{1, \varrho} \le a^1_{k,j} \ll b^1_{k,j} \le \dots \le  a^{M_j}_{k,j} \ll b^{M_j}_{k,j} \le T_{k,j}^{2, \varrho}$ which decompose $P^\varrho_{k,j}$ as follows:

There exists a possibly empty collection of non-trivial bubbles $\sigma_\infty^{j\ell}$ (for $j \in \{1,\ldots, M\}$, $\ell \in \{1,\ldots, M_j\}$) of $\mathcal N$ and there exist geodesic segments $\widetilde \gamma^{j\ell}_\infty$ (for $j \in \{1,\ldots, M\}$, $\ell \in \{0,\ldots, M_j\}$) of $\mathcal N$ such that:
\begin{enumerate}
\item For $j \in \{1,\ldots, M\}$, $\ell \in \{1,\ldots, M_j\}$ the map $u_k:[a^\ell_{k,j} , b^\ell_{k,j}] \times S^1 \to \mathcal N$ converges $(k\to\infty)$ in a suitable way to the bubble $\sigma_\infty^{j\ell}$.
\item For $j \in \{1,\ldots, M\}$, $\ell \in \{1,\ldots, M_j-1\}$ the map $u_k:[b^\ell_{k,j} , a^{\ell+1}_{k,j}] \times S^1 \to \mathcal N$ converges $(k\to\infty)$ in a suitable way (after reparametrizing to unit speed) to the geodesic segment $\widetilde \gamma^{j\ell}_\infty$.
\item For $j \in \{1,\ldots, M\}$, the maps $u_k:[T_{k,j}^{1, \varrho} , a^1_{k,j}] \times S^1 \to \mathcal N$ and $u_k:[b^{M_j}_{k,j},T_{k,j}^{2, \varrho}] \times S^1 \to \mathcal N$ converge $(k\to\infty)$ in a suitable way (after reparametrizing to unit speed) to the geodesic segments $\widetilde \gamma^{j,0}_\infty$ and $\widetilde \gamma^{j,M_j}_\infty$ accordingly.
\end{enumerate}
For details on the precise sense of convergence, we refer to \eqref{EQ: JINhd2njkd21JNDs82fn28u4fhn2f} and Proposition \ref{congeo}.
We define the limiting average lengths of the images of the collars by
\begin{equation}\label{lambdaj}
\Lambda^j := \lim_{\eta \to 0} \, \limsup_{k \to +\infty}
\int_{\eta^{-1}\delta_k^j}^{\eta}
\frac{1}{2\pi} \int_0^{2\pi}
\left| \frac{d u_k}{dr} \right| \, d\theta \, dr, \qquad j \in \{1,\ldots, M\}.
\end{equation}
In \eqref{lambdaj}, the coordinates $(r,\theta)$ are taken with respect to suitable conformal charts of the collars
\begin{equation}
P_{k,j}^\varrho \to B_\eta(0) \setminus \overline{B_{\delta_k^j/\eta}(0)},
\end{equation}
where
\begin{equation}
\delta_k^j := \exp(-1/l_k^j)\to 0 \qquad \mbox{as } k\to +\infty.
\end{equation}
(for more details see section \ref{SUBSECTION: Collar Framework}).

\noindent
{\bf Remark.}
It turns out that the convergence to the geodesic segments is strong enough to have that $\Lambda^j< \infty$ if and only if ${\rm length}(\widetilde \gamma^{j,0}_\infty), \dots, {\rm length}(\widetilde \gamma^{j,M_j}_\infty) < \infty$ (see \cite{CLW}, \cite{Zhu10} and also Section \ref{SECTION: Collar Analysis} of this paper)

In the paper \cite{DGR25} the first and the second authors proved for harmonic maps the following result.

\begin{theorem}[Lemma V.1 in \cite{DGR25}]
\label{th-morse-stability-degold}
There exists a constant $\Lambda^\ast > 0$, depending only on $n$ and $\mathcal{N}$, such that if the $\Lambda^j$ are all less than $\Lambda^\ast$, then for all sufficiently large $k$ one has
\begin{equation}
\label{0.5-b}
\mathrm{Ind}^\ast_{{E}}(u_k) \le
\mathrm{Ind}^\ast_{{E}}(u_\infty)
+ \sum_{j=1}^Q \mathrm{Ind}^\ast_{{E}}(v_\infty^j)
+\sum_{j=1}^M\sum_{\ell=1}^{M_j} \mathrm{Ind}^\ast_{E}(\sigma_\infty^{j\ell}),
\end{equation}
where $\mathrm{Ind}^\ast_{{E}}(u_k)=\mathrm{Ind}_{{E}}(u_k)+\mathrm{Null}_{{E}}(u_k)$ is the extended Morse index consisting of the morse index plus nullity.
\end{theorem}

It is a classical result that, below a critical length depending on the
geometry of the target, geodesic segments are stable for fixed-endpoint
variations; equivalently, the associated index form has no negative
directions (see, for instance, \cite{Mil63}). In the degeneration setting
of Theorem~\ref{th-morse-stability-degold}, this implies that a collar
whose limiting geodesic has length strictly smaller than the critical
threshold $\Lambda^\ast$ cannot contribute to the negative part of the
spectrum.

This observation explains the additional assumption
\begin{equation}
    \Lambda^j<\Lambda^\ast
    \qquad \text{for every } j
\end{equation}
in Theorem~\ref{th-morse-stability-degold}. Under this smallness condition,
all possible negative directions for $u_k$ are forced, in the limit, to
concentrate either on the limiting map or on one of the bubbles; no
negative direction can arise from the degenerating collars.

The aim of the present paper is to remove the smallness assumption
$\Lambda^j<\Lambda^\ast$. We prove that the index stability statement
remains valid under the sole hypothesis
\begin{equation}
    \Lambda^j<+\infty ,
\end{equation}
allowing the collars to converge to geodesics of arbitrary finite length,
possibly beyond the classical stability threshold. In this regime the
collars may themselves carry unstable directions, and their contribution
has to be incorporated into the limiting index formula.

The main result is the following.

\begin{theorem}
\label{th-morse-stability-degnew}
Assume that the limiting average lengths $\Lambda^j$ are finite. Then, for
all sufficiently large $k$,
\begin{eqnarray}
\label{0.5-b-new}
\mathrm{Ind}^\ast_{E}(u_k) &\le&
\mathrm{Ind}^\ast_{E}(u_\infty)
+\sum_{j=1}^Q  \mathrm{Ind}^\ast_{E}(v_\infty^j)
+\sum_{j=1}^M\sum_{\ell=1}^{M_j} \mathrm{Ind}^\ast_{E}(\sigma_\infty^{j\ell})
+\sum_{j=1}^M\sum_{\ell=0}^{M_j}  \mathrm{Ind}^\ast_{E}(\widetilde \gamma_\infty^{j\ell}),
\end{eqnarray}
and
\begin{eqnarray}
\label{0.6-b}
\mathrm{Ind}_{E}(u_k)
&\ge&
\mathrm{Ind}_{E}(u_\infty)
+\sum_{j=1}^Q \mathrm{Ind}_{E}(v_\infty^j)
+\sum_{j=1}^M\sum_{\ell=1}^{M_j} \mathrm{Ind}_{E}(\sigma_\infty^{j\ell})
+\sum_{j=1}^M\sum_{\ell=0}^{M_j} \mathrm{Ind}_{E}(\widetilde \gamma_\infty^{j\ell}).
\end{eqnarray}
\hfill $\Box$
\end{theorem}

The distinctive feature of Theorem~\ref{th-morse-stability-degnew} is the
explicit appearance of the geodesic limits
$\widetilde\gamma_\infty^{j\ell}$ associated with the degenerating
collars. Thus the limiting index is no longer described only by the base
map and the bubbles, as in the usual bubble-tree analysis. The collars may
retain genuine negative directions, and these are measured by the Morse
indices of the corresponding limiting geodesics. In this sense, the
geodesic necks form a third variational component of the degeneration,
alongside the limiting map and the harmonic spheres. To the best of our
knowledge, this is the first index-stability result for harmonic maps from
degenerating Riemann surfaces in which this geodesic contribution is
identified and quantified.\medskip
\noindent

The idea of the proof of Theorem \ref{th-morse-stability-degnew} is to apply the improved pointwise estimate (on annuli) of the gradient obtained in Lemma V.1 of \cite{DGR25} to give an independent proof of the following result which was previously established in \cite{CLW}, \cite{Zhu10}:

\begin{theorem}
[\cite{CLW}, \cite{Zhu10}, see Proposition \ref{congeo}]
\label{THEOREM: Conv to geod general}
After reparametrization, the convergence of the maps $u_k$ to the geodesic segments (in the according domains) occurs in the $C^1$-topology on compact subsets of the open cylindrical domains. 
\end{theorem}

The estimates derived in the proof of Theorem \ref{THEOREM: Conv to geod general} will then be used to perform the collar analysis of the Morse index.
The strategy of the proof consists in taking a suitable sequence of non-positive variations of the harmonic maps $u_k$ and prove  that these asymptotically produce non-positive variations of at least one of the limiting objects $u_\infty$, $v^j_\infty$, $\sigma_\infty^{j\ell}$, $\widetilde \gamma_\infty^{j\ell}$.
In this context "suitably" is meant in the following way:
Sylvester's law of inertia allows us the freedom of rescaling variations by a positive weight. This permits to find meaningful limiting variations.
We then proceed to show that the junction regions connecting different scales of convergence do not contribute to the negativity of the second variation. 

\medskip

\begin{remark}
\label{REMARK: Converse without positive Ricci}
The finiteness assumption on the limiting lengths is also related to a
partial converse to Theorem~\ref{th-morse-stability-degnew}. In
\cite{LLW17}, it is shown that for a sequence of harmonic maps
\begin{equation}
    u_k:(\Sigma_k,h_k)\to (\mathcal N,g)
\end{equation}
from degenerating Riemann surfaces, the assumptions
\begin{equation}
    \operatorname{Ric}_{\mathcal N}\geq \lambda>0,\qquad
    \sup_k E(u_k)<+\infty,\qquad
    \sup_k \operatorname{Ind}_{E}(u_k)<+\infty
\end{equation}
imply that
$$
    \Lambda^j<+\infty
    \qquad \text{for every } j.
$$
Thus, under a positive Ricci curvature assumption, bounded energy together
with bounded Morse index prevents the collars from producing geodesic
limits of infinite length. This curvature assumption is not merely
technical: Appendix~\ref{APPENDIX: Flat torus counterexample} shows that,
without such a positive lower bound, bounded index alone does not in
general control the limiting lengths, by constructing a sequence of
harmonic maps from degenerating flat tori such that
$$
    \operatorname{Ind}_E^\ast(u_k)=2
    \qquad \text{for all } k,~~\text{whereas}  ~~ \Lambda^1=+\infty ~~\Box
$$
\end{remark}

\medskip

We recall that the second derivative of $E$ at a harmonic map $u$ is defined on the space of infinitesimal variations of $u$, which are nothing but the sections of the pull-back by $u$ of the tangent bundle to $\mathcal{N}$:
\begin{equation}
\label{0.2-a}
V_u
:=
\Gamma(u^{-1}T\mathcal{N})
:=
\left\{
w\in W^{1,2}(\Sigma, \mathbb{R}^m)
\; ; \;
P_{u(x)}w=w \mbox{ for a.e. } x\in \Sigma
\right\},
\end{equation}
where for all $z\in \mathcal{N}$, $P_z$ is the $m\times m$ matrix corresponding to the orthogonal projection from $\mathbb{R}^m$ onto $T_z\mathcal{N}$.

The second derivative of $E$ at $u$ is a quadratic form on $V_u$ given by
\begin{equation}
\label{0.3}
Q_u(w)=D^2E_u(w)=\int_\Sigma |dw|_h^2\, d\mathrm{vol}_h-\int_\Sigma \left\langle {\mathbb A}_u(du,du)_h , {\mathbb A}_u(w,w)\right\rangle\, d\mathrm{vol}_h,
\end{equation}
where $\langle\cdot,\cdot\rangle$ denotes the scalar product in $\mathbb{R}^m$, ${\mathbb A}_z(X,Y)$ is the second fundamental form of the embedding $\mathcal{N}\hookrightarrow \mathbb{R}^m$ at the point $z\in \mathcal{N}$, taken over a pair of vectors $X,Y\in T_z\mathcal{N}$. 
The nonlinearity ${\mathbb A}_u(du,du)_h$ is equal to $\mathrm{Tr}_h({\mathbb A}_u(du\otimes du))$. Using local coordinates on $\Sigma$ it is given by
\begin{equation}
{\mathbb A}_u(du,du)_h:=\sum_{i,j=1}^2 h^{ij}\,{\mathbb A}_u(\partial_{x_i}u,\partial_{x_j}u).
\end{equation}
In the sphere case when $\mathcal{N}=S^n\subset \mathbb{R}^{n+1}$ the computation gives for instance
\begin{equation}
\label{0.3-a}
Q_u(w)=D^2E_u(w)=\int_\Sigma |dw|^2_h- |du|^2_h\,|w|^2\, d\mathrm{vol}_h.
\end{equation}
We define the \textit{Morse index} of $u$ to be
\begin{equation}
\mathrm{Ind}_{E}(u):=
\mathrm{Ind}(u):=
\sup\left\{
\dim(W)
\; ; \;
W \mbox{ is a vector subspace of }V_u
\mbox{ such that }
Q_u|_{W\setminus\{0\}}<0
\right\},
\end{equation}
and the \textit{nullity}
\begin{equation}
\mathrm{Null}_{E}(u):=\dim\bigl(\mathrm{Ker}(Q_u)\bigr).
\end{equation}

Let
$
    \gamma:[a,b]\to \mathcal N
$
be a geodesic segment parametrized with constant speed. Then $\gamma$ is
a critical point of the energy functional
\begin{equation}
    \widetilde E(\gamma)
    =
    \int_a^b |\dot\gamma(t)|_h^2\,dt
\end{equation}
with respect to variations whose endpoints are fixed.

The admissible infinitesimal variations are $W^{1,2}$-sections of the
pull-back bundle $\gamma^{-1}T\mathcal N$ which vanish at the endpoints.
We denote this space by
\begin{equation}
\label{EQ: geod var 0239ujr233}
    V_\gamma
    :=
    W^{1,2}_0\bigl([a,b],\gamma^{-1}T\mathcal N\bigr).
\end{equation}
Equivalently, after an isometric embedding
$\mathcal N\hookrightarrow \mathbb R^m$, one may write
\begin{equation}
\label{EQ: geod var embedded 0239ujr233}
    V_\gamma
    =
    \left\{
    \mathsf v\in W^{1,2}_0([a,b];\mathbb R^m)
    :
    \mathsf v(t)\in T_{\gamma(t)}\mathcal N
    \ \text{for a.e. } t\in[a,b]
    \right\}.
\end{equation}

\par
The second variation is given by $Q_\gamma:V_\gamma \rightarrow\mathbb{R}$,
\begin{equation}
\begin{aligned}
Q_\gamma(\mathsf{v}) 
= \int_a^b \left(|\dot{\mathsf{v}}|^2 - {\mathbb A}_\gamma(\dot \gamma, \dot \gamma) \cdot {\mathbb A}_\gamma(\mathsf{v},\mathsf{v}) \right)\, dt.
\end{aligned}
\end{equation}
The \underline{Morse index} of $\gamma$ is given by
\begin{equation}
{\operatorname{Ind}}_{\widetilde{E}}(\gamma):=
{\operatorname{Ind}}(\gamma):=\max \left\{\dim(W) ; W \text{ is a vector subspace of } V_\gamma \text{ such that } Q_\gamma|_{W\setminus\{0\}}<0\right\},
\end{equation}
and the \underline{nullity} of $\gamma$ by
\begin{equation}
{\operatorname{Null}}_{\widetilde{E}}(\gamma):= \dim(\ker Q_\gamma).
\end{equation}

\noindent
\begin{remark}[Simplifying assumptions and collar bubbles]
\label{REMARK: Reduction to concentration-free collars}
In order to isolate the main mechanism of the proof, we shall first work
under two simplifying assumptions. First, we assume that there is a single
degenerating collar, namely $J=1$, and we omit the corresponding index
$j$. Second, we assume that the sequence $u_k$ does not develop bubbles
either away from or inside this collar. Thus, after passing to the
Deligne--Mumford limit and using the identifications given by $\tau_k$, we
suppose that
\begin{equation}
\label{EQ: JINhd2njkd21JNDs82fn28u4fhn2f}
    u_k\circ \tau_k^{-1}
    \longrightarrow u_\infty
    \qquad
    \text{strongly in } C^1_{\mathrm{loc}}
    \bigl(
        \overline{\Sigma}\setminus\{\xi_1,\xi_2\}
    \bigr).
\end{equation}
Equivalently, away from the degenerating geodesic $\gamma_k$, the maps
$u_k$ converge locally in $C^1$, after the above identifications, to a
limiting harmonic map $u_\infty$ defined on the punctured limiting surface
\begin{equation}
    \overline{\Sigma}\setminus\{\xi_1,\xi_2\}.
\end{equation}

These assumptions are made only to avoid obscuring the new feature of the
argument with the additional notation required by the full bubble-tree
decomposition. The essential point of the present paper is the contribution
of the limiting geodesic segments, arising from the degenerating collars, to
the Morse index.

In Appendix~\ref{APPENDIX: Reduction to concentration-free collars}, we
explain how the argument has to be modified in the simplest nontrivial case
where a bubble forms at one concentration point inside a collar. The collar
is then decomposed into a bubble region, adjacent geodesic regions, and
concentration-free transition annuli between the different scales. On these
transition annuli, the boundary values of the corresponding weights can be
adjusted so as to glue the geodesic weights to the bubble-tree weight
introduced in \cite{DGR25}.

The appendix is meant to describe the main gluing mechanism rather than to
give the full set of estimates needed for arbitrary bubble configurations.
A complete treatment of the case of several bubbles, both inside and outside
the collars, together with the compatibility of the weights across the
different bubble and geodesic scales, will be carried out in the forthcoming
paper \cite{DMRS}.
\end{remark}
Henceforward, we will assume that the limiting average length of $u_k$ in the collars are finite, i.e. $\Lambda < \infty$.
Then, in the absence of concentration in the collar, the finiteness of $\Lambda$ implies the energy quantization established by different methods in \cite{DGR25,Zhu10}:
\begin{equation}
\label{EQ: oINijnfj1093jn131dS12d}
\lim _{\eta \rightarrow 0} \limsup _{k \rightarrow\infty} \norm{\nabla u_k}_{L^2(A(\eta,\delta_k))}=0.
\end{equation}

We recall that the $L^{2,1}$ quantization of sequences of harmonic maps, established first by Laurain and Rivi\`ere in \cite{LR14}, is crucial to obtain the so-called \textit{$C^0$-no-neck property}, or \emph{necklessness property}, which says roughly speaking that there is no distance between the bubbles $v^j_\infty$, $\sigma_{\infty}^{\ell,j}$ and the main part $u_\infty$, namely, the bubbles $v_\infty^j$, $\sigma_{\infty}^{\ell,j}$ are ``$L^\infty$-glued'' directly to the weak limit $u_\infty$.\footnote{We recall that for a given domain $\mathbb{D}$,
\begin{equation}
L^{2,\infty}(\mathbb{D})=
\left\{
f \text{ measurable} :
\sup_{\lambda\ge 0}\lambda \,
\bigl|\{x\in \mathbb{D}: |f(x)|\ge\lambda\}\bigr|^{1/2}<+\infty
\right\},
\end{equation}
and
\begin{equation}
L^{2,1}(\mathbb{D})=
\left\{
f \text{ measurable} :
2\int_{0}^{+\infty}
\bigl|\{x\in\mathbb{D}: |f(x)|\ge\lambda\}\bigr|^{1/2}
\, d\lambda <+\infty
\right\}.
\end{equation}
}
The energy identity and the necklessness property do not hold in general for harmonic maps from degenerating Riemann surfaces (see e.g. \cite{CLW,Par96,Zhu10}).

In the paper \cite{CLW}, following the ideas in \cite{Zhu10}, the authors re-prove that the necks of the above $u_k$ from long cylinders consist of geodesics in the target manifold, and then give the length formula of these geodesics in the case where there are no bubbles. The proof in \cite{CLW} strongly relies on the analysis of the Hopf coefficient.

\par
We would like to mention that in a forthcoming paper \cite{DMRS}, we investigate the improved stability of the extended Morse index in the spherical case, using somewhat different arguments. The associated conservation laws provide a natural framework for exploring further aspects of the problem, in particular the contribution of limit geodesics to the asymptotic behavior of the Morse index for degenerating Riemann surfaces.
If the collapsing geodesic is homologically trivial then, when the target is a round sphere,  the corresponding geodesic segment of the harmonic maps is expected to be a point (see [6]).
whether this fact extends to the general simply connected target or not is an interesting open problem.

\medskip
The paper is organized as follows.
In Section~\ref{SECTION: Collar Analysis} we develop the analytical
framework needed to study harmonic maps on degenerating collars. Although
the results are closely related to the classical neck analysis for harmonic
maps, our approach follows a different strategy from those usually employed
in the literature. In particular, we give an independent proof of the
convergence result stated in Theorem~\ref{THEOREM: Conv to geod general},
adapted to the collar geometry and to the spectral questions considered
later in the paper. We also prove an $L^{2,1}$-energy quantization result
for the junction regions connecting different scales of convergence. This
implies, in particular, that no neck energy is lost in these transition
regions.

Section~\ref{SECTION: Stability of the Morse Index} is devoted to the
second variation of the energy. We show that the junction necks connecting
different scales do not contribute to the negative part of the second
variation, and we analyze the diagonalization of the Jacobi operator. In
this section we also introduce the appropriate scaling weights and explain
how to construct admissible variations associated with the limiting
geodesic segments.

Finally, in Section~\ref{SECTION: Proof of Main Theorem} we combine these
ingredients to prove the main result, Theorem~\ref{th-morse-stability-degnew}.

\section{Collar Analysis}
\label{SECTION: Collar Analysis}

In this section we carefully study the behavior of the sequence of harmonic maps in the degenerating collar regions.
The goal is to establish convergence results and pointwise bounds based on the work of \cite{DGR25}.
Depending on the context it will be either convenient to represent the collar regions as the isometric copies as long cylinders, as conformal copies in annular domains, or to even brake conformality and reshape the long cylinders to unit length. This allows to see that the harmonic maps are converging to a geodesic segment of the target.

\subsection{Collar Framework}
\label{SUBSECTION: Collar Framework}

As explained in the appendix of \cite{RT18} or in \cite{Zhu10}, the neighborhood (referred to as collar) around the collapsing geodesic can be represented as follows:

\begin{lemma}
There is a neighborhood around $\gamma_k$, a so-called collar, which
is isometric to the cylinder
\begin{equation}
P_k=\left\{(t, \theta): \frac{2 \pi}{l_k} \arctan \left(\sinh \left(\frac{l_k}{2}\right)\right)<t<\frac{2 \pi}{l_k}\left(\pi-\arctan \left(\sinh \left(\frac{l_k}{2}\right)\right)\right), 0 \leq \theta \leq 2 \pi\right\}
\end{equation}
with metric the hyperbolic metric
\begin{equation}
d s^2=\left(\frac{l_k}{2 \pi \sin \frac{l_k t}{2 \pi}}\right)^2\ dt\ d\theta
\end{equation}
here $\gamma_k$ corresponds to $\left\{t=\frac{\pi^2}{l_k}\right\} \subset P_k$.
\end{lemma}

We define the following $\varrho$-subcollars of $P_k$
\begin{equation}
P_k^\varrho:=\left[T_k^{1, \varrho}, T_k^{2, \varrho}\right] \times S^1 \subseteq P_k,
\end{equation}
where
\begin{equation}
T_k^{1, \varrho}=\frac{2 \pi}{l_k} \arcsin \left(\frac{\sinh \left(\frac{l_k}{2}\right)}{\sinh \varrho}\right), \quad T_k^{2, \varrho}=\frac{2 \pi^2}{l_k}-\frac{2 \pi}{l_k} \arcsin \left(\frac{\sinh \left(\frac{l_k}{2}\right)}{\sinh \varrho}\right) .
\end{equation}
Whenever $l_k\le 2\rho$ one can verify that $P_k^\varrho$ is exactly the $\varrho$-thin part of $\left(\Sigma_k, h_k\right)$, namely
\begin{equation}
P_k^\varrho=\left\{z \in \Sigma_k, \operatorname{injrad}\left(z ; h_k\right) \leq \varrho\right\} .
\end{equation}
Notice that whenever $k$ is large one has
\begin{equation}
\frac{2}{l_k}\arctan \left(\sinh \left(\frac{l_k}{2}\right)\right) \approx 1
\end{equation}
as well as
\begin{equation}
\arcsin \left(\frac{\sinh \left(\frac{l_k}{2}\right)}{\sinh \varrho}\right) \approx \frac{l_k}{2 \sinh(\varrho)},
\end{equation}
and whenever $\rho$ is small one has
\begin{equation}
\arcsin \left(\frac{\sinh \left(\frac{l_k}{2}\right)}{\sinh \varrho}\right) \approx \frac{l_k}{2 \sinh(\varrho)} \approx \frac{l_k}{2 \varrho}.
\end{equation}
This means that for large $k$ and small $\rho$ one has 
\begin{equation}
\label{EQ: inufn29u89hjn1f923i}
P_k \approx_{iso} \mathtt{P}_k \coloneqq \left[\pi, \frac{2\pi^2}{l_k}-\pi\right] \times S^1
\end{equation}
and
\begin{equation}
\label{EQ: iuNU813EJIn1ud9nudF3}
P_k^\varrho \approx_{iso} \mathtt{P}_k^\varrho \coloneqq \left[\frac{\pi}{\varrho},\frac{2\pi^2}{l_k}-\frac{\pi}{\varrho}\right] \times S^1
\end{equation}
This means that the the cylinders in \eqref{EQ: inufn29u89hjn1f923i} and \eqref{EQ: iuNU813EJIn1ud9nudF3} are almost isometric and conformal equivalent under the natural affine diffeomorphism
\begin{equation}
\Psi: P_k^\varrho\to \mathtt{P}_k^\varrho.
\end{equation}
\textbf{Assumption:} Due to the above justification we will henceforward assume that the collars $P_k$ and $P_k^\varrho$ (and therefore the thin parts of the surfaces) are precisely given by $\mathtt P_k$ and $\mathtt P_k^\varrho$ with the flat Euclidean metric.
Now define 
\begin{equation}
\eta \coloneqq \eta(\varrho) \coloneqq \exp(-\pi/\varrho), \qquad \text{ and }
\delta_k \coloneqq \exp(-2\pi^2/l_k)
\end{equation}
such that under the change of variables 
\begin{equation}
\Phi: \mathtt P_k \to A(e^{-\pi},\delta_k);
\qquad (t,\theta) \mapsto \exp(-t) (\cos\theta,\sin\theta)
\end{equation}
one has the conformal equivalences 
\begin{equation}
\mathtt{P}_k\sim_{conf} A(e^{-\pi},\delta_k)\quad \text{ and } \quad
\mathtt{P}_k^\varrho \sim_{conf} A(\eta,\delta_k),
\end{equation}
where here we are making use of the notation $A(R,r)=B_{R}\setminus \bar{B}_{r/R}$.
Furthermore, to any function $f: \mathtt P_k^\varrho \to \R$ one has the scaling
\begin{equation}
 \int_{\mathtt{P}_k^\varrho} f \ dx = \int_{A(\et,\delta_k)} f\circ\Phi^{-1} \ \frac{dz}{|z|^2},
\end{equation}
We would also like to remark that the map $\Phi$ is conformal and therefore the energy changes at most by a universal constant.

\textbf{Notation:}
Henceforth, whenever we consider a function $f:\Sigma_k \to \mathbb{R}$ in the collar regions $\mathtt P_k$ and $\mathtt P_k^\varrho$ we simply denote it by $f$. Whenever we work on the annulus $A(\eta,\delta_k)$  we denote $f^{\sharp}\coloneqq f \circ \Phi^{-1}$.

\subsection{Pointwise Control of the Angular Derivative in the Collars}

As before consider the sequence of harmonic maps $u_k^\sharp:A(\eta,\delta_k)\to \mathcal N$ on the annuli with uniformly bounded energy (recall that the Dirchlet energy is conformal invariant) with the flat metric.

We set
\begin{equation}
T_{\eta,k}=\log\left(\frac{ \eta^2}{\delta_k}\right).
\end{equation}
For small $\eta>0$ and large $k\in \N$ the following estimate was shown in Lemma V.1 of \cite{DGR25}
\begin{equation}
\label{EQ: jUINDUQWn2u323fnu9inufnujf3UJLCNS91edq}
\begin{aligned}
|x|^2 |\nabla u_k^\sharp|^2
&\le C \norm{\nabla u_k^\sharp}_{L^2(A(\eta,\delta_k))}
\left[ \left(\frac{|x|}{\eta} \right)^\beta + \left(\frac{\delta_k}{|x|\eta} \right)^\beta \right]  + C\frac{\Lambda^2}{T_{\eta,k}^2} \\
&\le C \norm{\nabla u_k^\sharp}_{L^2(A(\eta,\delta_k))}
\left[ \left(\frac{|x|}{\eta} \right)^\beta + \left(\frac{\delta_k}{|x|\eta} \right)^\beta \right]  + C\frac{\norm{\nabla u_k^\sharp}_{L^{2,1}(A(\eta,\delta_k))}^2}{T_{\eta,k}^2} 
\end{aligned}
\end{equation}
for $x\in A(\eta,\delta_k)$, where we used with the Lorentz-H\"older inequality the estimate
\begin{equation}
\begin{aligned}
\label{EQ: UNui92efcijf2JNdqfdq321}
\Lambda &\le 2 \int_{\eta^{-1} \delta_k}^\eta \int_{S^1} \left|\frac{d u_k^\sharp}{d r}\right| d \theta dr
=2 \int_{A(\eta,\delta_k)} \left|\frac{d u_k^\sharp}{d r}\right| \frac{1}{|x|}\ dx \\
&\le C \norm{\frac{d u_k^\sharp}{d r}}_{L^{2,1}(A(\eta,\delta_k))} \norm{\frac{1}{|x|}}_{L^{2,\infty}(A(\eta,\delta_k))} \\
&\le C \norm{\nabla u_k^\sharp}_{L^{2,1}(A(\eta,\delta_k))}.
\end{aligned}
\end{equation}
(For small $\eta>0$ and large $k\in \N$.)

The aim of this subsection is to show that the term involving $1/T_{\eta,k}^2$ can be removed from the estimate whenever we are only bounding the angular derivative.
We start by proving the following elementary general result on harmonic maps:

\begin{lemma} [Angular $\varepsilon$-regulrity]
\label{LEMMA: Angular epsilon-regulrity}
Let $u:B_r\to \mathcal N$ be a harmonic map.
If $\|\nabla u\|_{L^2(B_r)}<\varepsilon_0$ then 
\begin{equation}
\|\partial_\theta u \|_{L^\infty(B_{r/2})} \le \frac{C}{r} \|\partial_\theta u\|_{L^2(B_r)}
\end{equation}
\end{lemma}

\begin{proof}
Scaling to the map $\widetilde u(y)\coloneqq u(ry)$ the general result can be derived from the $r=1$ case:
Letting $x=ry$ one finds $\partial_\theta u(x)=\partial_\theta \widetilde u(y)$ and $\norm{\partial_\theta u}_{L^2(B_r)}= r \norm{\partial_\theta\widetilde u}_{L^2(B_1)}$.
\\
Henceforward, we may assume that $r=1$.
We have the PDE
\begin{equation}
\Delta u = \mathbb A(u)(\nabla u,\nabla u)\qquad\text{in }B_1,
\end{equation}
where $\mathbb A$ denotes the second fundamental form of $\mathcal N\subset\mathbb R^m$.
Let $w:=\partial_\theta u$.
We start by showing the following Caccioppoli estimate:

\textbf{Claim 1}
\begin{equation}
\|\nabla w\|_{L^2(B_{3/4})}\le C \|w\|_{L^2(B_1)}.
\end{equation}
Differentiating the equation in $\theta$ yields, in $B_1$,
\begin{equation}\label{eq:w-eqn}
\Delta w
= d\mathbb A(u)[w](\nabla u,\nabla u) + 2\mathbb A(u)(\nabla w,\nabla u).
\end{equation}
Let $\eta\in C_c^\infty(B_1)$. Multiply \eqref{eq:w-eqn} by $\eta^2 w$ and integrate over $B_1$. Integrating by parts gives
\begin{equation}
\int_{B_1}\eta^2|\nabla w|^2
= -2\int_{B_1}\eta\,(\nabla\eta\cdot\nabla w)\cdot w
-\int_{B_1}\eta^2\, w\cdot d\mathbb A(u)[w](\nabla u,\nabla u)
-2\int_{B_1}\eta^2\, w\cdot\mathbb A(u)(\nabla w,\nabla u).
\end{equation}
We estimate each term on the right-hand side. By Cauchy--Schwarz and Young's inequality,
\begin{equation}\label{eq:cutoff}
2\Big|\int_{B_1}\eta\,(\nabla\eta\cdot\nabla w)\cdot w\Big|
\le \frac14\int_{B_1}\eta^2|\nabla w|^2 + C\int_{B_1}|\nabla\eta|^2|w|^2.
\end{equation}
Since $|d\mathbb A|$ is bounded, we have
\begin{equation}\label{eq:zero}
\Big|\int_{B_1}\eta^2\, w\cdot d\mathbb A(u)[w](\nabla u,\nabla u)\Big|
\le C\int_{B_1}\eta^2|\nabla u|^2|w|^2.
\end{equation}
Similarly, using boundedness of $\mathbb A$ and Young's inequality,
\begin{align}
2\Big|\int_{B_1}\eta^2\, w\cdot\mathbb A(u)(\nabla w,\nabla u)\Big|
&\le C\int_{B_1}\eta^2|\nabla u|\,|\nabla w|\,|w| \nonumber\\
&\le \frac{1}{4}\int_{B_1}\eta^2|\nabla w|^2 + C\int_{B_1}\eta^2|\nabla u|^2|w|^2. \label{eq:first}
\end{align}
Combining \eqref{eq:cutoff}--\eqref{eq:first} we obtain
\begin{equation}\label{eq:caccioppoli-bad}
\int_{B_1}\eta^2|\nabla w|^2
\le C\int_{B_1}|\nabla\eta|^2|w|^2
+ C\int_{B_1}\eta^2|\nabla u|^2|w|^2.
\end{equation}
In particular, choosing $0\le\eta\le1$ with $\eta\equiv 1$ on $B_{3/4}$, $\eta\equiv 0$ outside $B_{7/8}$, and $|\nabla\eta|\le C$, we conclude (using standard $\varepsilon$-regularity for $u$) that
\begin{equation}
\int_{B_{3/4}} |\nabla w|^2
\le C\int_{B_1} |w|^2,
\end{equation}
which shows claim 1. \\
Now as $w$ satisfies the elliptic equation \eqref{eq:w-eqn} we can apply interior elliptic regularity to bound
\begin{equation}
\begin{aligned} \label{eq: est of nabla square w}
\norm{\nabla^2 w}_{L^2(B_{1/2})} 
&\le C \left(\norm{w}_{L^2(B_{3/4})} + \norm{d\mathbb A(u)[w](\nabla u,\nabla u) + 2\mathbb A(u)(\nabla w,\nabla u)}_{L^2(B_{3/4})}  \right) \\
&\le C \left(\norm{w}_{L^2(B_{3/4})} + \norm{\nabla u}_{L^\infty(B_{3/4})}^2\norm{w}_{L^2(B_{3/4})} + \norm{\nabla u}_{L^\infty(B_{3/4})}\norm{\nabla w}_{L^2(B_{3/4})}  \right) \\
&\le C \norm{w}_{L^2(B_{1})},
\end{aligned}
\end{equation}
where in the last line we used Claim 1 and standard $\varepsilon$-regularity for $u$.
Using the Sobolev embedding $W^{2,2}(B_{1/2})\hookrightarrow L^\infty(B_{1/2})$, Claim 1 and \eqref{eq: est of nabla square w} we find
\begin{equation}
\norm{w}_{L^\infty(B_{1/2})}
\le C\norm{w}_{W^{2,2}(B_{1/2})} \le C\norm{w}_{L^2(B_{1/2})}.
\end{equation}
\end{proof}

This now allows us to improve the bound \eqref{EQ: jUINDUQWn2u323fnu9inufnujf3UJLCNS91edq} on the angular derivative as desired:

\begin{lemma}
\label{LEMMA: angular bd improved}
Let $u_k^\sharp$ be as before.
There exists some $\beta\in (0,2)$ and some constant $C>0$ such that for small $\eta$ and large $k$ there holds
\begin{equation}
|x^\perp \cdot \nabla u_k^\sharp(x) |^2 = |\partial_\theta u_k^\sharp(x) |^2 \le C \left[ \left(\frac{|x|}{\eta} \right)^\beta + \left( \frac{\delta_k}{|x|\eta} \right)^\beta  \right]
\|\nabla u_k^\sharp\|_{L^2(A(\eta,\delta_k))}^2,
\end{equation}
where $x\in A(\eta,\delta_k)$.
\end{lemma}

\begin{proof}
We will follow the strategy and the estimates of chapter III in \cite{DGR25}.
As explain in \cite{Riv07} the harmonic map equation can be rewritten using an antisymmetric potential $\Omega_k \in L^2(A(\et
,\delta_k), so(m))$ such that
\begin{equation}
-\Delta u_k^\sharp = \Omega_k \cdot \nabla u_k^\sharp \qquad \text{ in } A(\eta,\delta_k),
\end{equation}
where $|\Omega_k| \le C |\nabla u_k^\sharp|$.
According to the decomosition in section III.2 of \cite{DGR25} one has
\begin{equation}
\label{EQ: JINuinfe289ufn3n232f3qdq12FUJE}
A_k \nabla u_k=\nabla \varphi_k+\nabla^{\perp} \psi_k+\nabla \mathfrak{h}_k, \qquad \text{ in } A(\eta,\delta_k),
\end{equation}
where $\mathfrak h_{k}=\mathfrak h_{k}^++\mathfrak h_{k}^-+\mathfrak h_{k}^0$ is harmonic and
\begin{equation}
\mathfrak{h}_k^{+}=\Re \sum_{n>0} h_{n, k} z^n, \quad \mathfrak{h}_k^{-}=\Re \sum_{n<0} h_{n, k} z^n, \quad \mathfrak{h}_k^0=\mathfrak{h}_{0, k}+C_{\eta, k}^0 \log |z|.
\end{equation}
Also
\begin{equation}
\nabla A_k-A_k \tilde{\Omega}_k=\nabla^{\perp} B_k \quad \text { in } B_1(0),
\end{equation}
and
\begin{equation}
\begin{aligned}
\label{EQ: IUNjn2ffun2ff23fDWQjz}
& \int_{B_1}\left|\nabla A_k\right|^2+\left|\nabla A_k^{-1}\right|^2 d x^2+\int_{B_1}\left|\nabla B_k\right|^2 d x^2+\left\|\operatorname{dist}\left(\left\{A_k, A_k^{-1}\right\}, S O(m)\right)\right\|_{L^{\infty}\left(B_1(0)\right)}^2 \\
& \leq C \int_{A\left(2 \eta, \delta_k\right)}\left|\Omega_k\right|^2 d x^2 \leq C \int_{A\left(2 \eta, \delta_k\right)}\left|\nabla u_k^\sharp\right|^2 d x^2
\end{aligned}
\end{equation}
and we have also obviously
\begin{equation}
\int_{B_1}\left|\nabla B_k\right|^2 d x^2 \leq C \int_{A\left(2 \eta, \delta_k\right)}\left|\nabla u_k^\sharp\right|^2 d x^2
\end{equation}
where the Whitney extensions are chosen as in \cite{DGR25}.
Taking the inner product of \eqref{EQ: JINuinfe289ufn3n232f3qdq12FUJE} with $x^\perp$ we may write
\begin{equation}
A_k \partial_\theta u_k^\sharp = \partial_\theta \varphi_k + \nabla^\perp \psi_k \cdot x^\perp + \partial_\theta\mathfrak{h}_k^{+} + \partial_\theta\mathfrak{h}_k^{-},
\end{equation}
as crucially $\mathfrak{h}_k^{0}$ does not depend on $\theta$.
Using \eqref{EQ: IUNjn2ffun2ff23fDWQjz} with \eqref{EQ: oINijnfj1093jn131dS12d} it is clear that one can bound 
\begin{equation}
| \partial_\theta u_k^\sharp| \le  C |A_k \partial_\theta u_k^\sharp|.
\end{equation}
As in \cite{DGR25} we introduce the notation
\begin{equation}
s_1:=\left[\log \eta^{-1}\right], \quad s_2=\left[\log \left(\eta \delta_k^{-1}\right)\right], \quad A_j=B_{2^{-j}}(0) \backslash B_{2^{-j-1}}(0)
\end{equation}
For $j\in\{s_1,\dots,s_2\}$ estimate
\begin{equation}
\begin{aligned}
\int_{A_j} |\partial_\theta u_k^\sharp|^2
&\le C\left( \int_{A_j} |\partial_\theta \varphi_k|^2 + \int_{A_j} |\partial_\theta \psi_k|^2 + \int_{A_j} |\partial_\theta \mathfrak h_k^+|^2 + \int_{A_j} |\partial_\theta \mathfrak h_k^-|^2 \right) \\
&\le C2^{-2j}\left( \int_{A_j} |\nabla \varphi_k|^2 + \int_{A_j} |\nabla \psi_k|^2 + \int_{A_j} |\nabla \mathfrak h_k^+|^2 + \int_{A_j} |\nabla \mathfrak h_k^-|^2 \right) 
\end{aligned}
\end{equation}
Using the weighted Wente and the strategies of \cite{DGR25} we can derive (by following the proof of Proposition III.1)
\begin{equation}
\begin{aligned}
&\hspace{-15    mm} 2^{2j} \int_{A_j}\left|\partial_\theta u_k^\sharp \right|^2
\leq C \sum_{\ell=s_1}^{s_2} \mu^{|\ell-j|} \left( \int_{A_j} |\nabla \mathfrak h_k^+|^2 + \int_{A_j} |\nabla \mathfrak h_k^-|^2 \right) \\
&+C \sum_{\ell=s_1}^{s_2} \mu^{|\ell-j|} \gamma^{\ell} \norm{\nabla u_k^\sharp}_{L^2(A(\eta,\delta_k))}^4
+C \norm{\nabla u_k^\sharp}_{L^2(A(\eta,\delta_k))}^4\left(\mu^{j-s_1} +\mu^{s_2-j} \right) .\end{aligned}
\end{equation}
Inserting (III.23) as well as (III.24) from \cite{DGR25} one finds
\begin{equation}
 2^{2j} \int_{A_j}\left|\partial_\theta u_k^\sharp\right|^2
\le C \left[ \left(\frac{2^{-j}}{\eta} \right)^\beta + \left( \frac{\delta_k}{2^{-j}\eta} \right)^\beta  \right] \|\nabla u_k^\sharp\|_{L^2(A(\eta,\delta_k))}^2.
\end{equation}
Now for $x\in A_j$ we get with Lemma \ref{LEMMA: Angular epsilon-regulrity} applied to the ball $B_{2^{-j-2}}(x)$ one has
\begin{equation}
\begin{aligned}
|\partial_\theta u^\sharp(x)|^2 
&\le C 2^{2j} \norm{\partial_\theta u^\sharp}_{L^2(B_{2^{-j-2}}(x))}^2 \\
&\le C  2^{2j} \left( \int_{A_{j-1}}\left|\partial_\theta u_k^\sharp\right|^2 + \int_{A_j}\left|\partial_\theta u_k^\sharp\right|^2 + \int_{A_{j+1}}\left|\partial_\theta u_k^\sharp\right|^2 \right) \\
&\le C \left[ \left(\frac{2^{-j}}{\eta} \right)^\beta + \left( \frac{\delta_k}{2^{-j}\eta} \right)^\beta  \right] \|\nabla u_k^\sharp\|_{L^2(A(\eta,\delta_k))}^2.
\end{aligned}
\end{equation}
This proves the Lemma.
\end{proof}

\subsection{Convergence Behavior in the Collars}

The aim of this section is to prove that in the collar region, after rescaling, the harmonic maps are converging to a geodesic.

We denote by $(r,\theta)$ the planar polar coordinates of the annulus $A(\eta,\delta_k)$.
Let 
 \begin{equation}
 r=r(s)=e^{-(sT_{\eta,k}-\log(\eta))}, ~~s=\frac{-\log(r)+\log(\eta)}{T_{\eta,k}},~~~~s\in [0,1]\,.
 \end{equation}
Let us remark that this change of coordinates maps the annulus $A(\eta,\delta_k)$ to the unit length cylinder $[0,1]\times S^1$.
We also introduce
\begin{equation}
\tilde u_k(s,\theta)=u_k^\sharp(r(s),\theta).
\end{equation}
Notice that $\widetilde u$ is in general no longer a harmonic map as the change of coordinates $x\mapsto (\theta,s)$ is not conformal. 
We have the identities
 \begin{equation}
\label{EQ: niun2uif3iun2fmiojnmf232dahtGE}
\begin{aligned}
|\partial_s \tilde u_k(s,\theta)| &= |\partial_r  u_k^\sharp(r,\theta)|rT_{\eta,k}, \\
|\partial_\theta \tilde u_k(s,\theta)| &= |\partial_\theta  u_k^\sharp(r,\theta)|.
\end{aligned}
 \end{equation}
Let
\begin{equation}
\varepsilon_k(\sigma)=\left(\frac{\delta_k}{\eta^2}\right)^{\sigma}
\end{equation}
We observe that for $\sigma\in(0,1)$ we have $\sigma\le s\le 1-\sigma$ if and only if
\begin{equation}
\label{EQ: UINDhun23fu2fu93nj2f}
\frac{\delta_k}{\varepsilon_k(\sigma)\eta }\le r(s) \le \eta \varepsilon_k(\sigma).
\end{equation}

\begin{proposition}\label{estdertheta}
It holds 
 \begin{equation}
 \lim_{\eta\to 0}\limsup_{k\to\infty}\left\| \partial_\theta\tilde u_k\right\|_{L^{\infty}([0,1]\times S^1)}=0
 \end{equation}
 and for all $\sigma\in (0,1/2)$, any fixed $\eta>0$ and any power $\alpha\in \N$
\begin{equation}
\lim_{k\to+\infty}(T_{\eta,k})^\alpha \left\| \partial_\theta\tilde u_k \right\|_{L^{\infty}([\sigma,1-\sigma]\times S^1))} =0.
\end{equation}
\end{proposition}

\begin{proof}
Using \eqref{EQ: niun2uif3iun2fmiojnmf232dahtGE} together with the fact that $\partial_\theta u_k^\sharp = x^\perp \cdot \nabla u_k^\sharp$ one has with Lemma \ref{LEMMA: angular bd improved} to $(s,\theta)\in [0,1]\times S^1$
\begin{equation}\label{BLA}
\begin{aligned}
|\partial_\theta \widetilde u_k|^2
&= |x^\perp \cdot \nabla u_k^\sharp|^2 \\
&\le C \underbrace{\norm{\nabla u_k^\sharp}_{L^2(A(\eta,\delta_k))}^2}_{\to0}
\underbrace{\left[ \left(\frac{|x|}{\eta} \right)^\beta + \left(\frac{\delta_k}{|x|\eta} \right)^\beta \right]}_{\le 2},
\end{aligned}
\end{equation}
where we also made use of \eqref{EQ: oINijnfj1093jn131dS12d}.
Similar but with Lemma \ref{LEMMA: angular bd improved} we bound
\begin{equation}
\begin{aligned}
(T_{\eta,k})^\alpha|\partial_\theta \widetilde u_k|^2
&= (T_{\eta,k})^\alpha\ |x^\perp \cdot \nabla u_k^\sharp|^2 \\
&\le C (T_{\eta,k})^\alpha \norm{\nabla u_k^\sharp}_{L^2(A(\eta,\delta_k))}^2
\left[ \left(\frac{|x|}{\eta} \right)^\beta + \left(\frac{\delta_k}{|x|\eta} \right)^\beta \right] \\
&\le C (T_{\eta,k})^\alpha \norm{\nabla u_k^\sharp}_{L^2(A(\eta,\delta_k))}^2
\varepsilon(\sigma)^\beta \\
&= C \norm{\nabla u_k^\sharp}_{L^2(A(\eta,\delta_k))}^2
\underbrace{\left(\frac{\delta_k}{\eta^2}\right)^{\sigma\beta}(T_{\eta,k})^\alpha}_{\to0}.
\end{aligned}
\end{equation}
This shows the lemma.
\end{proof}

\begin{proposition}
\label{PROP: BD on dell s tilde u}
For every $\sigma\in (0,1)$   it holds:
 \begin{equation}
\limsup_{k\to+\infty} \left\| \partial_s\widetilde u_k \right\|_{L^{\infty}([\sigma,1-\sigma]\times S^1)}\le C\Lambda.
\end{equation}
\end{proposition}

\begin{proof}
Similar as above combining \eqref{EQ: niun2uif3iun2fmiojnmf232dahtGE} with \eqref{EQ: jUINDUQWn2u323fnu9inufnujf3UJLCNS91edq} one has
\begin{equation}
\begin{aligned}
|\partial_s \widetilde u_k|^2
&\le T_{\eta,k}^2\ |x|^2\  |\nabla u_k^\sharp|^2 \\
&\le C T_{\eta,k}^2\norm{\nabla u_k^\sharp}_{L^2(A(\eta,\delta_k))}
\left[ \left(\frac{|x|}{\eta} \right)^\beta + \left(\frac{\delta_k}{|x|\eta} \right)^\beta \right]
+ C \Lambda^2 \\
&\le C T_{\eta,k}^2\norm{\nabla u_k^\sharp}_{L^2(A(\eta,\delta_k))}
\varepsilon(\sigma)^\beta
+ C \Lambda^2 \\
&= C \norm{\nabla u_k^\sharp}_{L^2(A(\eta,\delta_k))}
\underbrace{\left(\frac{\delta_k}{\eta^2}\right)^{\sigma\beta}T_{\eta,k}^2}_{\to0}
+ C \Lambda^2.
\end{aligned}
\end{equation}
\end{proof}

To establish higher order estimates we will use the following Gagliardo-Nirenberg type interpolation Lemma

\begin{lemma}[Lemma A.1 in \cite{BBH93}]\label{lemmabbh}
Assume that
\begin{equation}
-\Delta v = f \qquad \text{on } \Omega \subset \mathbb{R}^2.
\end{equation}
Then for all $x \in \Omega$ we have
\begin{equation}\label{estbbh}
|\nabla v(x)|^2 \le C\left(
\|f\|_{L^{\infty}(\Omega)} \|v\|_{L^{\infty}(\Omega)}
+ \frac{1}{\operatorname{dist}^2(x,\partial\Omega)} \|v\|_{L^{\infty}(\Omega)}^2
\right),
\end{equation}
where $C>0$ is a global constant independent of the domain or the functions involved.
\end{lemma}

This now allows to show

\begin{lemma}
\label{LEMMA: Higher Order cont in the collars of u}
For any $x\in A(\eta,\delta_k)$ there holds
\begin{equation}
|\nabla\partial_\theta u_k^\sharp(x)|^2
\le C\left(
\|\nabla u_k^\sharp\|_{L^\infty(B_{|x|/2}(x))}^2
+\frac{1}{|x|^2}
\right)\|\partial_\theta u_k^\sharp\|_{L^\infty(B_{|x|/2}(x))}^2.
\end{equation}
for a universal constant $C>0$ that depends only on the geometry of $\mathcal N$.
\end{lemma}

\begin{proof}
The proof relies on differentiating the harmonic map equation for $u_k^\sharp$ and applying a localized Bernstein-type argument in conjunction with Lemma \ref{lemmabbh}:
We fix $x\in A(\eta,\delta_k)$ and set
\begin{equation}
R:=\frac{|x|}{2}, \qquad 
M_0:=\|\partial_\theta u_k^\sharp\|_{L^\infty(B_R(x))}, \qquad
M_1:=\|\nabla u_k^\sharp\|_{L^\infty(B_R(x))}.
\end{equation}
The harmonic map equation reads
\begin{equation}
-\Delta u_k^\sharp =\mathbb A_{u_k^\sharp}(\nabla u_k^\sharp,\nabla u_k^\sharp),
\end{equation}
Differentiating with respect to $\theta$ gives
\begin{equation}
-\Delta(\partial_\theta u_k^\sharp)
=
D\mathbb A_{u_k^\sharp}[\partial_\theta u_k^\sharp](\nabla u_k^\sharp,\nabla u_k^\sharp)
+2\mathbb A_{u_k^\sharp}(\nabla \partial_\theta u_k^\sharp,\nabla u_k^\sharp).
\end{equation}
Since $\mathbb A$ and $D\mathbb A$ are bounded on $\mathcal N$, we obtain
\begin{equation}
|-\Delta(\partial_\theta u_k^\sharp)|
\le C\big(|\partial_\theta u_k^\sharp||\nabla u_k^\sharp|^2+|\nabla u_k^\sharp|\,|\nabla \partial_\theta u_k^\sharp|\big).
\end{equation}
Now define
\begin{equation}
g(y):=(R-|y-x|)\,|\nabla \partial_\theta u_k^\sharp(y)|, \qquad y\in \overline{B_R(x)},
\end{equation}
and let $y_0\in \overline{B_R(x)}$ be a point where $g$ attains its maximum.
Note that if $y_0 \in \partial B_{R}(x)$, then $\nabla\partial_\theta u_k^\sharp(x)=0$ and the claim follows. Therefore we may assume that $y_0\in {B_R(x)}$.
Set
\begin{equation}
\rho:=\frac12(R-|y_0-x|)>0.
\end{equation}
Then for every $z\in B_\rho(y_0)$,
\begin{equation}
R-|z-x|\ge R-|y_0-x|-|z-y_0|>\rho.
\end{equation}
and hence, by maximality of $g(y_0)$,
\begin{equation}
|\nabla\partial_\theta u_k^\sharp(z)|\,\rho
\le |\nabla\partial_\theta u_k^\sharp(z)|(R-|z-x|)
\le g(y_0)
= |\nabla\partial_\theta u_k^\sharp(y_0)|(R-|y_0-x|)
=2\rho |\nabla\partial_\theta u_k^\sharp(y_0)|.
\end{equation}
Therefore
\begin{equation}
\label{EQ: JUINMiun9if2j3mf2fjnqdwwqQS}
\|\nabla\partial_\theta u_k^\sharp\|_{L^\infty(B_\rho(y_0))}
\le 2|\nabla\partial_\theta u_k^\sharp(y_0)|.
\end{equation}
Applying Lemma~\ref{lemmabbh} to each component of $v=\partial_\theta u_k^\sharp$ on the domain $\Omega=B_\rho(y_0)$ and evaluating at $y_0$, we get
\begin{equation}
|\nabla(\partial_\theta u_k^\sharp)(y_0)|^2
\le
C\left(
\|-\Delta(\partial_\theta u_k^\sharp)\|_{L^\infty(B_\rho(y_0))}
\|\partial_\theta u_k^\sharp\|_{L^\infty(B_\rho(y_0))}
+\frac{1}{\rho^2}\|\partial_\theta u_k^\sharp\|_{L^\infty(B_\rho(y_0))}^2
\right).
\end{equation}
Using $\|\partial_\theta u_k^\sharp\|_{L^\infty(B_\rho(y_0))}\le M_0$ (since $B_\rho(y_0)\subset B_R(x)$), we obtain
\begin{equation}
|\nabla\partial_\theta u_k^\sharp(y_0)|^2
\le
C\left(
\|f\|_{L^\infty(B_\rho(y_0))} M_0
+\frac{1}{\rho^2}M_0^2
\right),
\end{equation}
where
\begin{equation}
|f|\le  C\big(|\partial_\theta u_k^\sharp||\nabla u_k^\sharp|^2+|\nabla u_k^\sharp|\,|\nabla \partial_\theta u_k^\sharp|\big).
\end{equation}
Using the bound \eqref{EQ: JUINMiun9if2j3mf2fjnqdwwqQS}
\begin{equation}
\|f\|_{L^\infty(B_\rho(y_0))}
\le
C_N\big(M_0 M_1^2 +2M_1|\nabla\partial_\theta u_k^\sharp(y_0)|\big).
\end{equation}
Thus
\begin{equation}
|\nabla\partial_\theta u_k^\sharp(y_0)|^2
\le
C\left(
M_0^2 M_1^2 + M_0 M_1|\nabla\partial_\theta u_k^\sharp(y_0)| + \frac{1}{\rho^2}M_0^2
\right).
\end{equation}
By Young's inequality
\begin{equation}
CM_0 M_1|\nabla\partial_\theta u_k^\sharp(y_0)|
\le
\frac{1}{2} |\nabla\partial_\theta u_k^\sharp(y_0)|^2 + C M_0^2 M_1^2.
\end{equation}
Absorbing the second order term into the left-hand side yields
\begin{equation}
|\nabla\partial_\theta u_k^\sharp(y_0)|^2
\le
C\left(M_0^2 M_1^2 +\frac{1}{\rho^2}M_0^2\right).
\end{equation}
Multiplying by $\rho^2$ gives
\begin{equation}
\rho^2 |\nabla\partial_\theta u_k^\sharp(y_0)|^2
\le
C\left(R^2 M_0^2 M_1^2 + M_0^2\right),
\end{equation}
since $\rho\le R$.
Because
\begin{equation}
g(y_0)=2\rho |\nabla\partial_\theta u_k^\sharp(y_0)|,
\end{equation}
we get
\begin{equation}
g(y_0)^2 \le C\left(R^2M_0^2 M_1^2+M_0^2\right).
\end{equation}
Finally, since $g(x)=R|\nabla\partial_\theta u_k^\sharp(x)|\le g(y_0)$,
\begin{equation}
R^2|\nabla\partial_\theta u_k^\sharp(x)|^2
\le
C\left(R^2M_0^2 M_1^2+M_0^2\right).
\end{equation}
Dividing by $R^2$ and recalling $R=|x|/2$, $M_0$, and $M_1$, we conclude
\begin{equation}
|\nabla\partial_\theta u_k^\sharp(x)|^2
\le C\left(
\|\nabla u_k^\sharp\|_{L^\infty(B_{|x|/2}(x))}^2
+\frac{1}{|x|^2}
\right)\|\partial_\theta u_k^\sharp\|_{L^\infty(B_{|x|/2}(x))}^2.
\end{equation}
\end{proof}

\begin{lemma} \label{LEMMA: Hgher ordr bds on u} 
\begin{enumerate}
\item It holds \begin{equation} \label{EQ: IJNUfn2jnfwjf23f}
\lim_{\eta\to0} \limsup_{k\to \infty} \left( \norm{\partial_\theta^2 \widetilde u_k}_{L^\infty((0,1)\times S^1)}^2 + \frac{1}{T_{\eta,k}^2 } \norm{\partial_s \partial_\theta \widetilde u_k}_{L^\infty((0,1)\times S^1)} ^2\right)=0 
\end{equation}
\item For all $\sigma\in (0,1/2)$, any fixed $\eta>0$ and any power $\alpha\in \N$
\begin{equation} \label{EQ: JNINUIjuwenf2984hfun2j3dFDUHI}
\lim_{k\to+\infty}(T_{\eta,k})^\alpha \left( \norm{\partial_\theta^2 \widetilde u_k}_{L^\infty([\sigma,1-\sigma]\times S^1)}^2 + \norm{\partial_s \partial_\theta \widetilde u_k}_{L^\infty([\sigma,1-\sigma]\times S^1)} ^2\right) = 0.
\end{equation}
\item For all $\sigma\in (0,1/2)$ and any fixed $\eta>0$
 \begin{equation}
\limsup_{k\to+\infty} \left\| \partial_s^2 \widetilde u_k \right\|_{L^{\infty}([\sigma,1-\sigma]\times S^1)}\le C\Lambda^2.
\end{equation}
\end{enumerate}
\end{lemma}

\begin{proof}
Notice that $y \in B_{|x|/2}(x)$ implies $\frac{|x|}{2}\le |y| \le 2|x|$.
Hence, combining Lemma \ref{LEMMA: Higher Order cont in the collars of u}, \eqref{EQ: jUINDUQWn2u323fnu9inufnujf3UJLCNS91edq} and then Lemma \ref{LEMMA: angular bd improved} one has at any  $x\in A(\eta,\delta_k)$ that 
\begin{equation}
\label{EQ: JHINDZNU78h3nufdjnjniwHNdu1h3jnqdJN}
\begin{aligned}
|\nabla\partial_\theta u_k^\sharp(x)|^2
& \le C \frac{1}{|x|^2} \|\partial_\theta u_k^\sharp\|_{L^\infty(B_{|x|/2}(x))}^2 \\
& \le C \frac{1}{|x|^2} \left[ \left(\frac{|x|}{\eta} \right)^\beta + \left( \frac{\delta_k}{|x|\eta} \right)^\beta  \right] \|\nabla u_k^\sharp\|_{L^2(A(\eta,\delta_k))}^2
\end{aligned}
\end{equation}
Now note that
\begin{equation}
|\partial_\theta^2 \widetilde u_k| = |x^\perp \cdot \nabla \partial_\theta u_k^\sharp |, \qquad
|\partial_s \partial_\theta \widetilde u_k| = T_{\eta,k} |x\cdot\nabla\partial_\theta u_k^\sharp |
\end{equation}
and therefore with \eqref{EQ: JHINDZNU78h3nufdjnjniwHNdu1h3jnqdJN} we find that for any exponent $\alpha \in \N$ and $x\in A(\eta,\delta_k)$ [equivalently $(s(x),\theta(x))\in (0,1)\times S^1$]
\begin{equation} \label{EQ: JINui2n3ui23nmfjmf23wda85g}
|\partial_\theta^2 \widetilde u_k|^2 + \frac{1}{T_{\eta,k}^2 } |\partial_s \partial_\theta \widetilde u_k| ^2
\le C \left[ \left(\frac{|x|}{\eta} \right)^\beta + \left( \frac{\delta_k}{|x|\eta} \right)^\beta  \right] \|\nabla u_k^\sharp\|_{L^2(A(\eta,\delta_k))}^2
\end{equation}
Now \eqref{EQ: IJNUfn2jnfwjf23f} follows from applying \eqref{EQ: oINijnfj1093jn131dS12d} to \eqref{EQ: JINui2n3ui23nmfjmf23wda85g}.
Now to $x\in A(\eta\varepsilon, \delta_k)$ [equivalently $(s(x),\theta(x))\in (\sigma,1-\sigma)\times S^1$] we have with \eqref{EQ: JINui2n3ui23nmfjmf23wda85g}
\begin{equation}
\begin{aligned}
(T_{\eta,k})^\alpha \left( |\partial_\theta^2 \widetilde u_k|^2 +  |\partial_s \partial_\theta \widetilde u_k| ^2 \right)
&\le C (T_{\eta,k})^{\alpha+2} \left[ \left(\frac{|x|}{\eta} \right)^\beta + \left( \frac{\delta_k}{|x|\eta} \right)^\beta  \right] \|\nabla u_k^\sharp\|_{L^2(A(\eta,\delta_k))}^2 \\
&\le C \underbrace{(T_{\eta,k})^{\alpha+2} \left(\frac{\delta_k}{\eta^2} \right)^{\beta\sigma}}_{\to0},
\end{aligned}
\end{equation}
as desired.
Now notice that $\widetilde u_k$ satisfies the PDE
\begin{equation}
-\partial_s^2 \widetilde u_k - T_{\eta,k}^2 \,\partial_\theta^2 \widetilde u_k
=\mathbb A_{\widetilde u_k}\bigl(\partial_s \widetilde u_k,\partial_s \widetilde u_k\bigr) + T_{\eta,k}^2\,\mathbb A_{\widetilde u_k}\bigl(\partial_\theta \widetilde u_k,\partial_\theta \widetilde u_k\bigr),
\qquad \text{ in } [0,1] \times S^1
\end{equation}
and hence one has the pointwise bounds
\begin{equation}\label{EQ: JUINUzibnfiu2nefji2f2f23d2AQS31858}
\begin{aligned}
|\partial_s^2 \widetilde u_k|
&\le T_{\eta,k}^2 |\partial_\theta^2 \widetilde u_k| + C |\partial_s \widetilde u_k|^2 + C T_{\eta,k}^2 |\partial_\theta \widetilde u_k|^2
\end{aligned}
\end{equation}
Applying \eqref{EQ: JNINUIjuwenf2984hfun2j3dFDUHI}, Propostion \ref{estdertheta} and Proposition \ref{PROP: BD on dell s tilde u} to \eqref{EQ: JUINUzibnfiu2nefji2f2f23d2AQS31858} completes the proof.
\end{proof}

This allows us to now prove the claimed convergence result in Theorem \ref{THEOREM: Conv to geod general}:

\begin{proposition}\label{congeo}
There is a subsequence of $\widetilde u_k$ (that we keep denoting by $\widetilde u_k$) and a geodesic segment $\widetilde{\gamma}_{\infty}\in W^{1,2}_{loc}((0,1);\mathcal N)$ of $\mathcal N$ such that for any fixed $\sigma\in(0,1/2)$ one has
\begin{equation}
\widetilde u_k \rightharpoonup \widetilde \gamma_\infty \qquad \text{ weakly in } W^{1,2}((0,1)\times S^1)
\end{equation}
and 
\begin{equation}
\widetilde u_k \to \widetilde \gamma_\infty \qquad \text{ strongly in } C^1((\sigma,1-\sigma)\times S^1).
\end{equation}
\end{proposition}

\begin{proof}
The bounds in Proposition \ref{estdertheta}, Proposition \ref{PROP: BD on dell s tilde u} and Lemma \ref{LEMMA: Hgher ordr bds on u} together with Eberlein-Smulian and Arzel\`a-Ascoli are implying the existence of a map $\widetilde\gamma_\infty\in W^{1,2}((\sigma,1-\sigma)\times S^1)$ with
\begin{equation}
\widetilde u_k \rightharpoondown \widetilde \gamma_\infty \qquad \text{ weakly in } W^{1,2}((0,1)\times S^1)
\end{equation}
and 
\begin{equation}
\widetilde u_k \to \widetilde \gamma_\infty \qquad \text{ strongly in } C^1((\sigma,1-\sigma)\times S^1).
\end{equation}
The decay of Proposition \ref{estdertheta} implies $\partial_\theta \widetilde\gamma_\infty=0$.
Testing the harmonic map equation \eqref{EQ: HArm map Equation for u_k} with a test function $\varphi$ it is clear that
\begin{equation}
\begin{aligned}
\langle - \Delta \widetilde\gamma_\infty -\mathbb A_{\widetilde\gamma_\infty}(\nabla \widetilde\gamma_\infty, \nabla \widetilde\gamma_\infty),  \varphi \rangle
&= \langle \nabla \widetilde\gamma_\infty, \nabla \varphi \rangle_{L^2} - \langle  \mathbb A_{\widetilde\gamma_\infty}(\nabla \widetilde\gamma_\infty, \nabla \widetilde\gamma_\infty),  \varphi \rangle_{L^2} \\
&= \lim_{k\to\infty} \langle \nabla \widetilde u_k, \nabla \varphi \rangle_{L^2} - \langle  \mathbb A_{\widetilde u_k}(\nabla \widetilde u_k, \nabla \widetilde u_k),  \varphi \rangle_{L^2} \\
&=0.
\end{aligned}
\end{equation}
This implies that $\widetilde\gamma_\infty$ satisfies in distribution the geodesic equation
\begin{equation}
 -  \partial_s^2 \widetilde\gamma_\infty =\mathbb A_{\widetilde\gamma_\infty}(\partial_s \widetilde\gamma_\infty, \partial_s \widetilde\gamma_\infty).
\end{equation}
As this proves that $\widetilde \gamma_\infty$ is a geodesic of $\mathcal N$, the proof of the claim is finished.
\end{proof}

\subsection{No Neck Property for the Connecting Regions}

We have now established that, away from the collars, the sequence of harmonic maps $u_k$ converges to a base map $u_\infty$ and that, in the collars, after rescaling suitably, these are converging locally to a geodesic.
In this section we will treat the connecting regions of these two scales of convergence. 

\par
\bigskip
\begin{lemma} \label{LEMMA: L21 quant is conn necks}
Let $\sigma \in(0,1)$, and
\begin{equation}
\Omega=B_{\eta}(0) \backslash \bar{B}_{\varepsilon_k(\sigma) \eta}(0) \qquad \text{or} \qquad \Omega=B_{\frac{\delta_k}{\varepsilon_k(\sigma) \eta}}(0) \backslash \bar{B}_{\frac{\delta_k}{\eta}}(0)
\end{equation}
where $\varepsilon_k(\sigma)=\left(\frac{\delta_k}{\eta^2}\right)^\sigma$ as above.
Then we have $\nabla u_k^\sharp \in L^{2,1}(\Omega)$ with
\begin{equation}
\lim _{\sigma \rightarrow 0} \lim _{\eta \rightarrow 0} \limsup _{k \rightarrow+\infty}\left\|\nabla u_k^\sharp\right\|_{L^{2,1}(\Omega)}=0
\end{equation}
\end{lemma}

\begin{proof}

We   prove the Lemma in the case $\Omega=B(0,\eta)\setminus\bar{B}(0,\varepsilon_k(\sigma) \eta)$, the other case being analogous.\par
By definition of the $L^{2,1}$ norm we have 
\begin{equation}
||\nabla u_k^\sharp ||_{L^{2,1}(B(0,\eta)\setminus\bar{B}(0,\varepsilon_k(\sigma) \eta))}= 2 \int_{0}^\infty\left(\mu\{x\in B(0,\eta)\setminus\bar{B}(0,\varepsilon_k(\sigma) \eta)\,\,\vert\,\,|\nabla u_k^\sharp|\geq t\} \right)^{1/2} dt\,.
\end{equation}
In the following we use the notation $o_{k,\eta}(1)$ to denote a constant depending on $\eta,\delta_k$ that may change from line to line but always satisfies $\lim_{\eta\to0}\limsup_{k\to\infty}o_{k,\eta}(1)=0$.
We are going to use \eqref{EQ: jUINDUQWn2u323fnu9inufnujf3UJLCNS91edq} with the fact that \eqref{EQ: oINijnfj1093jn131dS12d}.
Using the triangle inequality and the homogeneity of the norm, we have
\begin{align}
&2 \int_{0}^\infty \left(\mu\{x\in B(0,\eta)\setminus\bar{B}(0,\varepsilon_k(\sigma) \eta)\,\,\vert\,\,|\nabla u_k^\sharp|\geq t\} \right)^{1/2}dt \nonumber\\ \noindent\leq&\underbrace{o_{k,\eta}(1)\frac{1}{\eta^{\beta/2}}2 \int_{0}^\infty \left(  \mu\{x\in B(0,\eta)\setminus\bar{B}(0,\varepsilon_k(\sigma) \eta)\,\,\vert\,\, |x|^{\beta/2-1}\geq t\}\right)^{1/2}dt}_{(1)}\nonumber\\ \noindent +&
\underbrace{o_{k,\eta}(1)\frac{(\delta_k)^{\beta/2}}{\eta^{\beta/2}}2 \int_{0}^\infty\left( \mu\{x\in B(0,\eta)\setminus\bar{B}(0,\varepsilon_k(\sigma) \eta)\,\,\vert\,\,\frac{1}{|x|^{1+\beta/2}}\geq t\} \right)^{1/2}dt}_{(2)}\nonumber\\ \noindent +&
\underbrace{\frac{\Lambda}{\log(\frac{\eta^2}{\delta_k})}2 \int_{0}^\infty \left( \mu\{x\in B(0,\eta)\setminus\bar{B}(0,\varepsilon_k(\sigma) \eta)\,\,\vert\,\, \frac{1}{|x|} \geq t\} \right)^{1/2}dt}_{(3)}
\end{align}
{\bf \underline{Estimate of $(3)$:}}
We have
\begin{eqnarray}
&&2\frac{\Lambda}{\log(\frac{\eta^2}{\delta_k})}  \int_{0}^\infty\left(\mu\{x\in B(0,\eta)\setminus\bar{B}(0,\varepsilon_k(\sigma) \eta),\,\vert\,\, \frac{1}{|x|}\geq t\}\right)^{1/2}dt\\
&&=2\frac{\Lambda}{\log(\frac{\eta^2}{\delta_k})}  \int_{0}^{(\varepsilon_k(\sigma) \eta)^{-1}}\left(\pi(\min(t^{-1},\eta)^2 -  \varepsilon_k(\sigma)^2\eta^2 )\right)^{1/2}dt\\
&&\le2\frac{\Lambda\sqrt{\pi}}{\log(\frac{\eta^2}{\delta_k})}  \int_{0}^{(\varepsilon_k(\sigma) \eta)^{-1}}\min(t^{-1},\eta)\ dt\\
&&\le2\frac{\Lambda\sqrt{\pi}}{\log(\frac{\eta^2}{\delta_k})} \left( \int_{\eta^{-1}}^{(\varepsilon_k(\sigma) \eta)^{-1}}\frac{1}{t}\ dt + \int_0^{\eta^{-1}}\eta\ dt\right) \\
&&=2\frac{\Lambda\sqrt{\pi}}{\log(\frac{\eta^2}{\delta_k})} \left( \log((\varepsilon_k(\sigma))^{-1}) +1 \right) \\
&&= 2\Lambda\sqrt{\pi} \left( \sigma + \frac{1}{\log(\frac{\eta^2}{\delta_k})} \right) .
\end{eqnarray}
For simplicity now we set $\alpha:=\beta/2$.\par
{\bf \underline{Estimate of $(1)$:}}
 
\begin{eqnarray}
&&\displaystyle{2 o_{k,\eta}(1) \frac{1}{\eta^{\alpha}}  \int_{0}^\infty
\left( \mu\{x\in B(0,\eta)\setminus\bar{B}(0,\varepsilon_k(\sigma) \eta)\,\,\vert\,\,   {|x|^{\alpha-1}} \geq t\} \right)^{1/2}dt}\\
&&\displaystyle{=2 o_{k,\eta}(1) \frac{1}{\eta^{\alpha}}  \int_{0}^{\left(1/{\varepsilon_k(\sigma) \eta}\right)^{1-\alpha}}
\left( \pi \min(\eta,t^{-1/(1-\alpha)})^2 - \pi \varepsilon_k(\sigma)^2 \eta^2 \right)^{1/2}dt}\\
&&\displaystyle{\le2 o_{k,\eta}(1) \frac{1}{\eta^{\alpha}}  \int_{0}^{\left(1/{\varepsilon_k(\sigma) \eta}\right)^{1-\alpha}}
\sqrt\pi \min(\eta,t^{-1/(1-\alpha)}) dt}\\
&&\le\displaystyle{ 2 o_{k,\eta}(1) \frac{\sqrt\pi}{\eta^{\alpha}} \int_{\left(\frac{1}{\eta}\right)^{1-\alpha}}^{\left(1/{\varepsilon_k(\sigma) \eta}\right)^{1-\alpha}}  t^{-1/(1-\alpha)}\ dt + \int_0^{\left(\frac{1}{\eta}\right)^{1-\alpha}}  \eta\ dt}\\
&&\displaystyle{  =2 o_{k,\eta}(1)\frac{\sqrt{\pi}}{\eta^{\alpha}}\left(\frac{\alpha-1}{\alpha} \left[t^{\frac{\alpha}{\alpha-1}}\right]\left|_{(\frac{1}{\eta})^{1-\alpha}}^{(1/{\varepsilon_k(\sigma) \eta})^{1-\alpha}}\right. + \eta^\alpha \right)}\\
&&\displaystyle{  =2 o_{k,\eta}(1)\frac{\sqrt{\pi}}{\eta^{\alpha}}\left(\frac{\alpha-1}{\alpha} (\varepsilon_k(\sigma)^\alpha\eta^\alpha - \eta^\alpha) + \eta^\alpha \right)}\\
&&\displaystyle{  \le2 o_{k,\eta}(1)\frac{\sqrt{\pi}}{\eta^{\alpha}}\left(\frac{1-\alpha}{\alpha} \eta^\alpha + \eta^\alpha \right)}\\
&&\displaystyle{ = 2  o_{k,\eta}(1)\frac{\sqrt{\pi}}{\alpha}}.
\end{eqnarray}
{\bf \underline{Estimate of $(2)$}:}
\begin{align}
&2o_{k,\eta}(1)\frac{(\delta_k)^{\alpha}}{\eta^{\alpha}}
\int_{0}^{\infty}\Big(\mu\{x\in B(0,\eta)\setminus \overline{B}(0,\varepsilon_k(\sigma)\eta)\mid |x|^{-(1+\alpha)}\ge t\}\Big)^{1/2}\,dt
\\
&=2o_{k,\eta}(1)\frac{(\delta_k)^{\alpha}}{\eta^{\alpha}}
\int_{0}^{\left(1/{\varepsilon_k(\sigma)\eta}\right)^{1+\alpha}}
\Big(\pi\,\min\!\big(\eta,t^{-1/(1+\alpha)}\big)^2-\pi\,\varepsilon_k(\sigma)^2\eta^2\Big)^{1/2}\,dt
\\
&\le 2o_{k,\eta}(1)\frac{(\delta_k)^{\alpha}}{\eta^{\alpha}}
\int_{0}^{\left(1/{\varepsilon_k(\sigma)\eta}\right)^{1+\alpha}}
\sqrt{\pi}\,\min\!\big(\eta,t^{-1/(1+\alpha)}\big)\,dt
\\
&\le 2o_{k,\eta}(1)\frac{(\delta_k)^{\alpha}\sqrt{\pi}}{\eta^{\alpha}}
\left(
\int_{\left(\frac{1}{\eta}\right)^{1+\alpha}}^{\left(1/{\varepsilon_k(\sigma)\eta}\right)^{1+\alpha}}
t^{-1/(1+\alpha)}\,dt
+\int_{0}^{\left(\frac{1}{\eta}\right)^{1+\alpha}} \eta\,dt
\right)
\\
&=2o_{k,\eta}(1)\frac{(\delta_k)^{\alpha}\sqrt{\pi}}{\eta^{\alpha}}
\left(
\frac{1+\alpha}{\alpha}\Big[t^{\alpha/(1+\alpha)}\Big]_{\left(\frac{1}{\eta}\right)^{1+\alpha}}^{\left(1/{\varepsilon_k(\sigma)\eta}\right)^{1+\alpha}}
+\eta\left(\frac{1}{\eta}\right)^{1+\alpha}
\right)
\\
&=2o_{k,\eta}(1)\frac{(\delta_k)^{\alpha}\sqrt{\pi}}{\eta^{\alpha}}
\left(
\frac{1+\alpha}{\alpha}\Big((\varepsilon_k(\sigma)\eta)^{-\alpha}-\eta^{-\alpha}\Big)
+\eta^{-\alpha}
\right)
\\
&=2o_{k,\eta}(1)\frac{(\delta_k)^{\alpha}\sqrt{\pi}}{\eta^{\alpha}}
\left(
\frac{1+\alpha}{\alpha}(\varepsilon_k(\sigma)\eta)^{-\alpha}-\frac{1}{\alpha}\eta^{-\alpha}
\right)
\\
&\le 2o_{k,\eta}(1)\frac{(\delta_k)^{\alpha}\sqrt{\pi}}{\eta^{\alpha}}
\left(
\frac{1+\alpha}{\alpha}(\varepsilon_k(\sigma)\eta)^{-\alpha}
\right)
\\
&=2o_{k,\eta}(1)\sqrt{\pi}\,\frac{1+\alpha}{\alpha}\,
\frac{(\delta_k)^{\alpha}}{\varepsilon_k(\sigma)^{\alpha}\eta^{2\alpha}} \\
&=2o_{k,\eta}(1)\sqrt{\pi}\,\frac{1+\alpha}{\alpha}\,
\left( \frac{\delta_k}{\eta^2} \right)^{\alpha(1-\sigma)}.
\end{align}
Combining Estimates (1), (2) and (3) and sending first $k\to \infty$, then $\eta\to0$ and then $\sigma\to0$ the lemma follows.
\end{proof}

\section{Stability of the Morse Index}
\label{SECTION: Stability of the Morse Index}

In this section we are going to prove that the connecting Necks connecting the geodesic segment $\widetilde \gamma_\infty$ and the map $u_\infty$ are asymptotically not contributing to the negativity of the second variation.

\subsection{Positive contribution of the connecting Necks}

The following result was proven in Lemma V.1 of \cite{DGR25} (and estimating the logarithmic term as in \eqref{EQ: UNui92efcijf2JNdqfdq321}).

\begin{lemma} \label{LEMMA: PTS bd of na u in conn regions}
Let $\sigma \in(0,1/2)$ and $\varepsilon_k(\sigma)=\left(\frac{\delta_k}{\eta^2}\right)^\sigma$ as before.
\begin{enumerate}
\item If $\Omega=B_{\eta}(0) \backslash \bar{B}_{\varepsilon_k(\sigma) \eta}(0)$, then for any $x\in \Omega$ one has
\begin{equation}
|x|^2 |\nabla u_k^\sharp|^2 \le C\left[\left(\frac{|x| }{\eta}\right)^\beta+\left(\frac{\eta\varepsilon}{|x|}\right)^\beta\right] \norm{\nabla u_k^\sharp}_{L^2(A\left(2 \eta, \delta_k\right))} +C \frac{\left\|\nabla u_k^\sharp\right\|_{L^{2,1}(\Omega)}^2}{\sigma^2\log ^2\left(\frac{\eta^2}{\delta_k}\right)}.
\end{equation}
\item If $\Omega=B_{{\delta_k}/{\varepsilon_k(\sigma) \eta}}(0) \backslash \bar{B}_{{\delta_k}/{\eta}}(0)$, then for any $x\in \Omega$ one has
\begin{equation}
|x|^2 |\nabla u_k^\sharp|^2 \le C\left[\left(\frac{|x| \eta\varepsilon}{\delta_k}\right)^\beta+\left(\frac{\delta_k}{\eta|x|}\right)^\beta\right] \norm{\nabla u_k^\sharp}_{L^2(A\left(2 \eta, \delta_k\right))} +C \frac{\left\|\nabla u_k^\sharp\right\|_{L^{2,1}(\Omega)}^2}{\sigma^2\log ^2\left(\frac{\eta^2}{\delta_k}\right)}.
\end{equation}
\end{enumerate}
\end{lemma}
Recall that $\eta = \exp(-\pi/\varrho)$, $\delta_k = \exp(-2\pi^2/l_k)$ and $T_{\eta,k}=\log \left(\frac{\eta^2}{\delta_k}\right)$.
Let us now introduce the weight function in the long cylinder collars $\mathtt{P}_k^\varrho$
\begin{equation}
\label{EQ: main Weight def}
\mathfrak w_{\sigma,\eta,k}(t,\theta)=
\begin{cases}
\left( \frac{e^{-t}}{\eta} \right)^\beta
+\left(\frac{\eta\varepsilon}{e^{-t}} \right)^\beta
+\frac{1}{\sigma^2T_{\eta,k}^2},
& \text{if } e^{-t}\in[\eta\varepsilon,\eta],\\
\left( \frac{e^{-t}}{\eta\varepsilon} \right)^\beta
+\left(\frac{\delta_k}{\eta\varepsilon\,e^{-t}} \right)^\beta
+\frac{1}{(1-2\sigma)^2T_{\eta,k}^2},
& \text{if } e^{-t}\in\left[\frac{\delta_k}{\eta\varepsilon},\,\eta\varepsilon\right],\\
\left( \frac{e^{-t}\eta\varepsilon}{\delta_k} \right)^\beta
+\left(\frac{\delta_k}{\eta\,e^{-t}} \right)^\beta
+\frac{1}{\sigma^2T_{\eta,k}^2},
& \text{if } e^{-t}\in\left[\frac{\delta_k}{\eta},\,\frac{\delta_k}{\eta\varepsilon}\right].
\end{cases}
\end{equation}
as well as the normalization weight on $\Sigma_k$
\begin{equation}
\label{EQ: UNuh3nfnjidqmk3r9u2mj096529}
\omega_{\eta,k}(q) =
\begin{cases}
\left( \left(\frac{e^{-t}}{\eta}\right)^\beta + \left(\frac{\delta_k}{\eta e^{-t}}\right)^\beta + \frac{1}{T_{\eta,k}^2} \right) &\text{ if } q=(t,\theta)\in \mathtt P_k^\varrho \qquad \left(e^{-t}\in\left[\frac{\delta_k}{\eta},\eta\right]\right), \\
\left( 1 + \left(\frac{\delta_k}{\eta^2}\right)^\beta + \frac{1}{T_{\eta,k}^2} \right) &\text{ if } q\in \Sigma_k \setminus \mathtt P_k^\varrho.
\end{cases}
\end{equation}

Let us remark that $\omega_{\eta,k}$ is well defined on all of $\Sigma_k$ and continuous since $\omega_{\eta,k}|_{\partial\mathtt P_k^\varrho}=\left( 1+ \frac{\delta_k^\beta}{\eta^{2\beta}} + \frac{1}{T_{\eta,k}^2} \right)$.
Introduce the notation
\begin{equation}
\widetilde \omega_{\eta,k}(s,\theta) \coloneqq \omega_{\eta,k}(t(s),\theta) \qquad
\widetilde {\mathfrak w}_{\sigma,\eta,k}(s,\theta) \coloneqq \mathfrak w_{\sigma,\eta,k}(t(s),\theta), \qquad
(s,\theta) \in [0,1] \times S^1.
\end{equation}
as well as
\begin{equation}
\omega_{\eta,k}^\sharp(x) \coloneqq \frac{1}{|x|^2} \omega_{\eta,k}(t(x),\theta(x)) \qquad
{\mathfrak w}_{\sigma,\eta,k}^\sharp(x) \coloneqq \frac{1}{|x|^2} \mathfrak w_{\sigma,\eta,k}(t(x),\theta(x)), \qquad
x\in A(\eta,\delta_k).
\end{equation}
For any $\sigma\in(0,1/2)$ observe the following crucial fact
\begin{equation}
\label{EQ: UInuh82fu9irjIUHNhuf2njuw8375nHnud}
\mathfrak w_{\sigma,\eta,k}(t,\theta) \ge \omega_{\eta,k}(t,\theta) \qquad (t,\theta) \in \mathtt P_k^\varrho.
\end{equation}
We introduce the standard notation
\begin{equation}
L^p_{\omega_{\eta,k}} \coloneqq \left\{ f:\Sigma_k \to \R ; \ \norm{f}_{L^p_{\omega_{\eta,k}}}^p \coloneqq \int_{\Sigma} |f|^p\ \omega_{\eta,k} \ dvol_{h_k}<\infty \ \right\}.
\end{equation}

\begin{theorem}\label{THEOREM: Pos in conn necks} Let $\sigma \in(0,1)$, $\Omega=B_\eta(0) \backslash \bar{B}_{\varepsilon_k(\sigma) \eta}(0)$\\ or $\Omega=B_{ \frac{\delta_k}{\varepsilon_k(\sigma) \eta}}(0) \backslash \bar{B}_{ \frac{\delta_k}{\eta}}(0)$ where $\varepsilon_k(\sigma)=\left(\frac{\delta_k}{\eta^2}\right)^\sigma$.
 For every $\beta \in\left(0, \log _2(3 / 2)\right)$ there exists some constant $\bar{\kappa}>0$ (independent of $k,\eta,\sigma$) such that for $k \in \mathbb{N}$ large, $\eta>0$ small and $\sigma>0$ small one has for all $w \in V_{u_k}$
\begin{equation}
\left(w^\sharp =0 \text { in } \Sigma_k \setminus \Omega \right) \Rightarrow Q_{u_k}(w) \geq \bar{\kappa} \int_{\Sigma_k}|w|^2 \mathfrak w_{\sigma,\eta,k} d v o l_{h_k} \geq \bar{\kappa} \int_{\Sigma_k}|w|^2 \omega_{\eta, k} d v o l_{h_k} \geq 0
\end{equation}
\end{theorem}

\begin{proof}
We will only treat the case $\Omega=B_\eta(0) \backslash \bar{B}_{\varepsilon_k(\sigma) \eta}(0)$, the other being completely analogous.
Let $w \in V_{u_k}$ with $w^\sharp=0 \text{ in } \Sigma \setminus \Omega$.
(Here again as usual $w^\sharp(x)= w (t(x),\theta(x))$).
According to Lemma IV.1 of \cite{DGR25} there exists some constant $c_0>0$ independent of $k$, $\eta$ and $\sigma$ such that	to any function $f:\Omega \to \R$ one has the Poincar\'e type estimate
\begin{equation}
\int_{\Omega} |\nabla f|^2 dx \ge c_0 \int_{\Omega} | f|^2 {\mathfrak w}_{\sigma,\eta,k}^\sharp\ dx
\end{equation}
We can now estimate
\begin{equation}
\begin{aligned}
Q_{u_k}(w_k)
&= \int_\Omega \left(\abs{\nabla w_k^\sharp}^2 -\mathbb A_{u_k^\sharp}(\nabla u_k^\sharp, \nabla u_k^\sharp) \cdot\mathbb A_{u_k^\sharp}(w_k^\sharp,w_k^\sharp) \right) \ dx \\
&\ge \int_{\Omega} \left( c_0 {\mathfrak w}_{\sigma,\eta,k}^\sharp - C |\nabla u_k^\sharp|^2  \right) |w_k^\sharp|^2 \ dx
\end{aligned}
\end{equation}
Lemma \ref{LEMMA: PTS bd of na u in conn regions} together with \eqref{EQ: oINijnfj1093jn131dS12d} and Lemma \ref{LEMMA: L21 quant is conn necks} imply that there exists a constant $\bar\kappa>0$ such that for $k \in \mathbb{N}$ large, $\eta>0$ small and $\sigma>0$ small one has
\begin{equation}
\left( c_0 {\mathfrak w}_{\sigma,\eta,k}^\sharp - C |\nabla u_k^\sharp|^2  \right) \ge \bar\kappa\  {\mathfrak w}_{\sigma,\eta,k}^\sharp \qquad \text{ in all of }\Omega. 
\end{equation}
Together with \eqref{EQ: UInuh82fu9irjIUHNhuf2njuw8375nHnud} this proves the claim. 
\end{proof}

\subsection{The Diagonalization of $Q_{u_k}$ with respect to
the Weights $\omega_{\eta,k}$}

In this section we will precisely explain how to diagonalize the Jacobi operator associated to the second variation.
We then proceed to show the stability of the index by suitably defining meaningful limiting variations.
This will illustrate why the choice of the weights \eqref{EQ: main Weight def} is suitable to prove our main result.

Let $\mathrm{n}_1,\dots,\mathrm{n}_{m-n} \in \Gamma\big((T\mathcal N)^\perp\big)$ be an orthonormal frame of the normal bundle of $\mathcal N \subset \R^m$.
Define
\begin{equation}
\label{EQ: iwenf8932injIUJN128ee12}
S_{u_{k}}(\nabla u_{k}) \coloneqq- \sum_{j=1}^{m-n} \Big\langle\mathbb A_{u_k}(\nabla u_k,\nabla u_k), \mathrm n_j(u_k) \Big\rangle\ D(\mathrm n_{j})_{u_k}
\end{equation}
such that for any tangent vector fields $X,Y \in \Gamma (u_k^{-1} T\mathcal N)$ we have
\begin{equation}
\mathbb A_{u_k}(\nabla u_k, \nabla u_k) \cdot\mathbb A_{u_k}(X,Y)
= (S_{u_{k}}(\nabla u_{k}) X) \cdot Y,
\end{equation}
and the pointwise bound
\begin{equation}
\abs{S_{u_{k}}(\nabla u_{k}) X}
\le C \abs{\nabla u_k}^2 \abs{X}.
\end{equation}
We denote by
\begin{equation}
\langle w, v\rangle_{\omega_{\eta, k}}
=\int_{\Sigma_k} w \cdot v\ \omega_{\eta, k} d v o l_{h_k}
\end{equation}
We consider the diagonalization of the self-adjoint operator $\mathcal{L}_{\eta, k}$ on $V_{u_k}$ with respect to $\langle,\rangle\omega_{\eta, k}$ given by
\begin{equation}
Q_{u_k}(w)
=\langle\mathcal{L}_{\eta, k} w, w\rangle_{\omega_{\eta, k}}
=\int_{\Sigma_k}\left[-\omega_{\eta, k}^{-1} \Delta_h w-\omega_{\eta, k}^{-1}\ S_{u_{k}}(\nabla u_{k}) w \right] \cdot w\ \omega_{\eta, k}\ d v o l_{h_k}
\end{equation}
The operator $\mathcal{L}_{\eta, k}$ is given by
\begin{equation}
\mathcal{L}_{\eta, k}= \omega_{\eta, k}^{-1} P_{u_k}\left( -\Delta_h w- S_{u_{k}}(\nabla u_{k}) w\right),
\end{equation}
where $P_{u_k}(x): \R^{m}\rightarrow T_{u_k(x)}\mathcal N$ is the orthogonal $m\times m$ projection matrix.
Note that by construction $\mathcal L_{\eta,k}$ is self-adjoint with respect to the inner product $\langle \cdot ,\cdot \rangle_{\omega_{\eta,k}}$, i.e.
\begin{equation}\label{EQ: wijenfuNIUD138en9qd}
\langle \mathcal L_{\eta,k} w,v \rangle_{\omega_{\eta,k}} = \langle w, \mathcal L_{\eta,k} v \rangle_{\omega_{\eta,k}} .
\end{equation}
We denote by
\begin{equation}
\mathcal{E}_{\eta, k}(\lambda):=\left\{w \in V_{u_k}: \quad \mathcal{L}_{\eta, k}=\lambda u\right\}
\end{equation}
the eigenspace for the eigenvalue $\lambda$ associated to the operator $\mathcal{L}_{\eta, k}$ and by
\begin{equation}
\Lambda_{\eta,k} \coloneqq\big\{ \lambda\in \R \ ; \ \mathcal E_{\eta,k}(\lambda) \setminus\{0\}  \neq \emptyset \big\}
\end{equation}
its spectrum.
Recall the definition of $V_{u_k}$ in \eqref{0.2-a} and consider also the larger space
\begin{equation}
U_{u_k} = \left\{ w \in L^{2}_{\omega_{\eta,k}}(\Sigma_k; \R^{m}) \forwhich w(x) \in T_{u_k(x)}\mathcal N, \quad \text{for a.e. } x\in\Sigma_k \right\}.
\end{equation}

The following spectral decomposition was proven in \cite{DGR25}.

\begin{lemma}[Spectral Decomposition]
\label{LEMMA: Spectral decomp of L}
There exists a Hilbert basis of the space $(U_{u_k}, \langle\cdot,\cdot\rangle_{\omega_{\eta,k}})$ of eigenfunctions of the operator $\mathcal L_{\eta,k}$ and the eigenvalues of $\mathcal L_{\eta,k}$ satisfy 
\begin{equation}
\lambda_1<\lambda_2<\lambda_3\ldots \rightarrow\infty.
\end{equation}
Furthermore, one has the orthogonal decomposition 
\begin{equation}
U_{u_k} = \bigoplus_{\lambda \in \Lambda_{\eta,k}}\mathcal E_{\eta,k}(\lambda).
\end{equation}
\end{lemma}

The following lemma was also proven in Lemma IV.3 in \cite{DGR25}:
\begin{lemma}[Sylvester Law of Inertia]
\label{LEMMA: Ind plus Null eq dim of eigenspaces}
\begin{equation}
\operatorname{Ind}(u_k) + \operatorname{Null}(u_k) = \operatorname{dim}\left( \bigoplus_{\lambda\le 0} \mathcal E_{\eta,k}(\lambda) \right)
\end{equation}
\end{lemma}

Set
\begin{equation}
\mu_{\eta,k}\coloneqq \norm{\frac{\abs{\nabla u_k}^2}{\omega_{\eta,k}}}_{L^\infty(\Sigma_k)}.
\end{equation}
Then one has

\begin{lemma}
\label{LEMMA: prop on the seq mu eta k}
\begin{equation}
\begin{aligned}
&\hspace{-6mm} \exists \eta_0>0: \ \exists k_0>0: \ \exists C>0: \\
&i)\ \mu_0\coloneqq \sup_{\eta\in(0,\eta_0)} \sup_{k\ge k_0} \mu_{\eta,k}< \infty,  \\
&ii)\ \forall\eta \in(0,\eta_0): \forall k \ge k_0: \inf \Lambda_{\eta,k} \ge -C\ \mu_{\eta,k}\ge-C\ \mu_0.
\end{aligned}
\end{equation}
\end{lemma}

\begin{proof}\ \\
\underline{i):} \\
We decompose
\begin{equation}
\mu_{\eta,k} \le \norm{\frac{\abs{\nabla u_k}^2}{\omega_{\eta,k}}}_{L^\infty(\Sigma_k\setminus \mathtt{P}_k^\varrho)} + \norm{\frac{\abs{\nabla u_k}^2}{\omega_{\eta,k}}}_{L^\infty(\mathtt{P}_k^\varrho)}
\end{equation}
Since $\abs{\nabla u_k}^2 = |\nabla u_k^\sharp|^2|x|^2$ we have with \eqref{EQ: jUINDUQWn2u323fnu9inufnujf3UJLCNS91edq} that
\begin{equation}
\limsup_{\eta \searrow0} \limsup_{k \rightarrow \infty}\norm{\frac{\abs{\nabla u_k}^2}{\omega_{\eta,k}}}_{L^\infty(\mathtt{P}_k^\varrho)}
= \limsup_{\eta \searrow0} \limsup_{k \rightarrow \infty}\norm{\frac{|\nabla u_k^\sharp|^2}{\omega_{\eta,k}^\sharp}}_{A(\eta,\delta_k)}
\le C \Lambda^2,
\end{equation}
where we also made use of \eqref{EQ: oINijnfj1093jn131dS12d}.
From the definition of $\omega_{\eta,k}$ in \eqref{EQ: UNuh3nfnjidqmk3r9u2mj096529} and from \eqref{EQ: JINhd2njkd21JNDs82fn28u4fhn2f}, it is clear that
\begin{equation}
\limsup_{k \rightarrow \infty}\norm{\frac{\abs{\nabla u_k}^2}{\omega_{\eta,k}}}_{L^\infty(\Sigma_k\setminus\mathtt{P}_k^\varrho)}
\le \norm{\nabla u_\infty}_{L^\infty(\Sigma)}^2.
\end{equation}
This demonstrates i).\\
\underline{ii):} \\
Let $\lambda \in \Lambda_{\eta,k}$.
Then there exists an eigenvector $0\ne w \in V_{u_k}$ of $\mathcal L_{\eta,k}$ corresponding to the eigenvalue $\lambda$, i.e. $\mathcal L_{\eta,k}(w)=\lambda w$.
We get
\begin{equation}
\begin{aligned}
\lambda \langle w,w \rangle_{\omega_{\eta,k}}
&=  \langle \mathcal L_{\eta,k}(w),w \rangle_{\omega_{\eta,k}} =Q_{u_k}(w) \\
&\ge  - C \int_{\Sigma_k} \abs{\nabla u_k}^2 \abs{w}^2 \ dvol_{\Sigma_k}
\end{aligned}
\end{equation}
Hence,
\begin{equation}
\begin{aligned}
\lambda \langle w,w \rangle_{\omega_{\eta,k}} 
&\ge -C \int_{\Sigma_k}   \abs{\nabla u_k}^2 \abs{w}^2 \ dvol_{\Sigma_k} \\
&\ge -C\ \mu_{\eta,k} \int_{\Sigma_k} \abs{w}^2 \omega_{\eta,k} \ dvol_{\Sigma_k} \\
&= - C\ \mu_{\eta,k} \langle w,w \rangle_{\omega_{\eta,k}}.
\end{aligned}
\end{equation}
This completes the proof of the lemma.
\end{proof}
In the following we focus on the limiting map $u_\infty:\Sigma\rightarrow\mathcal N$ as in \eqref{EQ: JINhd2njkd21JNDs82fn28u4fhn2f} and on the limiting geodesic $\widetilde \gamma_\infty: (0,1)\to\mathcal N$ in the collars as in Proposition \ref{congeo}
We proceed analogous to \cite{DGR25}.
We compute for $w\in V_{u_\infty}$ integrating by parts
\begin{equation}
\begin{aligned}
Q_{u_\infty}(w) 
= \int_{\Sigma} \left( -\Delta w - S_{u_{\infty}}(\nabla u_{\infty})w \right) \cdot w\ dvol_\Sigma.
\end{aligned}
\end{equation}
Note that for any fixed $\eta>0$ (i.e. any fixed $\varrho>0$) we have the pointwise limit
\begin{equation}
\omega_{\eta,k}(q) \rightarrow \omega_{\eta,\infty}(q) \coloneqq \begin{cases}
1, &\text{ if } q\in \Sigma \setminus \mathtt P_\infty^\varrho \\
\frac{\exp(-\beta\ dist(q,\Sigma\setminus \mathtt P_\infty^\varrho))}{\eta^\beta}, &\text{ if } q\in \mathtt P_\infty^\varrho
\end{cases}, \qquad \text{ as } k \rightarrow \infty,
\end{equation}
where here $\mathtt P_\infty^\varrho$ is obtained by pulling back the domain $\mathtt P_k^\varrho$ under $\tau_k$ to $\Sigma$ and sending $k\to \infty$ (Here we are using as usual the notations as in \eqref{EQ:DM-convergence}).
We would also like to remark that for $\eta>0$ small (i.e. $\varrho>0$ small) the domain $\mathtt P_\infty^\varrho$ naturally decomposes in to 2 connected components which are divided by the puncture (collapsed geodesic). 
We introduce 
\begin{equation}
\mathcal L_{\eta,\infty}: V_{u_{\infty}} \rightarrow V_{u_{\infty}};
\quad\mathcal L_{\eta,\infty}(w) \coloneqq \omega_{\eta,\infty}^{-1} \ P_{u_\infty} \left(- \Delta w  - S_{u_{\infty}}(\nabla u_{\infty})w  \right), 
\end{equation}
such that
\begin{equation}
Q_{u_\infty}(w)= \langle \mathcal L_{\eta,\infty} w, w \rangle_{\omega_{\eta,\infty}},
\end{equation}
where we used 
\begin{equation}
\langle w,v \rangle_{\omega_{\eta,\infty}} \coloneqq \int_{\Sigma} w \cdot v \ \omega_{\eta,\infty} \ dvol_{\Sigma}.
\end{equation}
As above a simple integration by parts shows that for any $w\in V_{\widetilde\gamma_\infty}$
\begin{equation}
\begin{aligned}
Q_{\widetilde\gamma_\infty}(w)
&= \int_0^1 \left(- \partial_t^2 w  - S_{\widetilde\gamma_\infty}(\partial_t \widetilde\gamma_\infty)w \right) \cdot w  \ dt.
\end{aligned}
\end{equation}
Define 
\begin{equation}
\widetilde \omega_{\infty}(t) \coloneqq 1, \qquad t\in[0,1].
\end{equation}
To $\sigma\in(0,1/2)$ and $\eta>0$ one has for any $t\in(\sigma,1-\sigma)$ the pointwise limit
\begin{equation}
T_{\eta,k}^2\ \widetilde \omega_{\eta,k}(t) \to \widetilde \omega_{\infty} (t),
\qquad \text{ as } k \rightarrow \infty.
\end{equation}
We introduce 
\begin{equation}
\label{EQ: JNUn2u3fn239fn3f2n}
\widetilde{\mathcal L}_{\infty}: V_{\widetilde\gamma_\infty} \rightarrow V_{\widetilde\gamma_\infty}; 
\qquad\widetilde{\mathcal L}_{\infty}(w)
\coloneqq \widetilde \omega_{\infty}^{-1}\ P_{\widetilde\gamma_\infty} \left(- \partial_s^2 w  -S_{\widetilde\gamma_\infty}(\partial_s \widetilde\gamma_\infty)w \right),
\end{equation}
such that 
\begin{equation}
Q_{\widetilde\gamma_\infty}(w)= \langle \widetilde{\mathcal L}_{\infty} w, w \rangle_{\widetilde\omega_{\infty}},
\end{equation}
where we used 
\begin{equation}
\langle w,v \rangle_{\widetilde\omega_{\infty}} \coloneqq \int_{0}^1 w \cdot v \ \widetilde\omega_{\infty} \ dt.
\end{equation}
Let
\begin{equation}
U_{u_\infty} = \left\{ w \in L^{2}_{\omega_{\eta,\infty}}(\Sigma; \R^{m}) \forwhich w(x) \in T_{u_\infty(x)}\mathcal N, \quad \text{for a.e. } x\in\Sigma \right\},
\end{equation}
and 
\begin{equation}
U_{\widetilde\gamma_\infty} = \left\{ w \in L^{2}_{\widetilde\omega_{\infty}}((0,1); \R^{m}) \forwhich w(0)=w(1)=0,\ w(x) \in T_{\widetilde\gamma_\infty}\mathcal N, \quad \text{for a.e. } x\in (0,1) \right\}.
\end{equation}
In Lemma IV.5 of \cite{DGR25} the following result was shown:

\begin{lemma}
\begin{enumerate}
\item The separable Hilbert space $(U_{u_\infty}, \langle\cdot,\cdot\rangle_{\omega_{\eta,\infty}})$ has a Hilbert basis consisting of eigenfunctions of $\mathcal L_{\eta,\infty}$.
\item The separable Hilbert space $(U_{\widetilde\gamma_\infty}, \langle\cdot,\cdot\rangle_{\widetilde \omega_{\infty}})$ has a Hilbert basis consisting of eigenfunctions of $\widetilde{\mathcal L}_{\infty}$.
\end{enumerate}
\end{lemma}

\noindent
We continue by introducing the limiting eigenspaces 
\begin{equation}
\mathcal E_{\eta,\infty}(\lambda) \coloneqq \{ w \in V_{u_\infty} \ ; \ \mathcal L_{\eta,\infty}(w)=\lambda w \},
\qquad
\widetilde{\mathcal E}_{\infty}(\lambda) \coloneqq \{ w \in V_{\widetilde\gamma_\infty} \ ; \ \widetilde{\mathcal L}_{\infty}(w)=\lambda w \}.
\end{equation}
And their nonpositive contribution
\begin{equation}
\mathcal E_{\eta,\infty}^0\coloneqq \bigoplus_{\lambda \le 0} \mathcal E_{\eta,\infty}(\lambda), \qquad
\widetilde{\mathcal E}_{\infty}^0\coloneqq \bigoplus_{\lambda \le 0} \widetilde{\mathcal E}_{\infty}(\lambda).
\end{equation}
In \cite{DGR25} in (IV.12) the following result was shown:
\begin{lemma}
\label{LEMMA: Ind plus Null eq dim of eig for infty lim func}
\begin{equation}
\begin{aligned}
&i)\ \dim\left(\mathcal E_{\eta,\infty}^0\right) 
\le \operatorname{Ind}(u_\infty) + \operatorname{Null}(u_\infty), \\
&ii)\ \dim\left(\widetilde{\mathcal E}_{\infty}^0\right) 
\le \operatorname{Ind}(\widetilde\gamma_\infty) + \operatorname{Null}(\widetilde\gamma_\infty),
\end{aligned}
\end{equation}
\end{lemma}
We consider the unit sphere (finite dimensional as the ambient space is finite dimensional) given by
\begin{equation}
\mathcal{S}_{\eta, k}^0:=\left\{w \in \bigoplus_{\lambda \leq 0} \mathcal{E}_{\eta, k}(\lambda) ;\langle w, w\rangle_{\omega_{\eta, k}}=1\right\} .
\end{equation}

\begin{lemma} \label{LEMMA: On the non zero of the lim functs}
For any $k \in \mathbb{N}$ let $w_k \in \mathcal{S}_{\eta, k}^0$. Then there exist a subsequence such that
\begin{equation}
\begin{gathered}
w_k \rightharpoonup w_{\infty}, \text { weakly in } W^{1,2}(\Sigma) \cap W_{\text {loc }}^{2,2}(\Sigma_k \backslash\{ \gamma_n \}) \\
\overline w_k \left(s,\theta\right) \coloneqq \frac{1}{\sqrt{T_{\eta,k}}}\ w_k \circ \Phi^{-1} (r(s),\theta) \rightharpoonup v_{\infty}(s), \text { weakly in } W_{\text {loc }}^{2,2}((0,1)\times S^1)
\end{gathered}
\end{equation}
and
\begin{equation}
\text { either } w_{\infty} \neq 0, \quad \text { or } v_{\infty} \neq 0 ,
\end{equation}
where $W_{\text {loc }}^{2,2}$-convergence is of course to be understood after pulling back the metric as in \eqref{EQ:DM-convergence},
namely $w_k\circ \tau_k^{-1} $ converges locally to $w_{\infty}$ in $\tilde\Sigma\setminus\{(\xi_1,\xi_2)\}$.
\end{lemma}

\begin{proof}
We are going to show the successively the following list of claims:

\begin{enumerate}
\item $\norm{\overline w_k}_{L^2((0,1)\times S^1)}^2 \le 1, \qquad \norm{\partial_s\overline w_k}_{L^2((0,1)\times S^1)}^2 \le \mu_0 \qquad \norm{\partial_\theta\overline w_k}_{L^2((0,1)\times S^1)}^2 \le \frac{\mu_0}{T_{\eta,k}^2}$
\item $\norm{\partial_s^2\overline w_k + T_{\eta,k}^2 \partial_\theta^2 \overline w_k}_{L^2((\sigma,1-\sigma)\times S^1)} \le C_{\sigma,\eta}$
\item $\|\partial_s^2\overline w_k\|_{L^2((\sigma,1-\sigma)\times S^1)}
+\|T_{\eta,k}\partial_s\partial_\theta\overline w_k\|_{L^2((\sigma,1-\sigma)\times S^1)}
+\|T_{\eta,k}^2\partial_\theta^2\overline w_k\|_{L^2((\sigma,1-\sigma)\times S^1)}
\le C_{\sigma,\eta}$
\item $\norm{\partial_\theta \overline w_k}_{L^2((\sigma,1-\sigma)\times S^1)} \le \frac{C_{\sigma,\eta}}{T_{\eta,k}^2} $
\item The Lemma itself.
\end{enumerate}
{\underline {\bf (i):}}
Using $ds=\frac{dr}{rT_{\eta,k}}$
\begin{equation}
\begin{aligned}
\norm{\overline w_k}_{L^2((0,1)\times S^1)}^2
&= \int_{S^1} \int_0^1 |\overline w_k|^2 ds \ d\theta \\
&= \int_{S^1} \int_{\delta_k/\eta}^\eta |w_k\circ \Phi^{-1}|^2 \frac{dr}{rT_{\eta,k}^2} \ d\theta \\
&=  \int_{A(\eta,\delta_k)} |w_k\circ \Phi^{-1}|^2 \frac{1}{T_{\eta,k}^2|x|^2} dx \\
&=  \int_{\mathtt P_k^\varrho} |w_k|^2 \frac{1}{T_{\eta,k}^2}\ dvol_{h_k} \\
&\le  \int_{\mathtt P_k^\varrho} |w_k|^2 \omega_{\eta,k} \ dvol_{h_k} \\
&\le  \int_{\Sigma_k} |w_k|^2 \omega_{\eta,k} \ dvol_{h_k}=1.
\end{aligned}
\end{equation}
Using $\partial_s \overline w_k = \frac{1}{\sqrt{T_{\eta,k}}} \partial_r (w_k\circ \Phi^{-1}) \ T_{\eta,k} r$ and $\partial_r (w_k\circ \Phi^{-1}) = \partial_t w_k/r$
\begin{equation}
\begin{aligned}
\norm{\partial_s\overline w_k}_{L^2((0,1)\times S^1)}^2
&= \int_{S^1} \int_0^1 |\partial_s\overline w_k|^2 ds \ d\theta \\
&= \int_{S^1} \int_{\delta_k/\eta}^\eta |\partial_r(w_k\circ \Phi^{-1})|^2\ r \ dr \ d\theta \\
&=  \int_{A(\eta,\delta_k)} |\partial_r(w_k\circ \Phi^{-1})|^2  dx \\
&= \int_{\mathtt P_k^{\varrho}} |\partial_t w_k|^2 dvol_{h_k} \\
&\le \int_{\mathtt P_k^{\varrho}} |\nabla w_k|^2 dvol_{h_k} \\
&\le \int_{\Sigma_k} |\nabla w_k|^2 dvol_{h_k} \\
&= \underbrace{Q_{u_k}(w_k)}_{\le 0} + \int_{\Sigma_k}\mathbb A_{u_k}(\nabla u_k,\nabla u_k)\mathbb A_{u_k}(w_k,w_k)\ dvol_{h_k} \\
&\le \int_{\Sigma_k} |w_k|^2\ \omega_{\eta,k}\ dvol_{h_k} \norm{\frac{\abs{\nabla u_k}^2}{\omega_{\eta,k}}}_{L^\infty(\Sigma_k)} \\
&\le \mu_0.
\end{aligned}
\end{equation}
Similar,
\begin{equation}
\begin{aligned}
\norm{\partial_\theta\overline w_k}_{L^2((0,1)\times S^1)}^2
&= \int_{S^1} \int_0^1 |\partial_\theta\overline w_k|^2 ds \ d\theta \\
&= \int_{S^1} \int_{\delta_k/\eta}^\eta |\partial_\theta(w_k\circ \Phi^{-1})|^2\ \frac{dr}{rT_{\eta,k}^2} \ d\theta \\
&=  \int_{A(\eta,\delta_k)} |\partial_\theta(w_k\circ \Phi^{-1})|^2 \frac{1}{T_{\eta,k}^2|x|^2} dx \\
&= \int_{\mathtt P_k^{\varrho}} |\partial_\theta w_k|^2 \frac{1}{T_{\eta,k}^2} dvol_{h_k} \\
&\le \frac{1}{T_{\eta,k}^2} \int_{\mathtt P_k^{\varrho}} |\nabla w_k|^2 dvol_{h_k} \\
&\le \frac{1}{T_{\eta,k}^2} \int_{\Sigma} |\nabla w_k|^2 dvol_{h_k} \\
&= \frac{1}{T_{\eta,k}^2} \left( \underbrace{Q_{u_k}(w_k)}_{\le 0} + \int_{\Sigma}\mathbb A_{u_k}(\nabla u_k,\nabla u_k)\mathbb A_{u_k}(w_k,w_k)\ dvol_{h_k} \right) \\
&\le \frac{1}{T_{\eta,k}^2}\int_{\Sigma} |w_k|^2\ \omega_{\eta,k}\ dvol_{h_k} \norm{\frac{\abs{\nabla u_k}^2}{\omega_{\eta,k}}}_{L^\infty(\Sigma)} \\
&\le \frac{\mu_0}{T_{\eta,k}^2}.
\end{aligned}
\end{equation}
\\ {\underline {\bf (ii):}}
To $k\in \N$ let here $(\phi_{k}^j)_{j=1,\dots,N_k}$ be an orthonormal basis of $\oplus_{\lambda \leq 0} \mathcal{E}_{\eta, k}(\lambda)$.
Then we can find some coefficients $c_k^1,\dots,c_k^{N_k}$ with 
\begin{equation}
w_k=\sum_{j=1}^{N_k} c_k^j \phi_k^j, \qquad \sum_{j=1}^{N_k}(c_k^j)^2=1.
\end{equation}
Letting as usual $\overline \phi_k^j= \widetilde \phi_k^j/\sqrt{{T_{\eta,k}}}$ one has
\begin{equation}
- \ \widetilde \omega_{\eta,k}^{-1} \ P_{\widetilde u_k} \left[
\frac{1}{T_{\eta,k}^2}\partial_s^2\overline w_k
+\partial_\theta^2 \overline w_k
+ S_{\widetilde u_k}\left(\frac{\partial_s \widetilde u_k}{T_{\eta,k}},\partial_\theta\widetilde u_k \right) \overline w_k
\right] = \sum_{j=1}^{N_k} c_k^j \lambda_k^j\ \overline \phi_k^j,
\end{equation}
where the operator $S$ is as in \eqref{EQ: iwenf8932injIUJN128ee12}.
Bound
\begin{equation}
\label{EQ: HUBZNu82hfnuj2fuj2df22dulkz6}
\begin{aligned}
& \norm{\partial_s^2\overline w_k + T_{\eta,k}^2 \partial_\theta^2 \overline w_k}_{L^2((\sigma,1-\sigma)\times S^1)}\\
& \le \norm{P_{\widetilde u_k} \left[\partial_s^2\overline w_k + T_{\eta,k}^2 \partial_\theta^2 \overline w_k\right]}_{L^2((\sigma,1-\sigma)\times S^1)}
+ \norm{(id-P_{\widetilde u_k}) \left[\partial_s^2\overline w_k + T_{\eta,k}^2 \partial_\theta^2 \overline w_k\right]}_{L^2((\sigma,1-\sigma)\times S^1)}
\end{aligned}
\end{equation}
First,
\begin{equation}
\label{EQ: JUNIUnu23nfd23nsjd2jnuuuujgr3v}
\begin{aligned}
&\norm{P_{\widetilde u_k} \left[\partial_s^2\overline w_k + T_{\eta,k}^2 \partial_\theta^2 \overline w_k\right]}_{L^2((\sigma,1-\sigma)\times S^1)} \\
& \le \norm{\widetilde \omega_{\eta,k} T_{\eta,k}^2 \sum_{j=1}^{N_k} c_k^j \lambda_k^j\ \overline \phi_k^j}_{L^2((\sigma,1-\sigma)\times S^1)}
+ \norm{S_{\widetilde u_k}\left(\partial_s \widetilde u_k,T_{\eta,k}\partial_\theta\widetilde u_k \right)  \overline w_k}_{L^2((\sigma,1-\sigma)\times S^1)} \\
& \le \norm{\widetilde \omega_{\eta,k} T_{\eta,k}^2}_{L^\infty((\sigma,1-\sigma)\times S^1)} \norm{\sum_{j=1}^{N_k} c_k^j \lambda_k^j\ \overline \phi_k^j}_{L^2((\sigma,1-\sigma)\times S^1)}\\
& \hspace{20mm}+ C\norm{\lvert \partial_s \tilde u_k\rvert^2+T_{\eta,k}^2 \lvert \partial_\theta \tilde u_k\rvert^2}_{L^\infty((\sigma,1-\sigma)\times S^1)} \norm{\overline w_k}_{L^2((\sigma,1-\sigma)\times S^1)} \\
& \le C_{\sigma,\eta},
\end{aligned}
\end{equation}
where we also used
\begin{equation}
\begin{aligned}
\norm{\sum_{j=1}^{N_k} c_k^j \lambda_k^j\ \overline \phi_k^j}_{L^2((\sigma,1-\sigma)\times S^1)}^2
& \le \frac{1}{T_{\eta,k}^2} \norm{\sum_{j=1}^{N_k} c_k^j \lambda_k^j\ \phi_k^j}_{L^2(\Sigma_k)}^2 \\
& \le \norm{\sum_{j=1}^{N_k} c_k^j \lambda_k^j\ \phi_k^j}_{L^2_{\omega_{\eta,k}}(\Sigma_k)}^2 \\
&= \sum_{j=1}^{N_k} (c_k^j\lambda_k^j)^2 \\
&\le (\inf \Lambda_{\eta,k})^2 \le C.
\end{aligned}
\end{equation}
Second, since $w_k\in V_{u_k}$ we have
\begin{equation}
P_{\widetilde u_k}\overline w_k=\overline w_k
\qquad \text{for all } (s,\theta)\in [0,1]\times S^1.
\end{equation}
Differentiate this identity in the $s$- and $\theta$-directions. For $\alpha\in\{s,\theta\}$ this gives
\begin{equation}
dP_{\widetilde u_k}[\partial_\alpha\widetilde u_k]\,\overline w_k
+P_{\widetilde u_k}\partial_\alpha\overline w_k
=\partial_\alpha\overline w_k,
\end{equation}
hence
\begin{equation}\label{eq:first-normal-general}
(id-P_{\widetilde u_k})\partial_\alpha\overline w_k
=
dP_{\widetilde u_k}[\partial_\alpha\widetilde u_k]\,\overline w_k.
\end{equation}
Differentiating once more, we obtain
\begin{equation}
\begin{aligned}
(id-P_{\widetilde u_k})\partial_\alpha^2\overline w_k
&=
d^2P_{\widetilde u_k}[\partial_\alpha\widetilde u_k,\partial_\alpha\widetilde u_k]\,\overline w_k
+dP_{\widetilde u_k}[\partial_\alpha^2\widetilde u_k]\,\overline w_k
+2\,dP_{\widetilde u_k}[\partial_\alpha\widetilde u_k]\,\partial_\alpha\overline w_k .
\end{aligned}
\end{equation}
Therefore
\begin{equation}
\begin{aligned}
&(id-P_{\widetilde u_k})\bigl[\partial_s^2\overline w_k+T_{\eta,k}^2\partial_\theta^2\overline w_k\bigr] \\
&=
d^2P_{\widetilde u_k}[\partial_s\widetilde u_k,\partial_s\widetilde u_k]\,\overline w_k
+T_{\eta,k}^2\,d^2P_{\widetilde u_k}[\partial_\theta\widetilde u_k,\partial_\theta\widetilde u_k]\,\overline w_k \\
&\quad
+dP_{\widetilde u_k}[\partial_s^2\widetilde u_k+T_{\eta,k}^2\partial_\theta^2\widetilde u_k]\,\overline w_k \\
&\quad
+2\,dP_{\widetilde u_k}[\partial_s\widetilde u_k]\,\partial_s\overline w_k
+2\,T_{\eta,k}^2\,dP_{\widetilde u_k}[\partial_\theta\widetilde u_k]\,\partial_\theta\overline w_k .
\end{aligned}
\end{equation}
Since $\mathcal N$ is closed and smoothly embedded, the derivatives of the orthogonal projection are uniformly bounded on a tubular neighborhood of $\mathcal N$. Hence
\begin{equation}
\begin{aligned}
&\left|(id-P_{\widetilde u_k})\bigl[\partial_s^2\overline w_k+T_{\eta,k}^2\partial_\theta^2\overline w_k\bigr]\right| \\
&\le
C\Big(
|\partial_s\widetilde u_k|^2
+T_{\eta,k}^2|\partial_\theta\widetilde u_k|^2
+|\partial_s^2\widetilde u_k+T_{\eta,k}^2\partial_\theta^2\widetilde u_k|
\Big)|\overline w_k| \\
&\qquad
+ C\,|\partial_s\widetilde u_k|\,|\partial_s\overline w_k|
+ C\,|T_{\eta,k}\partial_\theta\widetilde u_k|\,|T_{\eta,k}\partial_\theta\overline w_k|.
\end{aligned}
\end{equation}
Now $\widetilde u_k$ satisfies
\begin{equation}
-\partial_s^2 \widetilde u_k - T_{\eta,k}^2 \,\partial_\theta^2 \widetilde u_k
=\mathbb A_{\widetilde u_k}\bigl(\partial_s \widetilde u_k,\partial_s \widetilde u_k\bigr)
+ T_{\eta,k}^2\,\mathbb A_{\widetilde u_k}\bigl(\partial_\theta \widetilde u_k,\partial_\theta \widetilde u_k\bigr),
\end{equation}
and therefore
\begin{equation}
|\partial_s^2\widetilde u_k+T_{\eta,k}^2\partial_\theta^2\widetilde u_k|
\le
C\bigl(|\partial_s\widetilde u_k|^2+T_{\eta,k}^2|\partial_\theta\widetilde u_k|^2\bigr).
\end{equation}
Hence
\begin{equation}
\begin{aligned}
&\left|(id-P_{\widetilde u_k})\bigl[\partial_s^2\overline w_k+T_{\eta,k}^2\partial_\theta^2\overline w_k\bigr]\right| \\
&\le
C\bigl(|\partial_s\widetilde u_k|^2+T_{\eta,k}^2|\partial_\theta\widetilde u_k|^2\bigr)|\overline w_k|
+ C\,|\partial_s\widetilde u_k|\,|\partial_s\overline w_k|
+ C\,|T_{\eta,k}\partial_\theta\widetilde u_k|\,|T_{\eta,k}\partial_\theta\overline w_k|.
\end{aligned}
\end{equation}
Taking the $L^2((\sigma,1-\sigma)\times S^1)$-norm and using Proposition \ref{estdertheta}, Proposition \ref{PROP: BD on dell s tilde u}, together with the bounds from (i), we get
\begin{equation}
\label{EQ: NBHIUZb2nuzfn2un23f12sdHNZde}
\begin{aligned}
&\left\|(id-P_{\widetilde u_k})\bigl[\partial_s^2\overline w_k+T_{\eta,k}^2\partial_\theta^2\overline w_k\bigr]\right\|_{L^2((\sigma,1-\sigma)\times S^1)} \\
&\le
C\|\ |\partial_s\widetilde u_k|^2+T_{\eta,k}^2|\partial_\theta\widetilde u_k|^2\ \|_{L^\infty((\sigma,1-\sigma)\times S^1)}
\|\overline w_k\|_{L^2((0,1)\times S^1)} \\
&\qquad
+ C\|\partial_s\widetilde u_k\|_{L^\infty((\sigma,1-\sigma)\times S^1)}
\|\partial_s\overline w_k\|_{L^2((0,1)\times S^1)} \\
&\qquad
+ C\|T_{\eta,k}\partial_\theta\widetilde u_k\|_{L^\infty((\sigma,1-\sigma)\times S^1)}
\|T_{\eta,k}\partial_\theta\overline w_k\|_{L^2((0,1)\times S^1)} \\
&\le C_{\sigma,\eta}.
\end{aligned}
\end{equation}
Combining \eqref{EQ: HUBZNu82hfnuj2fuj2df22dulkz6}, \eqref{EQ: JUNIUnu23nfd23nsjd2jnuuuujgr3v} and \eqref{EQ: NBHIUZb2nuzfn2un23f12sdHNZde} we conclude that
\begin{equation}
\norm{\partial_s^2\overline w_k + T_{\eta,k}^2 \partial_\theta^2 \overline w_k}_{L^2((\sigma,1-\sigma)\times S^1)} \le C_{\sigma,\eta}.
\end{equation}
\\ {\underline {\bf (iii):}}
Let now $v_k=\zeta\,\overline w_k$, where $\zeta \in C^{\infty}_c(0,1)$ and
\begin{equation}
\zeta|_{[\sigma,1-\sigma]}=1, \qquad \supp\zeta \subset \subset \left(\frac{\sigma}{2},1-\frac{\sigma}{2}\right), \qquad 0\le \zeta \le 1, \qquad |\zeta^\prime|\le \frac{4}{\sigma}.
\end{equation}
We have
\begin{equation}
(\partial_s^2+T_{\eta,k}^2\partial_\theta^2)v_k
=
\zeta(\partial_s^2+T_{\eta,k}^2\partial_\theta^2)\overline w_k
+2\zeta'\partial_s\overline w_k+\zeta''\overline w_k.
\end{equation}
Therefore,
\begin{equation}
\begin{aligned}
\|(\partial_s^2+T_{\eta,k}^2\partial_\theta^2)v_k\|_{L^2((0,1)\times S^1)}
&\le
\|(\partial_s^2+T_{\eta,k}^2\partial_\theta^2)\overline w_k\|_{L^2((\sigma/2,1-\sigma/2)\times S^1)} \\
&\quad +2\|\zeta'\|_{L^\infty}\|\partial_s\overline w_k\|_{L^2((0,1)\times S^1)}
+\|\zeta''\|_{L^\infty}\|\overline w_k\|_{L^2((0,1)\times S^1)}.
\end{aligned}
\end{equation}
The first term is bounded by the previous estimate on
$\partial_s^2\overline w_k+T_{\eta,k}^2\partial_\theta^2\overline w_k$ on the slightly larger strip
$(\sigma/2,1-\sigma/2)\times S^1$, while the last two terms are bounded by the uniform
$L^2$-bounds on $\partial_s\overline w_k$ and $\overline w_k$. Hence
\begin{equation}
\|(\partial_s^2+T_{\eta,k}^2\partial_\theta^2)v_k\|_{L^2((0,1)\times S^1)}
\le C_\sigma.
\end{equation}
Since $v_k$ is compactly supported in $s$ and periodic in $\theta$, integration by parts gives
\begin{equation}
\begin{aligned}
\|(\partial_s^2+T_{\eta,k}^2\partial_\theta^2)v_k\|_{L^2}^2
&=
\|\partial_s^2 v_k\|_{L^2}^2
+T_{\eta,k}^4\|\partial_\theta^2 v_k\|_{L^2}^2
+2T_{\eta,k}^2\int (\partial_s^2 v_k)\cdot(\partial_\theta^2 v_k) \\
&=
\|\partial_s^2 v_k\|_{L^2}^2
+2T_{\eta,k}^2\|\partial_s\partial_\theta v_k\|_{L^2}^2
+T_{\eta,k}^4\|\partial_\theta^2 v_k\|_{L^2}^2.
\end{aligned}
\end{equation}
Hence
\begin{equation}
\|\partial_s^2 v_k\|_{L^2}
+\|T_{\eta,k}\partial_s\partial_\theta v_k\|_{L^2}
+\|T_{\eta,k}^2\partial_\theta^2 v_k\|_{L^2}
\le C_\sigma.
\end{equation}
Since $\zeta\equiv 1$ on $(\sigma,1-\sigma)$, this yields
\begin{equation}
\|\partial_s^2\overline w_k\|_{L^2((\sigma,1-\sigma)\times S^1)}
+\|T_{\eta,k}\partial_s\partial_\theta\overline w_k\|_{L^2((\sigma,1-\sigma)\times S^1)}
+\|T_{\eta,k}^2\partial_\theta^2\overline w_k\|_{L^2((\sigma,1-\sigma)\times S^1)}
\le C_\sigma.
\end{equation}
\\ {\underline {\bf (iv):}}
Now
\begin{equation}
\begin{aligned}
\norm{\partial_\theta \overline w_k}_{L^2((\sigma,1-\sigma)\times S^1)}^2
&= \int_{(\sigma,1-\sigma)\times S^1} \partial_\theta \overline w_k \cdot \partial_\theta \overline w_k \ d\theta ds \\
&= \int_{(\sigma,1-\sigma)\times S^1} \partial_\theta \overline w_k \cdot \partial_\theta \left(\overline w_k- \frac{1}{2\pi} \int_0^{2\pi}\overline w_k(s,\theta^\prime) d\theta^\prime \right) \ d\theta ds \\
&= \int_{(\sigma,1-\sigma)\times S^1} -\partial_\theta^2 \overline w_k \cdot  \left(\overline w_k- \frac{1}{2\pi} \int_0^{2\pi}\overline w_k(s,\theta^\prime) d\theta^\prime \right) \ d\theta ds \\
&\le \norm{\partial_\theta^2 \overline w_k}_{L^2((\sigma,1-\sigma)\times S^1)} \norm{\overline w_k- \frac{1}{2\pi} \int_0^{2\pi}\overline w_k(s,\theta^\prime) d\theta^\prime}_{L^2((\sigma,1-\sigma)\times S^1)} \\
&\le \frac{C}{T_{\eta,k}^2} \norm{\overline w_k- \frac{1}{2\pi} \int_0^{2\pi}\overline w_k(s,\theta^\prime) d\theta^\prime}_{L^2((\sigma,1-\sigma)\times S^1)}
\end{aligned}
\end{equation}
With Poincar\'e's inequality on the unit circle
\begin{equation}
\begin{aligned}
\left\| \overline w_k-\frac{1}{2\pi} \int_0^{2\pi}\overline w_k(s,\theta^\prime) d\theta^\prime \right\|_{L^2((\sigma,1-\sigma)\times S^1)}
&= \left(\int_{\sigma}^{1-\sigma} \left\|\overline w_k-\frac{1}{2\pi} \int_0^{2\pi}\overline w_k(s,\theta^\prime) d\theta^\prime \right\|_{L^2( S^1)}^2 ds \right)^{1/2} \\
&\le \left( \int_{\sigma}^{1-\sigma} \left\| \partial_\theta \overline w_k \right\|_{L^2( S^1)}^2 ds \right)^{1/2} \\
&= \left\| \partial_\theta \overline w_k \right\|_{L^2((\sigma,1-\sigma)\times S^1)}
\end{aligned}
\end{equation}
By absorbing $\left\| \partial_\theta \overline w_k \right\|_{L^2((\sigma,1-\sigma)\times S^1)}$ we get
\begin{equation}
\norm{\partial_\theta \overline w_k}_{L^2((\sigma,1-\sigma)\times S^1)} \le \frac{C}{T_{\eta,k}^2}
\end{equation}
\\ {\underline {\bf (v):}}
Up to extraction we already know that
\begin{equation}
w_k \rightharpoonup w_\infty
\quad\text{weakly in }W^{1,2}(\Sigma_k)\cap W^{2,2}_{\mathrm{loc}}(\Sigma_k\setminus\{\gamma_k\}),
\end{equation}
and that
\begin{equation}
\overline w_k(s,\theta):=\frac1{\sqrt{T_{\eta,k}}}\,
w_k\circ\Phi^{-1}(r(s),\theta)
\rightharpoonup v_\infty
\quad\text{weakly in }W^{2,2}_{\mathrm{loc}}((0,1)\times S^1).
\end{equation}
It remains to prove that
\begin{equation}
\text{either } w_\infty\neq 0,\qquad\text{or }v_\infty\neq 0.
\end{equation}
To derive a contradiction assume
\begin{equation}
\text{either } w_\infty= 0,\qquad\text{and }v_\infty= 0.
\end{equation}
Let $\chi \in C^\infty_c(0,1)$ be a cutoff function with $0\le \chi \le 1$,
\begin{equation}
\chi|_{\big[\frac{\log(2)}{T_{\eta,k}},\frac{\sigma}{2}\big]\cup\big[1-\frac{\sigma}{2},1-\frac{\log(2)}{T_{\eta,k}}\big]}=1, \qquad \supp \chi \subset \subset \big[0,\sigma\big]\cup\big[1-\sigma,1\big]
\end{equation}
and also
\begin{equation}
\bigg|\chi^\prime|_{\big[0,\frac{\log(2)}{T_{\eta,k}}\big] \cup \big[1-\frac{\log(2)}{T_{\eta,k}},1\big]}\bigg| \le 4 T_{\eta,k}, \qquad
\bigg|\chi^\prime|_{\big[\frac{\sigma}{2},\sigma \big] \cup \big[1-\sigma,1-\frac{\sigma}{2}\big]}\bigg| \le \frac{4}{\sigma}.
\end{equation}
We introduce the variation
\begin{equation}
\xi_k(q) \coloneqq
\begin{cases}
\chi(s(t))\,w_k(t,\theta), & q=(t,\theta)\in \mathtt P_k^\varrho,\\
0, & q\in \Sigma_k\setminus \mathtt P_k^\varrho.
\end{cases}
\end{equation}
and the according notation
\begin{equation}
\widetilde \xi_k(s,\theta) \coloneqq \xi_k(t(s),\theta), \qquad
\overline \xi_k(s,\theta) \coloneqq \frac{1}{\sqrt{T_{\eta,k}}} \widetilde \xi_k(s,\theta), \qquad
\xi_k^\sharp(x) \coloneqq (\xi_k \circ \Phi^{-1})(x).
\end{equation}
Note that the domain
\begin{equation}
\left[0,\frac{\log(2)}{T_{\eta,k}}\right] \times S^1\cup \left[1-\frac{\log(2)}{T_{\eta,k}},1\right] \times S^1
\subset [0,1] \times S^1
\end{equation}
corresponds in the long cylinder $\mathtt P_k^\varrho$ to
\begin{equation}
\mathtt O \coloneqq 
\left[\frac{\pi}{\varrho},\,\frac{\pi}{\varrho}+\log 2\right] \times S^1
\cup
\left[\frac{2\pi^2}{l_k}-\frac{\pi}{\varrho}-\log 2,\,
\frac{2\pi^2}{l_k}-\frac{\pi}{\varrho}\right] \times S^1
\subset \left[\frac{\pi}{\varrho},\,\frac{2\pi^2}{l_k}-\frac{\pi}{\varrho}\right] \times S^1,
\end{equation}
and in the annulus $A(\eta,\delta_k)$ to
\begin{equation}
A(\eta,\eta/2)\cup A(2\delta_k/\eta,\delta_k/\eta).
\end{equation}
And note that the domain
\begin{equation}
\left(\frac{\sigma}{2},1-\frac{\sigma}{2}\right)\subset [0,1]
\end{equation}
corresponds in the long cylinder $\mathtt P_k^\varrho$ to
\begin{equation}
\mathtt M \coloneqq
\left(\frac{\pi}{\varrho}+\frac{\sigma}{2}T_{\eta,k},\,
\frac{2\pi^2}{l_k}-\frac{\pi}{\varrho}-\frac{\sigma}{2}T_{\eta,k}\right) \times S^1
\subset \left[\frac{\pi}{\varrho},\,\frac{2\pi^2}{l_k}-\frac{\pi}{\varrho}\right] \times S^1,
\end{equation}
and in the annulus $A(\eta,\delta_k)$ to
\begin{equation}
A\!\left(\eta\,\varepsilon_k(\sigma/2),\,
\frac{\delta_k}{\eta\,\varepsilon_k(\sigma/2)}\right),
\qquad
\varepsilon_k(\sigma)=\left(\frac{\delta_k}{\eta^2}\right)^\sigma.
\end{equation}
{\bf Claim 1:}
\begin{equation}
|Q_{u_k}(w_k) - Q_{u_k}(\xi_k)| \to 0
\end{equation}
Proof:
\begin{equation}
\begin{aligned}
&Q_{u_k}(w_k) - Q_{u_k}(\xi_k) \\
& = \int_{\Sigma_k \setminus \mathtt P_k^\varrho} |\nabla w_k|^2 -\mathbb A_{u_k}(\nabla u_k,\nabla u_k) \cdot\mathbb A_{u_k}(w_k,w_k)\\
&\qquad + \int_{\mathtt O} (|\nabla w_k|^2 - |\nabla \xi_k|^2 ) -\mathbb A_{u_k}(\nabla u_k,\nabla u_k)\cdot (\mathbb A_{u_k}(w_k,w_k) -\mathbb A_{u_k}(\xi_k,\xi_k)) \\
& \qquad + \int_{\mathtt M} (|\nabla w_k|^2 - |\nabla \xi_k|^2 ) -\mathbb A_{u_k}(\nabla u_k,\nabla u_k) \cdot (\mathbb A_{u_k}(w_k,w_k) -\mathbb A_{u_k}(\xi_k,\xi_k))
\end{aligned}
\end{equation}
Because of strong convergence of $\nabla u_k$, $w_k$ and $\nabla w_k$

\medskip
\begin{equation}
\abs{\int_{\Sigma_k \setminus \mathtt P_k^\varrho} |\nabla w_k|^2 -\mathbb A_{u_k}(\nabla u_k,\nabla u_k) \cdot\mathbb A_{u_k}(w_k,w_k)}\to 0,
\end{equation}
as $k\to\infty$.
Using the fact that
\begin{equation}
|\nabla \xi_k|\le |\chi^\prime /T_{\eta,k}| \ |w_k| + |\chi|\ |\nabla w_k|\le  |w_k| + |\nabla w_k|, \qquad \text{ in }\mathtt O,
\end{equation}
together with strong convergence of $\nabla u_k$ bound
\begin{equation}
\begin{aligned}
&\abs{\int_{\mathtt O} (|\nabla w_k|^2 - |\nabla \xi_k|^2 ) -\mathbb A_{u_k}(\nabla u_k,\nabla u_k) \cdot (\mathbb A_{u_k}(w_k,w_k) -\mathbb A_{u_k}(\xi_k,\xi_k))} \\
&\le C \int_{\mathtt O} |w_k|^2 + |\nabla w_k|^2 \to 0,
\end{aligned}
\end{equation}
as $k\to\infty$.
Using
\begin{equation}
|\partial_s \overline \xi_k| \le C_\sigma|\overline w_k| + |\partial_s \overline w_k |, \qquad
|\partial_\theta \overline \xi_k| \le \chi |\partial_\theta \overline w_k| \qquad \text{ in }\mathtt M,
\end{equation}
With a change of variables estimate
\begin{equation}
\begin{aligned}
&\abs{\int_{\mathtt M} (|\nabla w_k|^2 - |\nabla \xi_k|^2 ) -\mathbb A_{u_k}(\nabla u_k,\nabla u_k) \cdot (\mathbb A_{u_k}(w_k,w_k) -\mathbb A_{u_k}(\xi_k,\xi_k))} \\
&= \Bigg|\int_{\frac{\sigma}{2}}^{1-\frac{\sigma}{2}} \int_0^{2\pi} (|\partial_s \overline w_k|^2-|\partial_s \overline \xi_k|^2 ) + T_{\eta,k}^2 (|\partial_\theta \overline w_k|^2-|\partial_\theta \overline \xi_k|^2) \\
& \qquad - (\mathbb A_{\widetilde u_k}(\partial_s\widetilde u_k,\partial_s\widetilde u_k) + T_{\eta,k}^2 \mathbb A_{\widetilde u_k}(\partial_\theta\widetilde u_k,\partial_\theta\widetilde u_k) ) \cdot (\mathbb A_{\widetilde u_k}(\overline w_k,\overline w_k) -\mathbb A_{\widetilde u_k}(\overline \xi_k,\overline \xi_k))\Bigg| \\
& \le C \int_{\frac{\sigma}{2}}^{1-\frac{\sigma}{2}} \int_0^{2\pi} |\partial_s \overline w_k|^2 + T_{\eta,k}^2 |\partial_\theta \overline w_k|^2 + (1+|\partial_s \widetilde u_k|^2 + T_{\eta,k}^2 |\partial_\theta \widetilde u_k|^2) |\overline w_k|^2 \\
& \le C \left(\norm{\overline w_k}_{L^2((\frac{\sigma}{2},1-\frac{\sigma}{2})\times S^1)}^2 + \norm{\partial_s\overline w_k}_{L^2((\frac{\sigma}{2},1-\frac{\sigma}{2})\times S^1)}^2 + T_{\eta,k}^2 \norm{\partial_\theta\overline w_k}_{L^2((\frac{\sigma}{2},1-\frac{\sigma}{2})\times S^1)}^2  \right) \\
& \le C \left(\norm{\overline w_k}_{L^2((\frac{\sigma}{2},1-\frac{\sigma}{2})\times S^1)}^2 + \norm{\partial_s\overline w_k}_{L^2((\frac{\sigma}{2},1-\frac{\sigma}{2})\times S^1)}^2 + \frac{1}{T_{\eta,k}^2}  \right) \to 0,
\end{aligned}
\end{equation}
as $k\to\infty$, where we also used the claim (iv).
Combining these results we obtain Claim 1.
\\ {\bf Claim 2:}
\begin{equation}
\int_{\Sigma_k} |\xi_k|^2 \omega_{\eta,k} \to 1.
\end{equation}
Proof:
Start by using that $|\xi_k| \le |w_k|$
\begin{equation}
\begin{aligned}
\abs{1- \int_{\Sigma_k} |\xi_k|^2 \omega_{\eta,k}}
&= \abs{\int_{\Sigma_k} (|w_k|^2 - |\xi_k|^2 ) \omega_{\eta,k}} \\
&\le \abs{\int_{(\Sigma_k\setminus \mathtt P_k^\varrho) \cup \mathtt O} (|w_k|^2 - |\xi_k|^2 ) \omega_{\eta,k}}
+ \abs{\int_{\mathtt M} (|w_k|^2 - |\xi_k|^2 ) \omega_{\eta,k}} \\
&\le \underbrace{\norm{\omega_{\eta,k}}_{L^\infty}}_{\le C} \underbrace{\int_{(\Sigma_k\setminus \mathtt P_k^\varrho) \cup \mathtt O} |w_k|^2 }_{\to0}
+ \abs{\int_{\mathtt M} (|w_k|^2 - |\xi_k|^2 ) \omega_{\eta,k}}
\end{aligned}
\end{equation}
Using $|\overline \xi_k| \le |\overline w_k|$ and also changing variables
\begin{equation}
\begin{aligned}
\abs{\int_{\mathtt M} (|w_k|^2 - |\xi_k|^2 ) \omega_{\eta,k}}
&= \abs{\int_{\frac{\sigma}{2}}^{1-\frac{\sigma}{2}} \int_0^{2\pi} (|\widetilde w_k|^2 - |\widetilde\xi_k|^2 ) \widetilde\omega_{\eta,k}\ T_{\eta,k} } \\
&= \abs{\int_{\frac{\sigma}{2}}^{1-\frac{\sigma}{2}} \int_0^{2\pi} (|\overline w_k|^2 - |\overline\xi_k|^2 ) \widetilde\omega_{\eta,k}\ T_{\eta,k}^2} \\
&\le 2\underbrace{\norm{\widetilde\omega_{\eta,k}\ T_{\eta,k}^2}_{L^\infty((\frac{\sigma}{2},1-\frac{\sigma}{2})\times S^1)}}_{\le C} \underbrace{\abs{\int_{\frac{\sigma}{2}}^{1-\frac{\sigma}{2}} \int_0^{2\pi} |\overline w_k|^2}}_{\to0}
\end{aligned}
\end{equation}
This proves Claim 2. \\
Now the fact that $Q_{u_k}(w_k)\le 0$ together with Claim 1 imply
\begin{equation}
\limsup_{k\to\infty} Q_{u_k}(\xi_k)\le 0.
\end{equation}
But also combining Claim 2 with Theorem \ref{THEOREM: Pos in conn necks}
\begin{equation}
\liminf_{k\to\infty} Q_{u_k}(\xi_k) \ge \overline \kappa>0.
\end{equation}
This is a contradiction and hence either $w_\infty \neq 0$ or $v_\infty \neq 0$
\end{proof}

\section{Proof of Main Theorem \ref{th-morse-stability-degnew}}
\label{SECTION: Proof of Main Theorem}

We can finally show

\begin{theorem} For large $k$ there holds
\begin{equation}
\operatorname{Ind}(u_k) + \operatorname{Null}(u_k)
\le \operatorname{Ind}(u_\infty) + \operatorname{Null}(u_\infty) + \operatorname{Ind}(\widetilde\gamma_\infty) + \operatorname{Null}(\widetilde\gamma_\infty)
\end{equation}
\end{theorem}

\begin{proof}
By Lemma \ref{LEMMA: Ind plus Null eq dim of eigenspaces} and Lemma \ref{LEMMA: Ind plus Null eq dim of eig for infty lim func} it suffices to show that for $k \in \N$ large and $\eta>0$ small
\begin{equation}
\operatorname{dim}\left( \bigoplus_{\lambda\le 0} \mathcal E_{\eta,k}(\lambda) \right) 
\le \dim\left(\mathcal E_{\eta,\infty}^0\right) + \dim\left(\widetilde{\mathcal E}_{\infty}^0\right).
\end{equation}
Let $M\in\N$ be fixed. 
For $k \in\N$ let $\phi_k^1,\ldots,\phi_k^M$ be a free orthonormal family of $U_{u_k}$ of eigenfunctions of the operator $\mathcal L_{\eta,k}$ with according negative eigenvalues $\lambda_{k}^1, \ldots, \lambda_{k}^M \le 0$.
For a contradiction we assume that 
\begin{equation}
\label{EQ: jnioecnHI1829he1n1ed231}
M > \dim\left(\mathcal E_{\eta,\infty}^0\right) + \dim\left(\widetilde {\mathcal E}_{\infty}^0\right).
\end{equation}
By Lemma \ref{LEMMA: On the non zero of the lim functs} we find that up to subsequences
\begin{equation}
\label{EQ: iuwenfiUNIUND9183hrn9231}
\phi_k^j \rightharpoondown \phi_\infty^j, \text{ weakly in } W^{2,2}_{loc}(\Sigma_k\setminus\{\gamma_k\})
\end{equation}
and 
\begin{equation}
\label{EQ: IUNi2h3nf2nf3njunjdqnwj2323fd21zujujdxqw}
\sigma_k^j(s,\theta) \coloneqq \frac{1}{\sqrt{T_{\eta,k}}}\ \phi_k^j(t(s),\theta) \rightharpoondown \sigma_\infty^j(z), \text{ weakly in } W^{2,2}_{loc}((0,1)\times S^1).
\end{equation}

{\bf Claim 1:} $\mathcal L_{\eta,\infty} \phi_\infty^j = \lambda_\infty^j \phi_\infty^j \text{ in } \Sigma$. \\
{\it Proof.}
Let $w\in W^{1,2}(\Sigma;\R^m)$ with compact support. Consider
\begin{equation}
\label{EQ: inwejnJININF130fn}
\begin{aligned}
\langle \mathcal L_{\eta,k} \phi_k^j , w \rangle_{\omega_{\eta,k}} 
&=\underbrace{\int_{\Sigma} \nabla  \phi_k^j \cdot P_{u_k} \nabla w \ dvol_\Sigma}_{\eqqcolon I_{\eta,k}} \\
&\hspace{20mm}-\underbrace{\int_{\Sigma}  S_{u_{k}}(\nabla u_{k}) \phi_k^j \cdot P_{u_k}w \ dvol_\Sigma}_{\eqqcolon II_{\eta,k}} \\
\end{aligned}
\end{equation}
First, using \eqref{EQ: iuwenfiUNIUND9183hrn9231} we know that
\begin{equation}
\nabla  \phi_k^j \rightharpoonup \nabla  \phi_\infty^j, \text{ weakly in } W^{1,2}_{loc}(\Sigma_k\setminus\{\gamma_k\})
\end{equation} 
and hence also with \eqref{EQ: JINhd2njkd21JNDs82fn28u4fhn2f}
\begin{equation}
I_{\eta,k} \rightarrow  \int_\Sigma \nabla  \phi_\infty^j \cdot P_{u_\infty} \nabla w \ dvol_\Sigma,  \qquad \text{ as } k \rightarrow \infty.
\end{equation}
Second, using \eqref{EQ: iuwenfiUNIUND9183hrn9231}, \eqref{EQ: JINhd2njkd21JNDs82fn28u4fhn2f} we know that
\begin{equation}
S_{u_{k}}(\nabla u_{k}) \phi_k^j \rightharpoondown S_{u_{\infty}}(\nabla u_{\infty}) \phi_\infty^j, \text{ weakly in } W^{2,2}_{loc}(\Sigma_k\setminus\{\gamma_k\})
\end{equation} 
and hence 
\begin{equation}
II_{\eta,k} \rightarrow \int_\Sigma S_{u_{\infty}}(\nabla u_{\infty}) \phi_\infty^j \cdot P_{u_\infty} w \ dvol_\Sigma,  \qquad \text{ as } k \rightarrow \infty.
\end{equation}
Going back to \eqref{EQ: inwejnJININF130fn} we have shown that 
\begin{equation}
\langle \mathcal L_{\eta,k} \phi_k^j , w \rangle_{\omega_{\eta,k}} \rightarrow \langle \mathcal L_{\eta,\infty} \phi_\infty^j , w \rangle_{\omega_{\eta,\infty}} , \qquad \text{ as } k \rightarrow \infty.
\end{equation}
This means that
\begin{equation}
\mathcal L_{\eta,k} \phi_k^j \rightharpoondown  \mathcal L_{\eta,\infty} \phi_\infty^j, \text{ weakly in } W^{1,2}_{loc}(\Sigma_k\setminus\{\gamma_k\}).
\end{equation}
This together with
\begin{equation}
\mathcal L_{\eta,k} \phi_k^j = \lambda_k^j \phi_k^j \rightharpoondown \lambda_\infty^j \phi_\infty^j, \text{ weakly in } W^{1,2}_{loc}(\Sigma_k\setminus\{\gamma_k\})
\end{equation}
 gives
\begin{equation}
\mathcal L_{\eta,\infty} \phi_\infty^j = \lambda_\infty^j \phi_\infty^j \text{ in } \Sigma .
\end{equation}
Now recall the defintion given in \eqref{EQ: JNUn2u3fn239fn3f2n}. \\
{\bf Claim 2:} $\widetilde {\mathcal L}_{\infty} \sigma_\infty^j = \lambda_\infty^j \sigma_\infty^j \text{ in } (0,1)$.

\medskip
{\it Proof.}
Let $\widetilde w\in W^{1,2}((0,1);\R^m)$ with compact support in $(0,1)$.
Fix $\sigma\in(0,1/2)$ with $\supp w\subset \subset (\sigma,1-\sigma)$.
We use the notation $w(t,\theta)\coloneqq \widetilde w(s(t))$, where $(t,\theta) \in \mathtt P^\varrho_k$.
We will pass in the identity
\begin{equation}
\label{EQ: weak form for sigma k j}
\sqrt{T_{\eta,k}} \langle \mathcal L_{\eta,k}\phi_k^j,w\rangle_{\omega_{\eta,k}} = \sqrt{T_{\eta,k}} \lambda_k^j \langle\phi_k^j,w\rangle_{\omega_{\eta,k}}
\end{equation}
to the limit.
On one hand, by changing variables
\begin{equation}
\begin{aligned}
&\sqrt{T_{\eta,k}} \langle \mathcal L_{\eta,k}\phi_k^j,w\rangle_{\omega_{\eta,k}} \\
&=\int_{(\sigma,1-\sigma)\times S^1}
\partial_s \sigma_k^j \cdot P_{\widetilde u_k}\partial_s \widetilde w\ ds\,d\theta \\
&\hspace{20mm}
-\int_{(\sigma,1-\sigma)\times S^1}
\Big(
S_{\widetilde u_k}(\partial_s\widetilde u_k)\sigma_k^j
+T_{\eta,k}^2 S_{\widetilde u_k}(\partial_\theta\widetilde u_k)\sigma_k^j
\Big)\cdot P_{\widetilde u_k}\widetilde w \ ds\,d\theta,
\end{aligned}
\end{equation}
where we used that $\partial_\theta \widetilde w=0$.
We now pass to the limit term by term.
Using \eqref{EQ: IUNi2h3nf2nf3njunjdqnwj2323fd21zujujdxqw} and Proposition \ref{congeo} we have
\begin{equation}
\int_{(\sigma,1-\sigma)\times S^1}
\partial_s \sigma_k^j \cdot P_{\widetilde u_k}\partial_s \widetilde w \ ds\,d\theta
\to
\int_{(\sigma,1-\sigma)\times S^1}
\partial_s \sigma_\infty^j \cdot P_{\widetilde\gamma_\infty}\partial_s \widetilde w \ ds\,d\theta .
\end{equation}
Proposition \ref{estdertheta} gives
\begin{equation}
T_{\eta,k}^2\int_{(\sigma,1-\sigma)\times S^1}
S_{\widetilde u_k}(\partial_\theta\widetilde u_k)\sigma_k^j\cdot P_{\widetilde u_k}\widetilde w \ ds\,d\theta \to 0.
\end{equation}
Using \eqref{EQ: IUNi2h3nf2nf3njunjdqnwj2323fd21zujujdxqw} and Proposition \ref{congeo}
\begin{equation}
\int_{(\sigma,1-\sigma)\times S^1}
S_{\widetilde u_k}(\partial_s\widetilde u_k)\sigma_k^j\cdot P_{\widetilde u_k}\widetilde w \ ds\,d\theta
\to
\int_{(\sigma,1-\sigma)\times S^1}
S_{\widetilde\gamma_\infty}(\partial_s\widetilde\gamma_\infty)\sigma_\infty^j\cdot P_{\widetilde\gamma_\infty}\widetilde w \ ds\,d\theta.
\end{equation}
On the other hand
\begin{equation}
\sqrt{T_{\eta,k}} \lambda_k^j \langle\phi_k^j,w\rangle_{\omega_{\eta,k}}
=\lambda_k^j \int_{(\sigma,1-\sigma)\times S^1}
\sigma_k^j\cdot \widetilde w \ \widetilde\omega_{\eta,k}T_{\eta,k}^2 \ ds\,d\theta
\end{equation}
Finally, since
\begin{equation}
\lambda_k^j\to \lambda_\infty^j
\end{equation}
up to subsequences, and
\begin{equation}
\widetilde\omega_{\eta,k}T_{\eta,k}^2 \to \widetilde\omega_\infty =1
\qquad\text{uniformly on }(\sigma,1-\sigma)\times S^1,
\end{equation}
we obtain
\begin{equation}
\lambda_k^j \int_{(\sigma,1-\sigma)\times S^1}
\sigma_k^j\cdot \widetilde w \ \widetilde\omega_{\eta,k}T_{\eta,k}^2 \ ds\,d\theta
\to
\lambda_\infty^j \int_{(0,1)\times S^1}
\sigma_\infty^j\cdot \widetilde w \ ds\,d\theta .
\end{equation}
Passing to the limit in \eqref{EQ: weak form for sigma k j}, we conclude
\begin{equation}
\begin{aligned}
&\int_0^1
\partial_s \sigma_\infty^j \cdot P_{\widetilde\gamma_\infty}\partial_s \widetilde w \ ds\,d\theta\\
&\hspace{20mm}
-\int_0^1
S_{\widetilde\gamma_\infty}(\partial_s\widetilde\gamma_\infty)\sigma_\infty^j\cdot P_{\widetilde\gamma_\infty}\widetilde w \ ds\,d\theta\\
&\qquad=
\lambda_\infty^j \int_0^1
\sigma_\infty^j\cdot \widetilde w \ ds\,d\theta .
\end{aligned}
\end{equation}
This proves that
\begin{equation}
\widetilde {\mathcal L}_{\infty} \sigma_\infty^j = \lambda_\infty^j \sigma_\infty^j
\qquad \text{ in } (0,1).
\end{equation}
Now since by \eqref{EQ: jnioecnHI1829he1n1ed231} $M > \dim({\mathcal E}_{\eta,\infty}^0 \times \widetilde{\mathcal E}_{\infty}^0)$ we have that the family $(\phi_\infty^j,\sigma_\infty^j)_{j=1\ldots M}$ is linearly dependent and we can find some $(c_\infty^1, \ldots, c_\infty^M)\neq 0$ such that
\begin{equation}
\label{EQ: wifnuJNIOFN103mnf223}
\sum_{j=1}^M c_\infty^j \phi_\infty^j=0 \qquad \text{ and } \qquad \sum_{j=1}^M c_\infty^j \sigma_\infty^j=0.
\end{equation}
Let 
\begin{equation}
w_k \coloneqq \frac{1}{\left(\sum_{j=1}^M (c_\infty^j)^2\right)^{\frac{1}{2}}} \sum_{j=1}^M c_\infty^j \phi_k^j.
\end{equation}
Then $w_k \in \mathcal S_{\eta,k}^0$ and by Lemma \ref{LEMMA: On the non zero of the lim functs} up to subsequences
\begin{equation}
w_k \rightharpoondown w_\infty, \text{ in } \dot W^{1,2}(\Sigma) \qquad \text{ and } \qquad
\overline w_{k}(s,\theta) \rightharpoondown v_\infty(y), \text{ in } \dot W^{1,2}((0,1)\times S^1)
\end{equation}
and either $w_\infty \ne 0$ or $v_\infty \ne 0$.
But by \eqref{EQ: wifnuJNIOFN103mnf223} one has $(w_\infty,v_\infty)=(0,0)$.
This is a contradiction.
\end{proof}

\section{Appendix}
\subsection{Flat-torus Example}
\label{APPENDIX: Flat torus counterexample}

We give the details of the example announced in Remark \ref{REMARK: Converse without positive Ricci}. For $k\in\N$, let
\[
(\Sigma_k,h_k)
=
\left((\R/\Z)^2,k^2dx^2+k^{-2}dy^2\right),
\qquad
\mathcal N=\mathbb T^2=\R^2/\Z^2,
\]
let $\widetilde\gamma(s)=(s\!\!\mod 1,0)$ be the unit speed closed geodesic of period one in $\mathbb T^2$, and define
\[
u_k(x,y)=\widetilde\gamma(kx).
\]
The maps $u_k$ are harmonic. We represent $(\Sigma_k,h_k)$ by the long cylinder
\[
(t,\theta)\in[0,2\pi k^2]\times S^1,
\qquad
S^1=\R/(2\pi\Z),
\]
where the boundary circles $\{0\}\times S^1$ and $\{2\pi k^2\}\times S^1$ are identified. The conformal change of coordinates $t=2\pi k^2x$ and $\theta=2\pi y$ gives
\[
h_k=\frac{1}{4\pi^2k^2}(dt^2+d\theta^2),
\qquad
u_k(t,\theta)=\widetilde\gamma\left(\frac{t}{2\pi k}\right).
\]
In particular, the conformal classes are degenerating and
\[
\int_0^{2\pi k^2}\int_0^{2\pi}
|\partial_tu_k|^2\,d\theta\,dt=1,
\]
so that the energies are uniformly bounded, independently of the normalization of the Dirichlet energy.

There is only one degenerating collar, represented by $[0,2\pi k^2]\times S^1$. Setting $r=e^{-t}$ and $\delta_k^1=e^{-2\pi k^2}$, the map in annular coordinates is
\[
u_k(r,\theta)
=
\widetilde\gamma\left(-\frac{\log r}{2\pi k}\right).
\]
Therefore, for every fixed $\eta\in(0,1)$,
\[
\int_{\eta^{-1}\delta_k^1}^{\eta}
\frac{1}{2\pi}\int_0^{2\pi}
\left|\frac{du_k}{dr}\right|\,d\theta\,dr
=
\frac{1}{2\pi k}\log\left(\frac{\eta^2}{\delta_k^1}\right)
=
k+\frac{\log\eta}{\pi k}
\longrightarrow+\infty.
\]
Hence $\Lambda^1=+\infty$.

It remains to compute the extended Morse index. The quotient structure of $\mathbb T^2$ identifies every variation field along $u_k$ with a pair $V=(f,g)$ of $\Z^2$-periodic functions on $\R^2$, or equivalently with two periodic functions on the fundamental domain $[0,1]^2$. The corresponding variation is
\[
u_{k,s}(x,y)
=
\bigl(kx+sf(x,y),sg(x,y)\bigr)\!\!\mod\Z^2.
\]
Since
\[
(h_k^{ij})
=
\begin{pmatrix}
k^{-2}&0\\
0&k^2
\end{pmatrix},
\qquad
d\mathrm{vol}_{h_k}=dx\,dy,
\]
the Dirichlet energy is
\[
\begin{split}
E(u_{k,s})
=
\frac12\int_{[0,1]^2}
\Bigl[
&k^{-2}\bigl((k+s\partial_xf)^2+s^2(\partial_xg)^2\bigr)
\\
&+
k^2s^2\bigl((\partial_yf)^2+(\partial_yg)^2\bigr)
\Bigr]\,dx\,dy.
\end{split}
\]
Taking derivatives and evaluating
\[
\begin{aligned}
Q_{u_k}(V)
&=
\left.\frac{d^2}{ds^2}\right|_{s=0}E(u_{k,s})
\\
&=
\int_{[0,1]^2}
\Bigl[
k^{-2}\bigl((\partial_xf)^2+(\partial_xg)^2\bigr)
+
k^2\bigl((\partial_yf)^2+(\partial_yg)^2\bigr)
\Bigr]\,dx\,dy.
\end{aligned}
\]
Notice also that
\[
|df|_{h_k}^2
=
k^{-2}(\partial_xf)^2+k^2(\partial_yf)^2,
\qquad
d\mathrm{vol}_{h_k}=dx\,dy,
\]
and similarly for $g$.
Thus
\[
Q_{u_k}(V)
=
\int_{\Sigma_k}
\left(|df|_{h_k}^2+|dg|_{h_k}^2\right)
\,d\mathrm{vol}_{h_k}.
\]

It follows immediately that $\mathrm{Ind}_E(u_k)=0$. Moreover, the kernel of $Q_{u_k}$ consists precisely of those $V=(f,g)$ for which $d f=d g=0$. Since $\Sigma_k$ is connected, $f$ and $g$ must be constant, and hence $\mathrm{Null}_E(u_k)=2$. Consequently,
\[
\mathrm{Ind}_E^\ast(u_k)=2
\qquad\text{for every }k,
\]
whereas $\Lambda^1=+\infty$. Thus, without a positive Ricci curvature assumption on the target, even constancy of the extended Morse index does not imply finiteness of the limiting average collar lengths.
 
\subsection{Matching the collar and bubble weights in the one-bubble case}
\label{APPENDIX: Reduction to concentration-free collars}

In this section we sketch  the construction of a global  weight when one bubble
forms in one degenerating collar.  Thus we assume $M=1$ and $M_1=1$.
   The compatibility problem for several bubbles and several bubble
scales will be treated in \cite{DMRS}.

\paragraph{Decomposition of the collar.}
We omit the collar index $j=1$.  There exist sequences $a_k,b_k$ satisfying
\begin{equation}
\label{EQ: appendix ordering of bubble cylinder}
 T_k^{1,\varrho}\leq a_k\ll b_k\leq T_k^{2,\varrho}
\end{equation}
such that
\begin{equation}
\label{EQ: appendix one bubble decomposition}
 \mathtt P_k^\varrho
 =
 \mathtt G_k^0\cup\mathtt B_k\cup\mathtt G_k^1,
\end{equation}
where
\begin{equation}
\label{EQ: appendix three collar components}
 \mathtt G_k^0
 :=
 [T_k^{1,\varrho},a_k]\times S^1,
 \qquad
 \mathtt B_k
 :=
 [a_k,b_k]\times S^1,
 \qquad
 \mathtt G_k^1
 :=
 [b_k,T_k^{2,\varrho}]\times S^1.
\end{equation}
After suitable reparametrizations, $\mathtt G_k^0$ and
$\mathtt G_k^1$ converge to the geodesic segments
$\widetilde\gamma_\infty^{1,0}$ and
$\widetilde\gamma_\infty^{1,1}$, respectively, whereas
$\mathtt B_k$ contains the unique bubble
$\sigma_\infty^{1,1}$.

The modulus parameter of the entire degenerating collar is
\begin{equation}
\label{EQ: appendix global collar parameter}
 \delta_k^{\mathrm{col}}
 :=
 \exp\left(-\frac{2\pi^2}{l_k}\right).
\end{equation}
It is independent of $\nu$, because $\nu\in\{0,1\}$ labels the two
components obtained by removing the bubble region from the same collar.
If one conformally normalizes the truncated component
$\mathtt G_k^\nu$ separately, its modulus may depend on $\nu$; we denote
that auxiliary modulus by $d_{\eta,k}^\nu$.  Thus
$ \delta_k^{\mathrm{col}}
 \quad\text{and}\quad
 d_{\eta,k}^\nu$ 
have different meanings: the first describes the whole collar, while the
second describes a chosen normalization of one truncated component.

Let $q_k$ be the concentration point and let $\lambda_k\to0$ be the
bubble scale.  In a conformal coordinate centred at $q_k$, the no-neck
property takes the two-parameter form
\begin{equation}
\label{EQ: appendix two parameter no neck}
 \lim_{R\to+\infty}\lim_{\rho\to0}\limsup_{k\to\infty}
 \int_{B_\rho(0)\setminus\bar{B}_{R\lambda_k}(0)}
 |\nabla u_k|^2
 =
 0.
\end{equation}
By a diagonal argument we can find \begin{equation}
\label{EQ: appendix rho conditions}
 \rho_k\longrightarrow0,
 \qquad
 \frac{\lambda_k}{\rho_k^2}\longrightarrow0.
\end{equation}
 and 
\begin{equation}
\label{EQ: appendix diagonal no neck}
 \int_{B_{\rho_k}(0)\setminus
 \bar{B}_{\lambda_k/\rho_k}(0)}
 |\nabla u_k|^2
 \longrightarrow0.
\end{equation}

 The rescaled maps
\begin{equation}
\label{EQ: appendix rescaled bubble maps}
 v_k(y):=u_k(\lambda_k y)
\end{equation}
are defined on
 $ B_{\rho_k/\lambda_k}(0)$ 
and converge locally on $\mathbb R^2$ to
$\sigma_\infty^{1,1}$.

\par
Fix $\eta>0$ sufficiently small and set
\begin{equation}
\label{EQ: appendix local dilation}
 y:=\frac{\eta}{\rho_k}x.
\end{equation}
Under this dilation, $\mathcal A_k$ becomes
\begin{equation}
\label{EQ: appendix normalized neck annulus}
 A\bigl(\eta,\delta_{\eta,k}^{\mathrm{neck}}\bigr)
 :=
 B_\eta(0)\setminus
 \bar{B}_{\delta_{\eta,k}^{\mathrm{neck}}/\eta}(0),
 \qquad
 \delta_{\eta,k}^{\mathrm{neck}}
 :=
 \frac{\eta^2\lambda_k}{\rho_k^2}.
\end{equation}
  
\paragraph{The neck weight and the concentration-ball weight.}
Set
\begin{equation}
\label{EQ: appendix bubble interface radius}
 s_{\eta,k}
 :=
 \frac{\delta_{\eta,k}^{\mathrm{neck}}}{\eta}.
\end{equation}

We recall that in \cite{DGR25}   the following  densities have been introduced respectively on the neck annulus $B_\eta(0)\setminus\bar{B}_{s_{\eta,k}}(0)$ and on the concentration disk $B_{s_{\eta,k}}(0)$:
\begin{equation}
\label{EQ: appendix local neck weight}
 \omega_{\eta,k}^{\mathrm{neck},\sharp}(y)
 :=
 \frac{1}{|y|^2}
 \left[
  \left(\frac{|y|}{\eta}\right)^\beta
  +
  \left(
   \frac{\delta_{\eta,k}^{\mathrm{neck}}}{\eta|y|}
  \right)^\beta
  +
  \frac{1}{
   \log^2(\eta^2/\delta_{\eta,k}^{\mathrm{neck}})}
 \right];
\end{equation}
 \begin{equation}
\label{EQ: appendix concentration ball weight2}
 \omega_{\eta,k}^{\mathrm{bub},\sharp}(y)
 :=
 \frac{\eta^2}{(\delta_{\eta,k}^{\mathrm{neck}})^2}
 \left[
  \frac{1}{\eta^4}
  \frac{(1+\eta^2)^2}
       {\left(
        1+(\delta_{\eta,k}^{\mathrm{neck}})^{-2}|y|^2
       \right)^2}
  +
  \left(
   \frac{\delta_{\eta,k}^{\mathrm{neck}}}{\eta^2}
  \right)^\beta
  +
  \frac{1}{
   \log^2(\eta^2/\delta_{\eta,k}^{\mathrm{neck}})}
 \right].
\end{equation}

The two regions $B_\eta(0)\setminus\bar{B}_{s_{\eta,k}}(0)$ and ${B}_{s_{\eta,k}}(0)$ meet along
\begin{equation}
\label{EQ: appendix inner bubble interface}
 \Gamma_{\eta,k}^{\mathrm{in}}
 :=
 \{|y|=s_{\eta,k}\}.
\end{equation}
Their common cylindrical boundary value is
\begin{equation}
\label{EQ: appendix matching constant definition}
 \alpha_k
 :=
 s_{\eta,k}^2
 \omega_{\eta,k}^{\mathrm{bub},\sharp}
 (s_{\eta,k}e^{\mathrm i\theta}).
\end{equation}
The weight is radial, so $\alpha_k$ is independent of $\theta$.
At $|y|=s_{\eta,k}$, the stereographic term satisfies
\begin{equation}
\label{EQ: appendix stereographic boundary computation}
 \frac{1}{\eta^4}
 \frac{(1+\eta^2)^2}
 {\left(
  1+(\delta_{\eta,k}^{\mathrm{neck}})^{-2}|y|^2
 \right)^2}
 =
 1.
\end{equation}
Consequently,
\begin{align}
\label{EQ: appendix matching constant}
 \alpha_k
 &=
 1+
 \left(
  \frac{\delta_{\eta,k}^{\mathrm{neck}}}{\eta^2}
 \right)^\beta
 +
 \frac{1}{
  \log^2(\eta^2/\delta_{\eta,k}^{\mathrm{neck}})}
 \notag=
 1+
 \left(\frac{\lambda_k}{\rho_k^2}\right)^\beta
 +
 \frac{1}{\log^2(\rho_k^2/\lambda_k)}.
\end{align}
By \eqref{EQ: appendix rho conditions},
\begin{equation}
\label{EQ: appendix matching constant limit}
 \alpha_k\longrightarrow1,~~~\mbox{as $k\to +\infty$}.
\end{equation}
In particular, there exist constants $c,C>0$ such that
\begin{equation}
\label{EQ: appendix matching constant bounds}
 0<c\leq\alpha_k\leq C<+\infty
\end{equation}
for all sufficiently large $k$.

The cylindrical version of the neck weight is
\begin{equation}
\label{EQ: appendix cylindrical neck weight}
 \omega_{\eta,k}^{\mathrm{neck}}(r)
 :=
 r^2\omega_{\eta,k}^{\mathrm{neck},\sharp}(re^{\mathrm i\theta})
 =
 \left(\frac r\eta\right)^\beta
 +
 \left(
  \frac{\delta_{\eta,k}^{\mathrm{neck}}}{\eta r}
 \right)^\beta
 +
 \frac{1}{
  \log^2(\eta^2/\delta_{\eta,k}^{\mathrm{neck}})},~~~r=e^{-t}.
\end{equation}
Observe that
\begin{equation}
\label{EQ: appendix symmetric neck traces}
 \omega_{\eta,k}^{\mathrm{neck}}(\eta)
 =
 \omega_{\eta,k}^{\mathrm{neck}}(s_{\eta,k})
 =
 \alpha_k.
\end{equation}

\paragraph{Local interpolation on the collar side.}
We describe the construction for a fixed
$\nu\in\{0,1\}$.  Orient the  coordinate $z=e^{-t+\mathrm
i\theta}$ on $\mathtt G_k^\nu$ so that the interface with the bubble
region is $\{|z|=\eta\}$.  We denote the interpolation region by  
\begin{equation}
\label{EQ: appendix interpolation region}
 \mathcal O_{\eta,k}^\nu
 :=
 \left\{z\in\mathbb R^2:
  \frac{\eta}{4}\leq|z|\leq\eta
 \right\}.
\end{equation}
In cylindrical coordinates this becomes
\begin{equation}
\label{EQ: appendix cylindrical interpolation region}
 \mathcal O_{\eta,k}^\nu
 =
 \left[
  -\log\eta,
  -\log\frac{\eta}{4}
 \right]\times S^1.
\end{equation}
 Let $\omega_{\eta,k}^{\mathrm{col},\nu}$ denote the cylindrical collar
weight in this coordinate.  If $d_{\eta,k}^\nu$ is the modulus parameter
of the   component $\mathtt G_k^\nu$, then
\begin{equation}
\label{EQ: appendix component collar weight}
 \omega_{\eta,k}^{\mathrm{col},\nu}(r,\theta)
 :=
 \left(\frac r\eta\right)^\beta
 +
 \left(\frac{d_{\eta,k}^\nu}{\eta r}\right)^\beta
 +
 \frac{1}{\log^2(\eta^2/d_{\eta,k}^\nu)},~~r=e^{-t}.
\end{equation}
On $\mathcal O_{\eta,k}^\nu$, one has
\begin{equation}
\label{EQ: appendix first collar term bounds}
 4^{-\beta}
 \leq
 \left(\frac r\eta\right)^\beta
 \leq
 1
\end{equation}
and
\begin{equation}
\label{EQ: appendix second collar term bound}
 0
 \leq
 \left(\frac{d_{\eta,k}^\nu}{\eta r}\right)^\beta
 \leq
 \left(\frac{4d_{\eta,k}^\nu}{\eta^2}\right)^\beta.
\end{equation}
Since $d_{\eta,k}^\nu/\eta^2\to0$, as $k\to +\infty$, it follows that
\begin{equation}
\label{EQ: appendix collar weight bounds}
 0<c
 \leq
 \omega_{\eta,k}^{\mathrm{col},\nu}
 \leq
 C<+\infty
 \qquad\text{on }\mathcal O_{\eta,k}^\nu.
\end{equation}

We choose a nondecreasing function
$\chi\in C^\infty([0,+\infty),[0,1])$ satisfying
\begin{equation}
\label{EQ: appendix cutoff properties}
 \chi(s)=0\quad\text{for }s\leq\frac14,
 \qquad
 \chi(s)=1\quad\text{for }s\geq\frac12.
\end{equation}
We modify the collar weight only on
$\mathcal O_{\eta,k}^\nu$, setting
\begin{equation}
\label{EQ: appendix modified collar weight}
 \mathfrak z_{\eta,k}^\nu(t,\theta)
 :=
 \begin{cases}
 \omega_{\eta,k}^{\mathrm{col},\nu}(t,\theta),
 &e^{-t}\leq\eta/4,\\[1mm]
 \displaystyle
 \left(
  1-\chi\left(\frac{e^{-t}}{\eta}\right)
 \right)
 \omega_{\eta,k}^{\mathrm{col},\nu}(t,\theta)
 +
 \chi\left(\frac{e^{-t}}{\eta}\right)\alpha_k,
 &\eta/4<e^{-t}<\eta/2,\\[3mm]
 \alpha_k,
 &\eta/2\leq e^{-t}\leq\eta.
 \end{cases}
\end{equation}
 The modified weight agrees with the original one away from the interface
and has the matching trace
\begin{equation}
\label{EQ: appendix exact collar bubble matching}
 \mathfrak z_{\eta,k}^\nu(-\log\eta,\theta)
 =
 \alpha_k
 =
 \eta^2
 \widehat\omega_{\eta,k}^{\mathrm{bub},\sharp}
 (\eta e^{\mathrm i\theta}).
\end{equation}

{\bf Claim:} 
 There is a constant $C$ such that for $\sigma\in (0,1/2)$ it holds:
\begin{equation}
\label{EQ: appendix local weight comparison}
 C^{-1}\omega_{\eta,k}^{\mathrm{col},\nu}
 \leq
 \mathfrak z_{\eta,k}^\nu
 \leq
 C\omega_{\eta,k}^{\mathrm{col},\nu}\leq C \mathfrak{w}_{\sigma,\eta,k}^\nu
 \qquad\text{in }\mathtt G_k^\nu.
\end{equation}

\medskip
\noindent
\textbf{Proof of the claim.}
We divide the proof into two steps.

\smallskip
\noindent
\emph{\bf{Step 1:} comparison between the modified weight and the collar
weight.}
By construction, the modified weight
$\mathfrak z_{\eta,k}^\nu$ coincides with
$\omega_{\eta,k}^{\mathrm{col},\nu}$ outside the interpolation region
$\mathcal O_{\eta,k}^\nu$. Therefore,
\begin{equation}
\label{EQ: appendix z equals collar away overlap}
 \mathfrak z_{\eta,k}^\nu
 =
 \omega_{\eta,k}^{\mathrm{col},\nu}
 \qquad\text{in }
 \mathtt G_k^\nu\setminus\mathcal O_{\eta,k}^\nu.
\end{equation}

Inside $\mathcal O_{\eta,k}^\nu$, the modified weight is a convex
interpolation between the collar weight and the matching constant
$\alpha_k^\nu$. More precisely,
\begin{equation}
\label{EQ: appendix z convex interpolation}
 \mathfrak z_{\eta,k}^\nu
 =
 (1-\chi_{\eta,k}^\nu)
 \omega_{\eta,k}^{\mathrm{col},\nu}
 +
 \chi_{\eta,k}^\nu\alpha_k^\nu,
 \qquad
 0\leq\chi_{\eta,k}^\nu\leq1.
\end{equation}
Equations \eqref{EQ: appendix matching constant bounds} and
\eqref{EQ: appendix collar weight bounds} imply that, on
$\mathcal O_{\eta,k}^\nu$, the two quantities
$\alpha_k^\nu$ and
$\omega_{\eta,k}^{\mathrm{col},\nu}$ are uniformly comparable:
\begin{equation}
\label{EQ: appendix matching collar comparability}
 C^{-1}\omega_{\eta,k}^{\mathrm{col},\nu}
 \leq
 \alpha_k^\nu
 \leq
 C\omega_{\eta,k}^{\mathrm{col},\nu}.
\end{equation}
Inserting this estimate into
\eqref{EQ: appendix z convex interpolation} gives
\begin{equation}
\label{EQ: appendix z collar comparison on overlap}
 C^{-1}\omega_{\eta,k}^{\mathrm{col},\nu}
 \leq
 \mathfrak z_{\eta,k}^\nu
 \leq
 C\omega_{\eta,k}^{\mathrm{col},\nu}
 \qquad\text{in }\mathcal O_{\eta,k}^\nu.
\end{equation}
Together with \eqref{EQ: appendix z equals collar away overlap}, this
proves
\begin{equation}
\label{EQ: appendix global z collar comparison}
 C^{-1}\omega_{\eta,k}^{\mathrm{col},\nu}
 \leq
 \mathfrak z_{\eta,k}^\nu
 \leq
 C\omega_{\eta,k}^{\mathrm{col},\nu}
 \qquad\text{in }\mathtt G_k^\nu.
\end{equation}

\smallskip
\noindent
\emph{{\bf Step 2: comparison with the  weight $\mathfrak{w}_{\sigma,\eta,k}^\nu$}}

Write $z=re^{\mathrm i\theta}$ and represent the normalized component
by
\begin{equation}
\label{EQ: appendix component normalized annulus}
 A_{\eta,k}^\nu
 :=
 B_\eta(0)\setminus
 \bar{B}_{d_{\eta,k}^\nu/\eta}(0).
\end{equation}
The collar weight is
\begin{equation}
\label{EQ: appendix explicit collar weight}
 \omega_{\eta,k}^{\mathrm{col},\nu}(r)
 =
 \left(\frac r\eta\right)^\beta
 +
 \left(\frac{d_{\eta,k}^\nu}{\eta r}\right)^\beta
 +
 \frac{1}{(T_{\eta,k}^\nu)^2}.
\end{equation}
Set
\begin{equation}
\label{EQ: appendix component scale ratio}
 a_{\eta,k}^\nu
 :=
 \frac{d_{\eta,k}^\nu}{\eta^2},
 \qquad
 \varepsilon_{\eta,k}^\nu
 :=
 \bigl(a_{\eta,k}^\nu\bigr)^\sigma,
 \qquad
 T_{\eta,k}^\nu
 :=
 \log\left(\frac{\eta^2}{d_{\eta,k}^\nu}\right),
 \qquad
 0<\sigma<\frac12.
\end{equation}
Since $a_{\eta,k}^\nu\to0$, for all sufficiently large $k$,
\begin{equation}
\label{EQ: appendix scale ratio comparison}
 0<a_{\eta,k}^\nu
 \leq
 \varepsilon_{\eta,k}^\nu
 \leq1.
\end{equation}

The  annulus is decomposed into
\begin{align}
\label{EQ: appendix inner subannulus}
 A_{\mathrm{in},\eta,k}^\nu
 &:=
 \left\{
  \frac{d_{\eta,k}^\nu}{\eta}
  \leq r\leq
  \frac{d_{\eta,k}^\nu}
       {\eta\varepsilon_{\eta,k}^\nu}
 \right\},
\\
\label{EQ: appendix middle subannulus}
 A_{\mathrm{mid},\eta,k}^\nu
 &:=
 \left\{
  \frac{d_{\eta,k}^\nu}
       {\eta\varepsilon_{\eta,k}^\nu}
  \leq r\leq
  \eta\varepsilon_{\eta,k}^\nu
 \right\},
\\
\label{EQ: appendix outer subannulus}
 A_{\mathrm{out},\eta,k}^\nu
 &:=
 \left\{
  \eta\varepsilon_{\eta,k}^\nu
  \leq r\leq\eta
 \right\}.
\end{align}
Their radial endpoints are correctly ordered because
\begin{equation}
\label{EQ: appendix ordered radial endpoints}
 \frac{d_{\eta,k}^\nu}{\eta}
 =
 \eta a_{\eta,k}^\nu
 \leq
 \eta(a_{\eta,k}^\nu)^{1-\sigma}
 =
 \frac{d_{\eta,k}^\nu}
      {\eta\varepsilon_{\eta,k}^\nu}
 \leq
 \eta(a_{\eta,k}^\nu)^\sigma
 =
 \eta\varepsilon_{\eta,k}^\nu
 \leq\eta.
\end{equation}

We now compare the two weights separately on these three regions.

\smallskip
\noindent
\underline{\emph{Outer sub-annulus.}}\par

On $A_{\mathrm{out},\eta,k}^\nu$,  we have 
\begin{equation}
\label{EQ: appendix positivity weight outer formula}
 \mathfrak w_{\sigma,\eta,k}^\nu(r)
 =
 \left(\frac r\eta\right)^\beta
 +
 \left(
  \varepsilon_{\eta,k}^\nu\frac{\eta}{r}
 \right)^\beta
 +
 \frac{1}{\sigma^2(T_{\eta,k}^\nu)^2}.
\end{equation}
Both
\eqref{EQ: appendix explicit collar weight} and
\eqref{EQ: appendix positivity weight outer formula} contain the term
\begin{equation}
\label{EQ: appendix common outer first term}
 \left(\frac r\eta\right)^\beta.
\end{equation}
For the second terms, we have
\begin{equation}
\label{EQ: appendix outer second term comparison}
 \frac{d_{\eta,k}^\nu}{\eta r}
 =
 a_{\eta,k}^\nu\frac{\eta}{r}
 \leq
 \varepsilon_{\eta,k}^\nu\frac{\eta}{r},
\end{equation}
and therefore, since $\beta>0$,
\begin{equation}
\label{EQ: appendix outer powered comparison}
 \left(\frac{d_{\eta,k}^\nu}{\eta r}\right)^\beta
 \leq
 \left(
  \varepsilon_{\eta,k}^\nu\frac{\eta}{r}
 \right)^\beta.
\end{equation}
Finally,
\begin{equation}
\label{EQ: appendix outer logarithmic comparison}
 \frac{1}{(T_{\eta,k}^\nu)^2}
 \leq
 \frac{1}{\sigma^2(T_{\eta,k}^\nu)^2}.
\end{equation}
Thus every term of the collar weight $ \mathfrak w_{\sigma,\eta,k}^\nu$  is bounded by the corresponding
term of the positivity weight on
$A_{\mathrm{out},\eta,k}^\nu$.

\smallskip
\noindent
\underline{\emph{Middle sub-annulus.}}
On $A_{\mathrm{mid},\eta,k}^\nu$, the  weight $ \mathfrak w_{\sigma,\eta,k}^\nu$ is
\begin{equation}
\label{EQ: appendix positivity weight middle formula}
 \mathfrak w_{\sigma,\eta,k}^\nu(r)
 =
 \left(
  \frac{r}{\eta\varepsilon_{\eta,k}^\nu}
 \right)^\beta
 +
 \left(
  \frac{d_{\eta,k}^\nu}
       {\eta\varepsilon_{\eta,k}^\nu r}
 \right)^\beta
 +
 \frac{1}{(1-2\sigma)^2(T_{\eta,k}^\nu)^2}.
\end{equation}
Since $\varepsilon_{\eta,k}^\nu\leq1$,
\begin{equation}
\label{EQ: appendix middle base comparisons}
 \frac r\eta
 \leq
 \frac{r}{\eta\varepsilon_{\eta,k}^\nu},
 \qquad
 \frac{d_{\eta,k}^\nu}{\eta r}
 \leq
 \frac{d_{\eta,k}^\nu}
      {\eta\varepsilon_{\eta,k}^\nu r}.
\end{equation}
Raising these inequalities to the positive power $\beta$ compares
the two power terms. Moreover, because $0<1-2\sigma<1$,
\begin{equation}
\label{EQ: appendix middle logarithmic comparison}
 \frac{1}{(T_{\eta,k}^\nu)^2}
 \leq
 \frac{1}{(1-2\sigma)^2(T_{\eta,k}^\nu)^2}.
\end{equation}
Hence
\begin{equation}
\label{EQ: appendix middle complete comparison}
 \omega_{\eta,k}^{\mathrm{col},\nu}
 \leq
 \mathfrak w_{\sigma,\eta,k}^\nu
 \qquad\text{on }A_{\mathrm{mid},\eta,k}^\nu.
\end{equation}

\smallskip
\noindent
\underline{\emph{Inner sub-annulus.}}\par
On $A_{\mathrm{in},\eta,k}^\nu$, the  weight $\mathfrak w_{\sigma,\eta,k}^\nu$ is
\begin{equation}
\label{EQ: appendix positivity weight inner formula}
 \mathfrak w_{\sigma,\eta,k}^\nu(r)
 =
 \left(
  \frac{r\eta\varepsilon_{\eta,k}^\nu}
       {d_{\eta,k}^\nu}
 \right)^\beta
 +
 \left(
  \frac{d_{\eta,k}^\nu}{\eta r}
 \right)^\beta
 +
 \frac{1}{\sigma^2(T_{\eta,k}^\nu)^2}.
\end{equation}
The \emph{second} power term  agrees   in the two
weights $\mathfrak w_{\sigma,\eta,k}^\nu(r)$ and $ \omega_{\eta,k}^{\mathrm{col},\nu}:$
\begin{equation}
\label{EQ: appendix common inner second term}
 \left(\frac{d_{\eta,k}^\nu}{\eta r}\right)^\beta.
\end{equation}
For the first terms, using
$d_{\eta,k}^\nu=a_{\eta,k}^\nu\eta^2$, we obtain
\begin{equation}
\label{EQ: appendix inner first term comparison}
 \frac{r\eta\varepsilon_{\eta,k}^\nu}
      {d_{\eta,k}^\nu}
 =
 \frac r\eta
 \frac{\varepsilon_{\eta,k}^\nu}
      {a_{\eta,k}^\nu}
 \geq
 \frac r\eta,
\end{equation}
because
$a_{\eta,k}^\nu\leq\varepsilon_{\eta,k}^\nu$. Since $\beta>0$,
\begin{equation}
\label{EQ: appendix inner powered comparison}
 \left(\frac r\eta\right)^\beta
 \leq
 \left(
  \frac{r\eta\varepsilon_{\eta,k}^\nu}
       {d_{\eta,k}^\nu}
 \right)^\beta.
\end{equation}
The logarithmic term is treated as on the outer sub-annulus. Therefore,
\begin{equation}
\label{EQ: appendix inner complete comparison}
 \omega_{\eta,k}^{\mathrm{col},\nu}
 \leq
 \mathfrak w_{\sigma,\eta,k}^\nu
 \qquad\text{on }A_{\mathrm{in},\eta,k}^\nu.
\end{equation}

Combining the estimates on the three sub-annuli gives
\begin{equation}
\label{EQ: appendix collar positivity final comparison}
 \omega_{\eta,k}^{\mathrm{col},\nu}
 \leq
 \mathfrak w_{\sigma,\eta,k}^\nu
 \qquad\text{in }\mathtt G_k^\nu.
\end{equation}
Together with \eqref{EQ: appendix global z collar comparison}, this
proves
\begin{equation}
\label{EQ: appendix local weight comparison}
 C^{-1}\omega_{\eta,k}^{\mathrm{col},\nu}
 \leq
 \mathfrak z_{\eta,k}^\nu
 \leq
 C\omega_{\eta,k}^{\mathrm{col},\nu}
 \leq
 C\mathfrak w_{\sigma,\eta,k}^\nu
 \qquad\text{in }\mathtt G_k^\nu.
\end{equation}
\hfill$\Box$ \paragraph{Pointwise estimate in the interpolation region.}
Let us explain how the pointwise estimate can be recovered with the modification to the weight.

Recall that, in the conformal coordinate
$z=re^{\mathrm i\theta}$ used on the component
$\mathtt G_k^\nu$, the interpolation region is
\begin{equation}
\label{EQ: appendix interpolation region recalled}
 \mathcal O_{\eta,k}^\nu
 :=
 \left\{
  z\in\mathbb R^2:
  \frac{\eta}{4}\leq|z|\leq\eta
 \right\}.
\end{equation}
Introduce the enlarged annulus
\begin{equation}
\label{EQ: appendix enlarged interpolation region}
 \widetilde{\mathcal O}_{\eta,k}^\nu
 :=
 \left\{
  z\in\mathbb R^2:
  \frac{\eta}{8}\leq|z|\leq2\eta
 \right\}.
\end{equation}
We may always assume that in this enlarged annulus there is no-concentration of the energy:
\begin{equation}
\label{EQ: appendix enlarged interpolation energy}
 \lim_{\eta\to0}\limsup_{k\to\infty}
 \int_{\widetilde{\mathcal O}_{\eta,k}^\nu}
 |\nabla_z u_k^\sharp|^2\,dz
 =
 0.
\end{equation}

For every $z\in\mathcal O_{\eta,k}^\nu$, we have 
\begin{equation}
\label{EQ: appendix epsilon regularity ball inclusion}
 B_{\eta/8}(z)
 \subset
 \widetilde{\mathcal O}_{\eta,k}^\nu.
\end{equation}
The $\varepsilon$-regularity estimate and conformal invariance give
\begin{align}
\label{EQ: appendix interpolation epsilon regularity}
 |\nabla u_k(t,\theta)|^2
 &=
 |z|^2|\nabla u_k^\sharp(z)|^2
 \leq
 C\frac{|z|^2}{\eta^2}
 \|\nabla u_k^\sharp\|_{L^2(B_{\eta/8}(z))}^2
 =
 o_{\eta,k}(1),
\end{align}
where
\begin{equation}
\label{EQ: appendix small quantity convention}
 \lim_{\eta\to0}\limsup_{k\to\infty}
 o_{\eta,k}(1)=0.
\end{equation}
Since $\mathfrak z_{\eta,k}^\nu\geq c>0$ on
$\mathcal O_{\eta,k}^\nu$, we conclude that
\begin{equation}
\label{EQ: appendix vanishing transition quotient}
 \lim_{\eta\to0}\limsup_{k\to\infty}
 \left\|
  \frac{|\nabla u_k|^2}{\mathfrak z_{\eta,k}^\nu}
 \right\|_{L^\infty(\mathcal O_{\eta,k}^\nu)}
 =
 0.
\end{equation}
This vanishing is asserted only on the concentration-free interpolation
region.  On the remainder of $\mathtt G_k^\nu$, the original pointwise
collar estimate gives
\begin{equation}
\label{EQ: appendix geodesic component quotient}
 \limsup_{\eta\to0}\limsup_{k\to\infty}
 \left\|
  \frac{|\nabla u_k|^2}{\mathfrak z_{\eta,k}^\nu}
 \right\|_{L^\infty(\mathtt G_k^\nu)}
 <
 +\infty.
\end{equation}

\paragraph{The global weight.}
We now combine the weights constructed on the different parts of
$\Sigma_k$. Recall that the degenerating collar is decomposed as
\begin{equation}
\label{EQ: appendix collar decomposition recalled}
 \mathtt P_k^\varrho
 =
 \mathtt G_k^0
 \cup
 \mathtt B_k
 \cup
 \mathtt G_k^1,
\end{equation}
where $\mathtt B_k$ contains the concentration region and the unique
bubble, while $\mathtt G_k^0$ and $\mathtt G_k^1$ are the two
concentration-free components connecting the bubble region to the two
ends of the collar.

On $\mathtt G_k^\nu$, $\nu\in\{0,1\}$, we use the modified collar
weight $\mathfrak z_{\eta,k}^\nu$. On $\mathtt B_k$, we retain the
complete bubble weight $\omega_{\eta,k}^{\mathrm{bub}}$, including its
definition inside the concentration ball. Finally, away from the
degenerating collar, we use the usual base-map weight
$\omega_{\eta,k}^{\mathrm{base}}$. We therefore define
\begin{equation}
\label{EQ: appendix global weight definition}
 \mathfrak z_{\eta,k}
 :=
 \begin{cases}
  \mathfrak z_{\eta,k}^0,
  &\text{in }\mathtt G_k^0,\\[1mm]
  \omega_{\eta,k}^{\mathrm{bub}},
  &\text{in }\mathtt B_k,\\[1mm]
  \mathfrak z_{\eta,k}^1,
  &\text{in }\mathtt G_k^1,\\[1mm]
  \omega_{\eta,k}^{\mathrm{base}},
  &\text{in }\Sigma_k\setminus\mathtt P_k^\varrho.
 \end{cases}
\end{equation}

At the interface between $\mathtt G_k^\nu$ and $\mathtt B_k$, the
modified collar weight is constant and equal to the corresponding
matching value $\alpha_k^\nu$. By
\eqref{EQ: appendix exact collar bubble matching}, this value agrees
with the cylindrical trace of the bubble weight:
\begin{equation}
\label{EQ: appendix global interface matching}
 \left.
 \mathfrak z_{\eta,k}^\nu
 \right|_{\partial\mathtt G_k^\nu\cap\partial\mathtt B_k}
 =
 \alpha_k^\nu
 =
 \left.
 \omega_{\eta,k}^{\mathrm{bub}}
 \right|_{\partial\mathtt G_k^\nu\cap\partial\mathtt B_k}.
\end{equation}
Thus the piecewise definition
\eqref{EQ: appendix global weight definition} is continuous across the
bubble--collar interfaces. Outside $\mathtt P_k^\varrho$, the weight is completed by the
thick-part construction of \cite{DGR25}. Since the present argument
concerns only the bubble--collar interfaces, we do not repeat that
construction here.

   .

On each
concentration-free component $\mathtt G_k^\nu$, the estimates proved
above give
\begin{equation}
\label{EQ: appendix quotient bound geodesic pieces}
 \limsup_{\eta\to0}\limsup_{k\to\infty}
 \left\|
  \frac{|\nabla u_k|^2}
       {\mathfrak z_{\eta,k}^\nu}
 \right\|_{L^\infty(\mathtt G_k^\nu)}
 <
 +\infty.
\end{equation}
On $\mathtt B_k$, the corresponding estimate follows from the bubble
weight construction of \cite{DGR25}:
\begin{equation}
\label{EQ: appendix quotient bound bubble piece}
 \limsup_{\eta\to0}\limsup_{k\to\infty}
 \left\|
  \frac{|\nabla u_k|^2}
       {\omega_{\eta,k}^{\mathrm{bub}}}
 \right\|_{L^\infty(\mathtt B_k)}
 <
 +\infty.
\end{equation}
Combining
\eqref{EQ: appendix quotient bound geodesic pieces} and
\eqref{EQ: appendix quotient bound bubble piece}, we obtain
\begin{equation}
\label{EQ: appendix global quotient bound}
 \limsup_{\eta\to0}\limsup_{k\to\infty}
 \left\|
  \frac{|\nabla u_k|^2}
       {\mathfrak z_{\eta,k}}
 \right\|_{L^\infty(\mathtt P_k^\varrho)}
 <
 +\infty.
\end{equation}

It remains to record the positivity property on the
concentration-free components. By
\eqref{EQ: appendix local  weight comparison}, there exists a
constant $C>0$, independent of $\eta$ and $k$, such that
\begin{equation}
\label{EQ: appendix global positivity weight comparison}
 \mathfrak z_{\eta,k}^\nu
 \leq
 C\mathfrak w_{\sigma,\eta,k}^\nu
 \qquad\text{in }\mathtt G_k^\nu.
\end{equation}

   The compactness argument of
Sections~\ref{SECTION: Stability of the Morse Index} and
\ref{SECTION: Proof of Main Theorem} can therefore be applied on
$\mathtt G_k^0$ and $\mathtt G_k^1$, whereas the contribution of the
unique bubble contained in $\mathtt B_k$ is treated by the bubble-tree
analysis of \cite{DGR25}. This establishes the required compatibility
of the weights in the case of one degenerating collar and one bubble,
namely $M=1$ and $M_1=1$.

\end{document}